\documentclass[a4paper]{article}
\usepackage[normalem]{ulem}
\usepackage{titlesec}
\usepackage{comment}
\usepackage[english]{babel}
\usepackage[T1]{fontenc}
\usepackage{amssymb}
\usepackage{amsthm}
\usepackage[a4paper,top=2.5cm,bottom=2.5cm,left=2cm,right=2cm,marginparwidth=1.75cm]{geometry}
\usepackage{enumerate}
\usepackage{wrapfig}
\usepackage{enumerate}
\usepackage{enumitem}

\usepackage[dvipsnames]{xcolor}
\usepackage{amsmath,amsfonts,amssymb,amsthm}
\usepackage{mathtools}
\usepackage{bm}
\usepackage{commath}
\usepackage[export]{adjustbox}
\usepackage{subfig}
\usepackage{amsmath}
\usepackage{graphicx}
\usepackage[colorlinks=true, allcolors=blue,hypertexnames=false]{hyperref}
\usepackage{float}
\usepackage{tabularx}   % auto-wrapping X column
\usepackage{booktabs}   % nicer rules

\newtheorem{theorem}{Theorem}[section]
\newtheorem{definition}{Definition}[section]
\newtheorem{proposition}{Proposition}
\newtheorem*{proposition*}{Proposition}
\newtheorem{lemma}{Lemma}
\newtheorem{remark}{Remark}
\theoremstyle{definition}
\newtheorem{example}[theorem]{Example}

\title{From multitype branching to multitype coalescents and back: Genealogies, duality and examples}
\author{Adri\'an Gonz\'alez Casanova\and Noemi Kurt \and Imanol Nuñez \and Jos\'e Luis P\'erez}

\begin{document}
\maketitle

\begin{abstract}
    We establish new connections between multitype $\Lambda$-coalescents and multitype continuous state branching processes (CSBPs) through duality with a $\Lambda-$Wright-Fisher frequency process, and a homeomorphism between the spaces of the two classes of processes. The approach relies on a sequential sampling procedure for the frequency process of multitype CSBPs and provides an interpretation of multitype $\Lambda-$coalescents as the local genealogies of multitype CSBPs which gives further insight into different approaches to generalise $\Lambda-$coalescents to the multitype setup. As an application of the homeomorphism, we introduce and discuss multitype Beta-coalescents as the local genealogy of multitype $\alpha$-stable branching processes.
\end{abstract}

\section{Introduction}
% \textcolor{blue}{Start with being more upfront about what we are doing before discussing the literature. Mention main results and multitype Lambda controversy and ongoing research. Some of the discussion may be moved to the interpretations section later on, see below.}

\label{sec:intro}

\

    In this paper, we show that the multitype $\Lambda$-coalescents introduced by Johnston, Kyprianou and Rogers\ in \cite{johnstonMultitypeLcoalescents2022a} 
    describe the local genealogies of multitype continuous-state branching processes. To establish this connection, we extend the approach of Caballero et al.\ \cite{caballeroRelativeFrequencyTwo2023} to the multitype setting. We obtain a natural homeomorphism between the spaces of multitype continuous-state branching processes and multitype $\Lambda$-coalescents that has a genealogical interpretation.

    The key mechanism is an approximation by a sequential sampling procedure, also known in the literature as \emph{culling}. Through this procedure, we construct an autonomous frequency process associated with the branching process parameters.
    The limit of this frequency process belongs to a broader class that we call \emph{multitype $\Lambda$-Wright--Fisher processes}. They are characterised by a $d-$dimensional system of stochastic differential equations for which we prove the existence of a unique strong solution.
    %This result is nontrivial, as there is no general multidimensional theory directly applicable to SDEs of this type.

    Once existence and uniqueness are established, we show that the multitype $\Lambda$-Wright--Fisher process is in moment duality with the block-counting process of a multitype $\Lambda$-coalescent. Moreover, starting from a multitype continuous-state branching process and applying the sequential sampling procedure, we obtain that the parameters of the associated multitype $\Lambda$-Wright--Fisher process and the dual multitype $\Lambda$-coalescent are given explicitly through the homeomorphism, thus providing the aforementioned genealogical interpretation.
   As an application of this homeomorphism, we provide a natural definition of multitype Beta coalescents extending the well-known fact that  in the one-dimensional setting, Beta coalescents arise as the local genealogies of $\alpha$-stable branching processes.
   %The homeomorphism allows us to distinguish between component-wise Beta coalescent and genuinely multitype Beta-coalescents, since it provides a criterion for the possible occurrence of large mergers involving several types depending on the parameter of the Beta distribution.

\medskip

We now discuss some of the large literature concerning the genealogy of (single type) continuous state branching processes (CSBPs) and their relation to other processes, in particular $\Lambda-$coalescents. Le Gall and Le Jan \cite{legallBranchingProcessesLevy1998}, building on previous work where a discrete height process encodes the genealogy of a Galton--Watson process (see, for instance, \cite{legallRandomTreesApplications2005}), proposed a method to encode the genealogy of a CSBP through a real-valued random process also known as the height process. This approach was further developed by Duquesne and Le Gall \cite{duquesneRandomTreesLevy2005}, who reformulated and solved problems about the genealogy of CSBPs in terms of spectrally positive L\'evy processes. Later, Bertoin and Le Gall \cite{bertoinBolthausenSznitmanCoalescent2000} introduced a new perspective, suggesting that the genealogy of a CSBP could be studied using flows of subordinators. Specifically, they showed that the genealogy of Neveu's CSBP is related to the Bolthausen--Sznitman coalescent. In a series of works \cite{bertoinStochasticFlowsAssociated2003,bertoinStochasticFlowsAssociated2005,bertoinStochasticFlowsAssociated2006}, Bertoin and Le Gall defined the $\Lambda$-Fleming--Viot process, proving its moment duality with the $\Lambda$-coalescent introduced by Pitman \cite{Pitman1999} and Sagitov \cite{Sagitov1999}.

\medskip

 Pitman characterized the $\Lambda$-coalescents as the class of exchangeable and consistent coalescents where distinct merger events cannot occur simultaneously. Using the lookdown construction by Donnelly and Kurtz \cite{donnellyCountableRepresentationFlemingViot1996,donnellyParticleRepresentationsMeasureValued1999}, Birkner, Blath, Capaldo, Etheridge, M\"ohle, Schweinsberg and Wakolbinger \cite{birknerAlphaStableBranchingBetaCoalescents2005} showed that the genealogy of a CSBP is a Markov process time changed by the total mass process if and only if the CSBP is $\alpha$-stable, in which case the time-changed genealogy corresponds to a Beta-coalescent. This result generalises the findings of Bertoin and Le Gall \cite{bertoinBolthausenSznitmanCoalescent2000} which correspond to  $\alpha = 1$. Notably, if $X$ and $Y$ are two independent realizations of a CSBP, the frequency process $R = X / (X + Y)$ can be time-changed to be the dual of the block-counting process of a $\Lambda$-coalescent if and only if the branching mechanism is $\alpha$-stable.

 \medskip

 In \cite{bertoinStochasticFlowsAssociated2006}, Bertoin and Le Gall established that, locally, the genealogy of a $\Lambda$-Fleming--Viot process behaves similarly to the genealogy of a CSBP. This connection allowed them to link the notion of ``coming down from infinity'' for $\Lambda$-coalescents to the extinction of the corresponding CSBP. However, their proof was analytical, and they conjectured the existence of a deeper, probabilistic link between $\Lambda$-coalescents and CSBPs. Building on earlier ideas \cite{donnellyParticleRepresentationsMeasureValued1999,birknerAlphaStableBranchingBetaCoalescents2005}, Berestycki, Berestycki and Limic \cite{berestyckiSmalltimeCouplingLambdacoalescents2014} constructed a ``small-time'' coupling between a $\Lambda$-coalescent and an associated CSBP, addressing the conjecture posed by Bertoin and Le Gall in \cite{bertoinStochasticFlowsAssociated2006} and further strengthening the connection between CSBPs and $\Lambda$-coalescents in the one-dimensional setting.

 \medskip

Caballero, Gonz\'alez Casanova and P\'erez \cite{caballeroRelativeFrequencyTwo2023}, motivated by studying genetic dynamics in populations with two types of individuals with potentially distinct reproduction mechanisms, introduced the sequential sampling procedure, referred to as ``culling'' in their work and inspired by Gillespie's framework \cite{Gillespie74}. They showed that, if $X$ and $Y$ are two independent CSBPs, the sequential sampling procedure discretizes the frequency process $R$ to derive an autonomous frequency process under constant total population size. When $X$ and $Y$ share the same reproduction mechanism, this autonomous frequency process is dual to the block-counting process of a $\Lambda$-coalescent, with $\Lambda$ explicitly determined by the branching mechanism's parameters. Caballero et al. constructed a homeomorphism between the space of CSBPs and of $\Lambda$-coalescents. If we call this homeomorphism Gillespie's transform, an alternative and equivalent definition of the $\Lambda$-coalescents is the image of the CSBP under Gillespie's transform.

\medskip

The two classes of processes, CSBPs and $\Lambda$-coalescents, both generalise to a multitype set up. Multitype branching processes are a well-established area of research, and there are many examples of multitype aspects in coalescent processes. For instance, structured coalescents, where blocks have different types, and interact depending on their types are a prominent class. The two-island model is the simplest example of this category. More recent examples include the seed bank coalescent with simultaneous switching~\cite{Blathetal2020}, and the nested coalescent~\cite{blancasTreesTreesSimple2018}. Hence, one may ask for the most natural generalisation of the notion of $\Lambda-$coalescents to the multitype set up, capturing the essential properties. Extending Pitman's characterization to the multitype case yields a reasonable definition: A multitype \(\Lambda\)-coalescent in the Pitman sense is a multitype exchangeable and consistent coalescent where simultaneous merger events are forbidden. It is not hard to see that the examples mentioned above are multitype \(\Lambda\)-coalescents in the Pitman sense.

\medskip

A slightly more restrictive definition of a multitype $\Lambda-$coalescent was recently given by Johnston, Kyprianou and Rogers \cite{johnstonMultitypeLcoalescents2022a}, who impose that, in addition to the exchangeability and consistency, any simultaneous \emph{transitions} of blocks are forbidden. They show that this set of conditions is equivalent to a characterisation in terms of finite $\Lambda-$measures whose masses at 0 correspond to the switching rates of single blocks and the pairwise mergers, generalising the corresponding result for $\Lambda-$coalescents of one type \cite{Pitman1999}. However, the model of Johnston, Kyprianou and Rogers excludes simultaneous type switches, and hence the seed bank coalescent with multiple switches \cite{Blathetal2020} and the nested coalescent~\cite{blancasTreesTreesSimple2018} don't fall into their class of multitype $\Lambda-$coalescents. Thus, the question of the most natural generalisation -- does one allow simultaneous type switches or not -- so far remains open.

\medskip

Concerning the connection between multitype CSBPs and multitype coalescents, Johnston and Lambert \cite{johnstonCoalescentStructureUniform2021} recently used a Poissonization technique to investigate the genealogy of multitype CSBPs. While they showed that the genealogy of a CSBP is not generally Markovian, they proved that it can locally be represented as a $\Lambda$-coalescent. For multitype CSBPs, the local genealogy corresponds to the multitype $\Lambda$-coalescent of Johnston, Kyprianou and Rogers \cite{johnstonMultitypeLcoalescents2022a}. 

In the present work, we also introduce the Gillespie transform to the multitype case for a characterization of those multitype $\Lambda$-coalescents that can be constructed from multitype CSBPs as their genealogy. We show that also in this case we are in the class of multitype $\Lambda-$coalescents of  \cite{johnstonMultitypeLcoalescents2022a}, hence the Johnston, Kyprianou and Rogers definition of a multitype $\Lambda$-coalescent is a natural extension of the notion of $\Lambda$-coalescents also in view of the genalogical connection with multitype CSBPs. 
Furthermore, we provide another link to multitype population models and multitype colaescents via duality. Indeed, the limit of the sequentially sampled process of the relative frequency associated to a multitype CSBP coincides with the multitype $\Lambda-$Wright-Fisher frequency process, which in turn is the moment dual of the block counting process of a multitype $\Lambda$-coalescent. This partially extends recent results by M\"ohle \cite{mohleMultitypeCanningsModels2023}.
%by providing a forward process that is dual to the backward-in-time genealogical process limit 

%Following Caballero et al. \cite{caballeroRelativeFrequencyTwo2023}, we construct an autonomous multitype frequency process via sequential sampling (Theorem \ref{theo:culling_limit}). We prove that this autonomous frequency process is in moment duality with the block-counting process of a multitype $\Lambda$-coalescent (Theorem \ref{theo:block_duality}). This result allows us to establish a homeomorphism between the parameter spaces of multitype CSBPs and multitype $\Lambda$-coalescents (Theorem \ref{theo:homeomorphismSpaces}).

Hence, our work extends some of the connections between CSBPs and $\Lambda-$coalescents to the multitype set up, and thus also shows that the multitype characterisation of Johnston, Kyprianou and Rogers is in this respect the natural generalisation of the notion of multiple merger coalescents as it is equivalent to defining multitype $\Lambda-$coalescents as the image of mulititype CSBPs via the multitype Gillespie-transform.

%While the Gillespie-$\Lambda$-coalescent agrees with the Pitman-$\Lambda$-coalescent in one dimension, our results demonstrate that this is not the case in the multitype setting {\color{red} Do we actually showed this? I guess this is stated in the previous paragraph where we say that it does not include the coalescent with multiple swithces, etc. because they impose the additional condition that no simultaneous switches are permitted.}. Furthermore, we show that the Gillespie-multitype-$\Lambda$-coalescents correspond to the class characterized by Johnston et al. This leads to two natural open problems for future research, namely the characterisation of the multitype $\Lambda$-coalescent in the Pitman sense, and determining the superset of multitype CSBPs that is mapped to the multitype $\Lambda$-coalescent in the Pitman sense via Gillespie's homeomorphism.

\section{Multitype processes}
%\textcolor{blue}{I changed the section title, because we don't really have one model we investigate. Check if you agree or have a better suggestion. I also rewrote the next paragraph, please check}
In this section, we state and discuss the main mathematical objects of this paper, namely multitype $\Lambda$-coalescents, multitype (Wright-Fisher-)frequency processes and multitype continuous state branching processes (MCSBPs). We provide their definitions and discuss some examples. Our main results connect these classes of multitype processes: Firstly, we obtain a moment duality between the block-counting process of a multitype $\Lambda-$coalescent and its corresponding frequency process, and secondly we prove a homeomorphism between MCSBPs and multitype $\Lambda-$coalescents. The latter allows to define in a canonical way multitype coalescent processes, extending to processes that have not yet been explicitly studied in the literature. In particular, as an application of the homeomorphism we introduce a class of multitype Beta-coalescents that in some sense is ``truly multitype''. Finally, we will link multitype frequency processes and MCSBPs via a  sequential sampling procedure (sometimes also called culling method in the literature), deriving a Markov process that can be interpreted as the frequency process of a MCSBP population model with fixed total mass. Therefore, as in the case of one type, multitype $\Lambda-$coalescents are interpreted as the local genealogy of MCSBPs.
%For this ``sequentially sampled frequency process'' we then derive a moment duality with the block-counting process of the corresponding multitype $\Lambda-$coalescent, which provides new insight into the relation between MCSBPs and multitype multiple merger coalescents, and a genealogical interpretation of the homeomorphism.
%Then we go on to study the frequency process of our CSBPs via a \emph{sequential sampling procedure} (sometimes also called culling method in the literature).
%We derive a duality with the block counting process of a multitype $\Lambda$-coalescent, that serves as a stepping stone in deriving our main result, the homeomorphism between the spaces of multitype CSBPs and multitype $\Lambda$-coalescents.
\medskip

\subsection{Multitype \texorpdfstring{$\mathbf{\Lambda}$}{L}-coalescents}
\label{subsec:lambdaCoal}
$\Lambda$-coalescents or coalescents with multiple mergers (or collisions) have famously been introduced by Pitman \cite{Pitman1999}, Sagitov \cite{Sagitov1999}, and Donnelly and Kurtz \cite{donnellyParticleRepresentationsMeasureValued1999}.  They are characterised as being Markov processes taking values in the partitions of $\mathbb{N},$ with transitions consisting of mergers of finitely many blocks of a partition, satisfying three conditions:
\begin{itemize}
    \item \emph{Exchangeability}, that is invariance of the law of the process under permutation of the blocks of a partition,
    \item \emph{Consistency}, meaning that the restriction of the process to partitions of $[n]$ is again Markov for any $n\in\mathbb{N},$
    \item No \emph{simultaneous} occurrence of mergers (\emph{asynchronity}).
\end{itemize}

Here and throughout the paper we write $[n] = \{1, \ldots, n\}$ and $[n]_0 = \{0\} \cup [n]$ for $n \in \mathbb{N}$. The well-known result of Pitman states that such processes are uniquely characterised by a finite measure $\Lambda$ on $[0,1]$, from which the rates of merging a $k$-tuple out of $b$ distinct blocks to one new block are calculated as
\begin{equation}\label{eq:Pitman}
    \lambda_{b,k}=\int_{[0, 1]} x^{k-2}(1-x)^{b-k} \Lambda(dx).
\end{equation}
Equivalently, the previous expression can also be written as
\begin{equation} \label{eq:Lambdacoal}
    \lambda_{b, k} = \Lambda(\{0\}) 1_{\{k = 2\}} + \int_{(0, 1]} x^k  (1 - x)^{b - k} \frac{\Lambda(dx)}{x^2} ,
\end{equation}
where it becomes explicit that the mass of $\Lambda$ at $0$ only contributes to pairwise mergers.

\medskip

Recently, Johnston, Kyprianou and Rogers \cite{johnstonMultitypeLcoalescents2022a} have generalised this concept of multiple mergers to multitype coalescents.
They consider processes taking values in \emph{$d$-type partitions} of $\mathbb{N}$, meaning that every block in the partition has one out of $d$ different \emph{types} (or colours). Formally, a $d-$type partition $\pi=(\pi_1,...,\pi_d)$ consists of disjoint collections $\pi_i$ of subsets (blocks) of $\mathbb{N}$ such that $\bigcup_{i\in [d]}\pi_i$ is a partition of $\mathbb{N}.$
Blocks may change type, and finitely many blocks (of same or different types) may merge to form a single new block (of one type).
We may equivalently think of the multitype setting as providing a spatial subdivision of the population, with types corresponding to islands or colonies, and type changes corresponding to migration between colonies.
In \cite{johnstonMultitypeLcoalescents2022a} it is proved that the rates of $d$-type coalescent processes which are exchangeable and consistent, and where \emph{no multiple transitions of blocks occur simultaneously}, are exactly the multitype $\Lambda$-coalescents defined by a generalisation of the condition \eqref{eq:Pitman}.

\medskip

As in \cite{johnstonMultitypeLcoalescents2022a} we thus define the multitype $\Lambda-$ coalescents as follows. In this multitype setting, we will write $[n] = [n_1] \times \cdots \times [n_d]$ for $n \in \mathbb{N}^d$, and $[n]_0$ analogously. We write $\delta_{ij}$ for the Kronecker symbol as usually.

\begin{definition}[Multitype $\Lambda-$coalescent]\label{def:multitype-lambda}
For  $i, j \in [d]$ let $\rho_{ij} \in \mathbb{R}_+$, and for $i \in [d]$, let $\mathcal{Q}_i$ be a measure on the unit cube $[0, 1]^d$, without mass at the point $\{0\}$, satisfying the integrability condition
\begin{equation} \label{eq:intcond_lambda}
    \int_{[0, 1]^d \setminus \{0\}} \sum_{j = 1}^d u_j^{1 + \delta_{ij}} \mathcal{Q}_i(du) < \infty .
\end{equation}
A \emph{multitype $\Lambda-$coalescent} with defining pair 
\begin{equation} \label{eq:charact_1_mtlc}
    \bigl( (\rho_{ij})_{i,j \in [d]}, (\mathcal{Q}_1, \ldots, \mathcal{Q}_d) \bigr) 
\end{equation} 
is a strong Markov process $\Pi=(\Pi(t))_{t\geq 0}$ taking values in the $d-$type partitions of $\mathbb{N},$ such that the only permissible transitions are (multiple) merger of blocks, such that given that there are currently $b=(b_1,...,b_d) \in \mathbb{N}_0^d$ blocks of respective type $j\in[d]$, the rate at which $k \in [b]_0$ blocks are involved in a merger where the resulting block is of type $i$ is
\begin{equation} \label{eq:lambda_i_multitype_coal}
    \lambda_{b, k}^i
    = \rho_{ii} 1_{\{k = 2 e_i\}}
    + \sum_{j \in [d] \setminus \{i\}} \rho_{ij} 1_{\{k = e_j\}}
    + \int_{[0, 1]^d \setminus \{0\}}  u^{k} (1 - u)^{b - k} \mathcal{Q}_i(du),
\end{equation}
where $u^{k} (1 - u)^{b - k} =\prod_{j=1}^d u_j^{k_j} (1 - u_j)^{b_j - k_j}$ for $u\in[0,1]^d$, and for each $i \in [d]$, $e_i$ denotes the $i$-th canonical vector.
\end{definition}
%More precisely, according to \cite{johnstonMultitypeLcoalescents2022a}, a \emph{multitype $\Lambda$-coalescent}, which we will denote simply by $\Pi$, is characterized  by the pair

Observe the corresponding roles of the parameters $\rho_{ij}$ in \eqref{eq:lambda_i_multitype_coal} and $\Lambda{(0)}$ in \eqref{eq:Lambdacoal}. For $i, j \in [d]$, $\rho_{ij}$ is then the rate at which a block of type $j$ becomes a block of type $i$ when $i \neq j$, corresponding to a \emph{type switch} between types $i$ and $j,$ or in a different interpretation, \emph{migration} between colony $i$ and $j$, while $\rho_{ii}$ is the rate at which two blocks of type $i$ merge into a block of type $i$. Moreover, note that $\mathcal{Q}_i(du)$ corresponds to the (possibly infinite) measure $x^{-2}\Lambda(dx)$ rather than $\Lambda(dx)$%{\color{red} I do not understand this last phrase what do you mean with " $\mathcal{Q}_i(du)$ corresponds to the (possibly infinite) measure $x^{-2}\Lambda(dx)$ rather than $\lambda$"?} \textcolor{blue}{I made a small change to formula (2), is it clearer now? The point is htat $\Lambda$ is a finite measure, but the $Q_i$ are not, so $Q_i$ is really $\Lambda_i(dx)x^{-2}$}
, and that the integrability condition \eqref{eq:intcond_lambda} distinguishes between the cases $i\neq j$ and $i=j$.

%\medskip

\begin{example}[Multitype Kingman coalescents] \label{ex:mkingman}
    A multitype $\Lambda$-coalescent is a multitype Kingman coalescent if $\mathcal{Q}_i \equiv 0$ for all $i \in [d]$. 
    In this case, blocks evolve only through binary mergers and type changes. 
    More precisely, two blocks of type $j$ merge into a single block of type $j$ at rate $\rho_{jj}$, while a block of type $j$ switches to type $i$ at rate $\rho_{ij}$. This multitype coalescent appears as the dual process to the multi-allelic stepping stone model studied by Shiga in \cite{Shiga1982} in the neutral case. 
It also has been studied by Notohara in \cite{notohara90} as a natural generalization of Kingman's coalescent to geographically structured populations.
\end{example}

% \textcolor{blue}{NK: I'm sure multitype Kingman has been looked at before JKR, is there a standard reference?}

\begin{example}[The seed bank coalescent] \label{ex:seedbank}
    Consider the case $d = 2$, with $\mathcal{Q}_1 = \mathcal{Q}_2 \equiv 0$, $\rho_{11}, \rho_{12}, \rho_{21} > 0$, and $\rho_{22} = 0$. 
    This may be interpreted as follows: Type $1$ represents an active population, in which binary mergers occur at rate $\rho_{11}$, while type $2$ represents a dormant population, in which no mergers take place. 
    Blocks may switch between active and dormant states at rates $\rho_{12}$ and $\rho_{21}$. 
    This corresponds to the seed bank coalescent introduced by Blath et al. in \cite{Blathetal2016}.
\end{example}

In view of this second example we would like to highlight an issue arising when generalising the notion of $\Lambda$-coalescents to the multitiype setting, namely the precise notion of asynchronity. 
While in the definition of \cite{johnstonMultitypeLcoalescents2022a} asynchronity is required for \emph{any} type of transition, that is, both for type switches and mergers, it is possible to define multitype coalescent processes that are exchangeable and consistent with no simultaneous \emph{mergers} of two or more blocks, but possibly \emph{simultaneous type changes}, that may equally well be considered multitype generalisations of $\Lambda$-coalescents.
A simple example for the latter definition is given by the so-called seedbank ($\Lambda$-)coalescent with simultaneous switches of \cite{Blathetal2020}. This model has no simultaneous multiple mergers, but simultaneous type switches. Hence it doesn't fall into the definition of multitype $\Lambda$-coalescents of \cite{johnstonMultitypeLcoalescents2022a} that we also use in this paper, although one might argue that they should be included in this definition.
This shows that the question of the most natural generalisation of $\Lambda$-coalescents to the multitype setting should be treated with some care.
The main results of our paper show that actually the definition of \cite{johnstonMultitypeLcoalescents2022a} is very natural in a structural sense, as we will see more clearly below.

We conclude this subsection by another example which is a simple extension of the Beta coalescent to the multitype setup. Later on, using our main results connecting multitype $\Lambda-$coalescents and multitype CSBPs, we will be able to define ``truly multitype'' Beta coalescents.

\begin{example}[Multitype entry-wise Beta coalescents] \label{ex:betas}
    For each $i \in [d]$, let $\rho_{ii} = 0$, and fix parameters $\alpha_{ij} \in (0, 1 + \delta_{ij})$ for $j \in [d]$. Then the measures $Q_i, i\in [d]$ on $[0,1]^d$ defined as
    \begin{equation} \label{eq:betaProd}
        \mathcal{Q}_i(ds) = \sum_{j \in [d]} \mathrm{c}_{ij} s_j^{1 - \alpha_{ij}} (1 - s_j)^{\alpha_{ij} - 1} \frac{ds_j}{s_j^2} \prod_{k \in [d] \setminus \{j\}} \delta_0(ds_k).%\delta_{s_j e_j}(ds) \otimes\bigotimes_{k\in [d], k\neq j}\delta_{0}(ds_k) 
    \end{equation}
    with some nonnegative constants $\mathrm{c}_{ij}$ for $i,j\in[d]$, satisfy condition \eqref{eq:intcond_lambda}. The resulting multitype $\Lambda-$coalescent evolves as follows:
    % Whenever a merger event produces a block of type $i$, a type $j \in J_i$ is selected{\color{red} with what probability?}, and only blocks of type $j$ participate in the merger. 
    %\textcolor{cyan}{
    For each $i,j\in [d]$ such that $c_{ij}>0$, at ``rate'' $c_{ij} s^{-2} ds$, the blocks of type $j$ participate in a merger that produces one block of type $i$.
    % Conditional on this choice, e
    Each block of type $j$ participates, independently of the others, with probability $v$, where $v$ is sampled from a $\mathrm{Beta}(2 - \alpha_{ij}, \alpha_{ij})$ distribution.
    %Namely, 
    % \begin{equation} \label{eq:betaProd}
    %     \mathcal{Q}_i(ds) = \sum_{j \in J_i} \mathrm{const}_{ij} \, \delta_{0}(ds_1) \otimes \cdots \otimes \delta_0(ds_{j-1}) \otimes s_j^{1 - \alpha_{ij}} (1 - s_j)^{\alpha_{ij} - 1} \frac{ds_j}{s_j^2} \otimes \delta_0(ds_{j + 1}) \otimes \cdots \otimes \delta_0(ds_d).
    % \end{equation}
    % \textcolor{cyan}{Here, and throughout, the symbol $\otimes$ denotes the product of measures. More precisely, if $(\nu_j){j \in J}$ is a family of measures, then $\bigotimes_{j \in J} \nu_j$ stands for the corresponding product measure on the product space.} 
    % In particular, different types contribute independently to the merger, and the resulting dynamics are governed by simultaneous multiple mergers with type-dependent frequencies.
    In particular, there might be blocks of the same type merging into a block of another type, but blocks of different types may not merge in this example.
    %{\color{red} what do you mean by "contribute independently to the merger? I thought only one type contributes? and also what do you mean by type-dependent frequencies in here?}\textcolor{blue}{yes, and what's the use of the $J_i$ if only one type contributes anyway - this can be put into the $c_{ij}$ if we want to weight it somehow. I'm in favor of lighten notation, even if it comes at the cost of not being the most general case, which here it isn't anyway. Plus, should we abandon the $\otimes$? if we just write it as a normal product there shouldn't be any misunderstanding.}
\end{example}

\begin{example}[Multitype Bolthausen--Sznitman coalescent]
    Let $a \in (0, \infty)^d$ be given. 
    We say that a multitype coalescent is a multitype Bolthausen--Sznitman coalescent with parameter $a$ if for each $i \in [d]$, $\rho_{ii} = 0$ and 
    \[
        \mathcal{Q}_i(ds) = a_i \frac{ds_i}{s_i^2} %\otimes \bigotimes_{k \in [d] : k \neq i} 
        \prod_{k \in [d] \setminus \{i\}}\delta_0(d s_k).
    \]
    This is a particular case of Example \ref{ex:betas}, in which $c_{ij}=0$ for $j\neq i$ and $\alpha_{ii} = 1$ for all $i \in [d]$. 
    It can also be obtained from a different extension of Beta coalescents to the multitype case, cf. Definition \ref{def:multiBeta}.
\end{example}

\subsection{Multitype frequency process and duality}
\label{sec:lambdaWFfreq}

For single-type coalescents, and for some multitype models such as the multitype
Kingman coalescent, Wright--Fisher-type population models whose frequency
processes are dual to the block-counting processes of the corresponding
coalescents are well studied. In the present setting, we introduce a general
class of \emph{multitype $\Lambda$-Wright--Fisher frequency processes} which are
related, via moment duality, to multitype $\Lambda$-coalescents.

These processes generalize the classical Wright--Fisher diffusion by incorporating
both mutation mechanisms between types and jump events corresponding to
simultaneous multiple mergers in the genealogy. A process
$\mathsf{R} = (\mathsf{R}(t))_{t \geq 0}$ in this class can be represented through
the system of stochastic differential equations
\begin{equation} \label{eq:lambdaWF}
    \begin{split}
        d\mathsf{R}_i(t) 
        & = \sum_{j \in [d] \setminus \{i\}} \Bigl( \rho_{ji} - \int_{[0, 1]^d} u_i \mathcal{Q}_j(du) \Bigr) (\mathsf{R}_j(t) - \mathsf{R}_i(t)) dt \\
        & \quad + \sqrt{ \rho_{ii} \mathsf{R}_i(t) (1 - \mathsf{R}_i(t)) } dB_i(t) \\
        & \quad + \sum_{j = 1}^d \int_{[0, 1]^d} \int_0^1 u_i (1_{\{v \leq \mathsf{R}_j(t-)\}} - \mathsf{R}_i(t-)) \widetilde{\mathcal{N}}_j(dt, du, dv) \,.
    \end{split}
\end{equation}
In the previous display, $(\rho,\mathcal{Q})$ %denotes the defining pair of a
%multitype $\Lambda$-coalescent, 
is such that $\rho \in \mathbb{R}_+^{d \times d}$ and $\mathcal{Q} = (\mathcal{Q}_1, \ldots, \mathcal{Q}_d)$ is a vector of measures on $[0, 1]^d$ satisfying \eqref{eq:intcond_lambda},
$B$ is a $d$-dimensional Brownian motion, and,
for each $j \in [d]$, $\mathcal{N}_j(dt,du,dv)$ is a Poisson random measure on
$\mathbb{R}_+ \times [0,1]^d \times [0,1]$ with intensity
$dt\,\mathcal{Q}_j(du)\,dv$. We denote by
$\widetilde{\mathcal{N}}_j(dt,du,dv)$ the associated compensated random measure.
The well-posedness of \eqref{eq:lambdaWF} is one of the main results of this
section. 
%{\color{red} Below we prove the well-posedness of \eqref{eq:lambdaWF} for any defining pair of a multitype $\Lambda$-coalescent, but we used it for the proof of the culling and it is not given in terms of coalescents but of the parameters of a MCSBP, unless we use the homeomorphism which is proven by the homemomorphism which becomes a bit circular. So maybe change the statement of the next theorem by something like  For any $(\rho,\mathcal{Q})$ which satisfies...$\rho\in\R^{d\times d}_+$ and such that $Q$ satifies \eqref{eq:intcond_lambda}?}
\begin{theorem} \label{theo:SDEProperties}
    For each %fixed defining pair $(\rho,\mathcal{Q})$ of a multitype
    %$\Lambda$-coalescent, 
    $\rho \in \mathbb{R}_+^{d\times d}$ and each vector of measures $\mathcal{Q} = (\mathcal{Q}_1, \ldots, \mathcal{Q}_d)$ on $[0, 1]^d$ satisfying \eqref{eq:intcond_lambda},
    there exists a unique strong solution to
    \eqref{eq:lambdaWF}. This solution is Feller. Moreover, its infinitesimal
    generator is given by
    \begin{equation} \label{eq:lambdaWFGen}
        \begin{split}
            \mathcal{A}_{\mathrm{WF}} f(r)
            &=
            \sum_{i = 1}^d \sum_{j \in [d] \setminus \{i\}}
            \rho_{ji} (r_j-r_i) \partial_i f(r)
            + \sum_{i = 1}^d
            \frac{\rho_{ii}}{2} r_i(1-r_i) \partial_{ii} f(r) \\
            &\quad
            + \sum_{k = 1}^d
            \int_{[0,1]^d}
            \bigl(
                r_k f(r+(1-r)\odot u)
                +(1-r_k)f(r-r\odot u)
                -f(r)
            \bigr)
            \mathcal{Q}_k(du),
        \end{split}
    \end{equation}
    for $f \in \mathcal{C}^2([0,1]^d)$, and  $\odot$ denotes the Hadamard product of componentwise multiplication.
\end{theorem}

%Theorem \ref{theo:SDEProperties} will be proved in Section \ref{sec:proof}.
Since we work in a multitype setting, known existence and uniqueness results cannot be applied directly. We therefore provide a complete proof. Due to the technical nature of the proof, we defer it to Appendix \ref{proof_theo:SDEProperties}. % This proof is based
%on the sequential sampling procedure introduced below, and also provides a link
%with multitype continuous-state branching processes, which will be discussed in
%the next section. {\color{red}[I guess this last line should be deleted?] }\textcolor{blue}{yes, and maybe this whole short paragraph, once we have reordered the proofs.}

\medskip
Once existence has been established, we apply standard generator methods to obtain a moment duality for the multitype $\Lambda$-Wright--Fisher process. This duality generalizes known results for previously studied special cases. The proof of this result is deferred to Appendix \ref{proof_theo:block_duality}.
For $r \in \mathbb{R}_+^d$ and $n \in \mathbb{N}^d$ we write
$r^n := \prod_{i = 1}^d r_i^{n_i}$. We also use the usual multi-index notation: for $n \in \mathbb{N}_0^d$ and $k \in [n]_0$,
$\binom{n}{k} := \prod_{i = 1}^d \binom{n_i}{k_i}.$

\begin{theorem} \label{theo:block_duality}
    %Fix $z=(z_1,...,z_d)\in(0,\infty)^d$.
    Fix %$(\rho, \mathcal{Q})$, a defining pair for a multitype $\Lambda$-coalescent.
    $\rho \in \mathbb{R}_+^{d\times d}$ and a vector of measures $\mathcal{Q} = (\mathcal{Q}_1, \ldots, \mathcal{Q}_d)$ on $[0, 1]^d$ satisfying \eqref{eq:intcond_lambda}.
   For every $r \in [0,1]^d$, every $n \in \mathbb{N}_0^d$, and every $t \geq 0$, we have
    \begin{equation}\label{eq:duality}
        \mathbb{E}_r \Bigl[ \prod_{i = 1}^d \bigl(\mathsf{R}_i(t)\bigr)^{n_i} \Bigr]
        = \mathbb{E}_n \Bigl[ \prod_{i = 1}^d r_i^{N_i(t)} \Bigr] ,
    \end{equation}
    where $N$ is a Markov process on $\mathbb{N}_0^d$ with initial value $N_0=n$ and transition rates, for $n,m\in\mathbb{N}_0^d$, given by
    \begin{equation}\label{eq:rates_app_dual}
        q_{nm} = \begin{dcases*}
            \binom{n_i}{2} \Bigl( \rho_{ii} + \int_{[0, 1]^d} u^{2 e_i} (1 - u)^{n - 2 e_i} \mathcal{Q}_i(du) \Bigr)
            & if $m = n - 2 e_i + e_i$, \\
            n_j \Bigl( \rho_{ij} + \int_{[0, 1]^d} u^{e_j} (1 - u)^{n - e_j} \mathcal{Q}_i(du) \Bigr)
            & if $m = n - e_j + e_i$, \\
            \binom{n}{k} \int_{[0, 1]^d} u^k (1 - u)^{n - k}
            \mathcal{Q}_i(du)
            & if $m = n - k + e_i$ with $k \in [n]_0 \setminus \mathfrak{E}_i$, \\
            0 & otherwise,
        \end{dcases*}
    \end{equation}
    where $\mathfrak{E}_i = \{0, e_i, 2 e_i\} \cup \bigcup_{j \in [d] \setminus \{i\}} \{e_j\}$ for each $i \in [d]$.
\end{theorem}

As usual for this type of duality, the expectation on the left-hand side of \eqref{eq:duality} refers to the process $\mathsf{R}$ with initial value $r$, whereas the expectation on the right-hand side refers to the process $N$ with initial value $n$. Observe that, for a fixed pair $(\rho,\mathcal{Q})$ %\in\mathbf{L}^{\mathrm{coal}}$ 
satisfying the conditions of Definition \ref{def:multitype-lambda}, the process $N=(N(t))_{t\geq 0}$ is exactly the block-counting process of the corresponding multitype $\Lambda$-coalescent.

\begin{example} As expected, by choosing $\mathcal{Q}_i \equiv 0$ for all $i \in [d]$, as in Example \ref{ex:mkingman}, we obtain the block-counting process of the multitype Kingman coalescent, and \eqref{eq:lambdaWF} reduces to the stepping-stone model, with no selection nor mutation, of Shiga \cite{Shiga1982}, for which the duality is well known. Similarly, the choices in Example \ref{ex:seedbank} recover the duality between the seed-bank coalescent and the seed-bank diffusion studied in \cite{Blathetal2016}. \end{example}

 \begin{example}
     Taking the measures $\mathcal{Q}_1, \ldots, \mathcal{Q}_d$ as in Example \ref{ex:betas}, we obtain a frequency process whose generator is described by 
     \begin{align*}
         \mathcal{G} f(r) 
         & = \sum_{i = 1}^d \sum_{j \in [d] \setminus \{i\}} \rho_{ij} (r_j - r_i) \partial_i f(r) \\ 
         & \quad + \sum_{i = 1}^d \sum_{j=1}^d c_{ij} \int_0^1 \Bigl( r_i f(r + (1 - r_j) s e_j) + (1 - r_i) f(r - r_j s e_j) - f(r)  \Bigr) s^{1 - \alpha_{ij}} (1 - s)^{\alpha_{ij} - 1} \frac{ds}{s^2} . 
     \end{align*}
     Suppose that the frequency process gives us the relative frequency of individuals with characteristic $+$ in $d$ populations. 
     We note that in this model, when there is a reproductive event, it involves a source population $i \in [d]$, a recipient population $j\in[d]$, and a value $s$ sampled from the $\mathrm{Beta}(2 - \alpha_{ij}, \alpha_{ij})$ distribution. 
     With these given at the time of a reproductive event, the relative frequency of individuals with characteristic $+$ in population $j$ increases by $(1 - r_j) s$ with probability $r_i$, and decreases by $r_j s$ with probability $1 - r_i$.  
 \end{example}

\subsection{Multitype continuous state branching processes}
\label{subsec:csbp}
Continuous state branching processes (CSBPs) are Mar\-kov processes with values in $[0,\infty]$ that are the continuous space and time analogues of Galton-Watson processes, and thus give rise to an interpretation of population dynamics.  A one-dimensional CSBP $X$ is described in terms of its infinitesimal generator given, for any $f\in\mathcal{C}_0^2(\mathbb{R}_+)$, by
\begin{equation}
    \mathcal{L}_Xf(x)
     = x b f'(x)
    + x \frac{c}{2} f''(x)
    + x \int_{(0,\infty)} (f(x+w) - f(x) - w 1_{(0,1]}(w) f'(x)) \mu(dw), \qquad x \geq 0,
\end{equation}
where $b\in\mathbb{R}$, $\gamma$ and $c$ are non-negative numbers, and $\mu$ 
is a measure on $\mathbb{R}_+$ satisfying the integrability condition $\int_{(0,\infty)} (1\wedge w^2) \mu(dw) < \infty$. We can think of $b$ as a drift or mass production term, $c$ as a diffusion or variance term, and $\mu$ is associated to large reproduction events.

\medskip

This class of processes has the branching property, meaning that the sum of independent copies of the process started at $x$ and $y$ respectively has the same law as the process started at $x+y.$  According to Lamperti \cite{Lamperti1967}, $X$ can also be obtained from a L\'evy process with nonnegative jumps and is characterised via its Laplace-exponent.

\medskip

In what follows we will consider \emph{multitype continuous state branching processes} (MCSBPs), which are strong Markov processes $X=(X(t))_{t \geq 0}$ with values in $\mathbb{R}_+^d$, that generalise the above setup.
We will think of $d$ different types, or equivalently of a system of $d$ colonies, with interacting CSBPs for each type (or colony). 

For a formal definition we follow Duffie et al. \cite{duffieAffineProcessesApplications2003} (see also \cite{barczyStochasticDifferentialEquation2015a,watanabeTwoDimensionalMarkov1969}). Let $U_d := \mathbb{R}_+^d \setminus \{0\}$. % and $e_i$ be the $i$-th canonical vector. 
For $x \in \mathbb{R}_+^d$ we write $|x| = \sum_{i = 1}^d x_i$ for the $l^1-$norm. Let $\xi=(\xi_1,\dots,\xi_d):U_d \to [-1,1]^d$ be the continuous truncation function given by
    \begin{equation*}
        \xi_i(x) = (1 \wedge |x_i|) \frac{x_i}{|x_i|}.
    \end{equation*}
%\textcolor{cyan}{[As $U_d \subset \mathbb{R}_+^d$, it holds that $\xi_i(w) = 1 \wedge w_i$. Could we use $1 \wedge w_i$ instead of $\xi_i(w)$?]}
Moreover, we write $\mathbb{R}_{(+)}^{d \times d}$ for  the set of $d \times d$ matrices $B = (b_{ij})_{i, j \in [d]}$ such that $b_{ij} \in \mathbb{R}_+$ for $i \neq j$.
%Moreover, for $d \in \mathbb{N}$ we put $[d] = \{1, \ldots, d\}$ and $[d]_0 = \{0\} \cup [d]$.

\begin{definition}[Multitype continuous state branching processes] \label{def:mcsbp}
Let $B = (b_{ij})_{i, j = 1}^d \in \mathbb{R}_{(+)}^{d \times d}$, $c \in \mathbb{R}_+^d$, and for $i \in [d]$, let $\mu_i$ be a Borel measure on $U_d$ such that
        \begin{equation} \label{eq:integrability_conditions_mu}
            \int_{U_d}
            \Bigl( \xi^2_i(w) + \sum_{j \in [d] \setminus \{i\}}  \xi_j(w) \Bigr)
            \mu_i(dw)
            < \infty.
        \end{equation}
Write $\mu=(\mu_i)_{i\in[d]}.$ 
A \emph{multitype continuous-state branching process}, abbreviated as \emph{MCSBP}, with $d$ types
%A \emph{$d-$ type continuous state branching process} ($d$-type CSBP) 
%{\color{red} maybe it is confusing using two notations: MSCBP and $d$-type CSBP? Maybe we should homogeneize the notation?} 
and defining triplet $(B,c,\mu)$, is a Markov process $X=(X(t))_{t \geq 0}$ on $\mathbb{R}^d_+$ with generator $\mathcal{L}$ that acts on functions $f \in \mathcal{C}_c^2(\mathbb{R}_+^d)$ %\textcolor{blue}{what is the difference between $C^2_0$ in the one-type case and $C^2_c$ in the multittype?}{\color{red} I guess there is actually no difference , I think $C^2_0$ is also a core for the generator in both cases. But to be safe, following Duffie etal. \cite{duffieAffineProcessesApplications2003}, I think we can stick with $C^2_c(\mathbb{R}^d_+)$} by
\begin{align} \label{eq:branching_generator}
    \mathcal{L} f(x) = {}
    &\langle B x, \nabla f(x) \rangle
    + \sum_{i = 1}^d c_i x_i \partial_{ii} f(x) \\
    &  + \sum_{i = 1}^d x_i \int_{U_d}
    \bigl( f(x + w) - f(x) - (1\wedge w_i) \partial_i f(x) \bigr) \mu_i(dw) .\notag
\end{align}
%\textcolor{cyan}{[See remark above on $\xi_i(w) = 1 \wedge w_i$.]}
\end{definition}
Under the above assumptions on the parameters, in particular on $\mu,$ this process is well-defined (see \cite{duffieAffineProcessesApplications2003}). The interpretation of the parameters is as follows: We interpret $b_{ij}$ as the rate at which mass in the subpopulation of type $i$ is produced according to the mass of type $j$, the value $c_i$ as a diffusion (or variance) term within type $i,$ and the measure $\mu_i$ governs large reproduction events of type $i$. %\textcolor{blue}{maybe comment on the conditions? Bertoin asked me whether the large reproduction of one type in the other type doesn't destroy the branching property, can we say concisely in a non-technical way why the conditions we have ensure that it goes well? }{\color{red} Well there are some characterization results like the classical from \cite{watanabeTwoDimensionalMarkov1969} that show that in order to preserve the branching property the branching mechanism has to be driven by a subordinator outside the diagonal. Which gives $B\in\mathbb{R}^{d\times d}_{(+)}$ and \eqref{eq:integrability_conditions_mu}, but they are analytical results without that much intuition. In my opinion, intuitively, it goes back to the approximation of multitype G-W processes, which can be represented as skip-free random walks time changed. But the random walks can only take the value $-1$ in the diagonal because the individual that reproduces has to die, and that implies that the components associated to cross-branching can only take non-negative values.}

We now present four examples of MCSBPs that are related to the examples we presented for the multitype $\Lambda$-coalescents, as will soon be evident. 
%The first three will be seen to be related to Examples \ref{ex:mkingman}, \ref{ex:seedbank}, and \ref{ex:betas}. 
%In the fourth example, we introduce{\color{red} I guess stable MCSBPs are well known to use the term "introduce"?} the notion of a stable MCSBP.

\begin{example}[Multitype Feller diffusion] \label{ex:fellers}
    We say that an MCSBP is a multitype Feller diffusion if $\mu_i \equiv 0$  for all $i \in [d]$. % and $c_i \geq 0$ for all $i \in [d]$ {\color{red} This last condition is always fulfilled for a MCSBP, right?} \textcolor{cyan}{[Right. I think it is a leftover of asking $c_i > 0$.]}. 
    In this setting, there are no large reproduction events. 
    Moreover, when the system is in state $x \in \mathbb{R}_+^d$, the subpopulation of type $i$ evolves as a Feller diffusion with variance term %} {\color{red}what is a variance coefficient?} 
    $c_i$ and $b_{ij} x_j$ is the rate at which mass of type $i$ is produced from the subpopulation of type $j$. 
\end{example}

\begin{example}[Feller diffusion with reservoir] \label{ex:reservoir}
    Consider the case $d = 2$. 
    Assume that $\mu_1 = \mu_2 \equiv 0$, $c_1 > 0$, $b_{22} = c_2 = 0$, and $b_{12}, b_{21} > 0$. 
    In particular, type $2$ experiences no branching fluctuations, neither through diffusion nor jump mechanisms.
    Under these assumptions, type $2$ acts as a reservoir for type $1$. 
    Indeed, large reproduction events in type $1$ do not transfer mass to type $2$. %{\color{red} But if $b_{12}>0$ doesn't this imply that in a reproduction event in colony 1 some mass is transferred to colony 2? I guess the only difference. Additionally if colony 2 acts as a reservoir should we impose $b_{22}=0$ too?}  \textcolor{cyan}{[Agreed]}
    However, the two types remain coupled through the linear interaction: Type $1$ receives mass from type $2$ at rate $b_{12} x_2$, while type $2$ is influenced by type $1$ at rate $b_{21} x_1$. 
    Consequently, type $2$ acts as a source or sink of mass for type $1$.
\end{example}

\begin{example}[Entry-wise \textcolor{blue}{$\alpha$-}stable MCSBP] \label{ex:stableCSBPs}
    For each $i,j \in [d]$, let %$J_i \subset [d]$ be non-empty, and suppose that for each $j \in J_i$, 
    $\alpha_{ij} \in (0, 1 + \delta_{ij})$. 
    We say that an MCSBP is an entry-wise $\alpha$-stable MCSBP %{\color{red} We say that a MCSBP is a stable MCSBP?} 
    if $c_i = 0$ for all $i \in [d]$, and if for each $i \in [d]$, % as well as ,
    \begin{equation} \label{eq:muEntryStable}
        \mu_i(dw) = \sum_{j \in [d]} \mathrm{c}_{ij} \frac{dw_j}{w_j^{1 + \alpha_{ij}}} \prod_{k\in [d],k\neq j}\delta_0(dw_k),
    \end{equation}
    %{\color{red} Maybe clarify the notation $\bigotimes$?} \textcolor{cyan}{[Done in the Beta example.]}
    for some non-negative constants $c_{ij}$ with $i,j\in[d]$.
    In this case, the dynamics are driven entirely by large reproduction events. 
    More precisely, a reproduction event that occurs in type $i \in [d]$ affects only types $j\in[d]$ such that $c_{ij}>0$; for each $j \in [d]$, the mass transferred to type $j$ has a power-law distribution with index $\alpha_{ij}$. %, while types in $[d] \setminus J_i$ receive no mass. \textcolor{blue}{again I think introducing the $J_i$ doesn't give an advantage, and I would prefer to write \begin{equation} \label{eq:muEntryStable}
      %  \mu_i(dw) = \sum_{j \in [d]} \mathrm{c}_{ij} \frac{dw_j}{w_j^{1 + \alpha_{ij}}} \prod_{k\in [d],k\neq j}\delta_0(dw_k).
    %\end{equation}, if there is no particular reason otherwise.}
\end{example}

\begin{example}[\textcolor{blue}{$\alpha-$}stable MCSBP] \label{ex:realStableMCSBP}
    For each $i \in [d]$, let $\alpha_i \in (0, 2)$ be given, and let $\eta_i$ be a finite measure on $\mathbb{S}_+^d := \{x \in U_d : \lVert x \rVert = 1\}$, with $\eta_i(\mathbb{S}_+^d \setminus \{e_i\}) = 0$ if $\alpha_i \in [1, 2)$. 
    We say that a MCSBP is a\textcolor{blue}{n $\alpha$-}stable MCSBP if, for each $i \in [d]$, $c_i=0$ and 
    \begin{equation} \label{eq:muStable}
        \mu_i(A) = \int_{\mathbb{S}_+^d} \eta_i(d\theta) \int_0^\infty \frac{dr}{r^{1 + \alpha_i}} 1_{A}(r \theta) \,. 
    \end{equation}
\end{example}

\begin{remark}
    The constraints $\alpha_{ij} \in (0,1)$ for $i \neq j$ in Example \ref{ex:stableCSBPs}, and the condition $\eta_i(\mathbb{S}_+^d \setminus {e_i}) = 0$ whenever $\alpha_i \in [1,2)$ in Example \ref{ex:realStableMCSBP}, are imposed to ensure that the integrability condition \eqref{eq:integrability_conditions_mu} holds in both examples.
\end{remark}

A brief comment on the naming convention adopted in Examples \ref{ex:stableCSBPs} and \ref{ex:realStableMCSBP}:
As will be discussed in Section \ref{subsec:topologicalMCSBPs}, Theorem 1 in \cite{caballeroAffineProcessesMathbb2017} shows that an MCSBP (with $d$ types) %$d$-type CSBP 
can be constructed from $d$ independent $d$-dimensional Lévy processes via a multi-parameter time change. %\textcolor{blue}{do we really discuss that? we should definitely.}{\color{red} yeah I think this is used to show the convergence of MCSBPs in terms of the convergence of their parameters.}
From this perspective, Example \ref{ex:stableCSBPs} corresponds to the case where the components of the Lévy processes are independent and each component is a stable process. 
In contrast, in Example \ref{ex:realStableMCSBP}, each Lévy process is itself a $d$-dimensional stable Lévy process. 
This distinction motivates the use of the term ``entry-wise'' in Example \ref{ex:stableCSBPs}.

\section{A homeomorphism between MCBPs and multitype $\Lambda$-coalescents}\label{subsec:popModelHomeomorphism}

We now turn to the main result of this paper linking multitype $\Lambda-$coalescents and MCSBPs. We prove that the spaces of these processes, both endowed with the Prohorov metric of weak convergence on the spaces of c\`adl\`ag paths, are homeomorphic. Moreover, the homeomorphism provides a genealogical relationship between the MCSBPs and the multitype $\Lambda-$coalescents, since it is given explicitly in terms of a one-to-one correspondence of the defining tuples of the two processes. Together with the duality result in Theorem \ref{theo:block_duality} and the sequential sampling procedure that we give below in Section \ref{sect:seqsamp_dualit}, this result allows to interpret the multitype $\Lambda-$coalescents as the local genealogy of MCSBPs in a similar manner as in the one-type case, \cite{birknerAlphaStableBranchingBetaCoalescents2005,caballeroRelativeFrequencyTwo2023,johnstonCoalescentStructureUniform2021}. 

\medskip

Let us recall that a multitype $\Lambda-$coalescent is characterised in terms of a pair $(\rho, \mathcal{Q})$ with $\rho\in \mathbb{R}_+^{d\times d}$ and $\mathcal{Q}=(\mathcal{Q}_1,...,\mathcal{Q}_d)$ a vector of measures $\mathcal{Q}_i$ on $[0,1]^d$ such that 
%$\mathcal{Q}_i\bigl([0, 1]^d \setminus [0, 1)^d\bigr) = 0 \text{ for all } i \in [d]$ and 
$\sum_{i = 1}^d \int_{[0, 1]^d \setminus \{0\}} \sum_{j = 1}^d u_j^{1 + \delta_{ij}} \mathcal{Q}_i(du) < \infty.$ We write $\mathbf{L}_{\mathrm{prop}}^{\mathrm{coal}}$ for the set of these tuples $(\rho, \mathcal{Q})$ that satisfy the additional restriction $\mathcal{Q}_i\bigl([0, 1]^d \setminus [0, 1)^d\bigr) = 0 \text{ for all } i \in [d]$. The defining triplet of a MCSBP is given by $(B,c,\mu),$ with $B\in \mathbb{R}_+^{d\times d}, c\in \mathbb{R}_+^d$ and $\mu=(\mu_1,...,\mu_n)$ measures on $U_d$ such that $\int_{U_d}\Bigl( \xi^2_i(w) + \sum_{j \in [d] \setminus \{i\}}  \xi_j(w) \Bigr)\mu_i(dw) < \infty.$  For a fixed $a\in\mathbb{R}^d_+$, we write $\mathbf{\Psi}_a^{\mathrm{branch}}$ for the set of such tuples such that the diagonal of $B$ is fixed to $b_{ii}=a_i, i\in[d].$

\medskip

Additionally, given $z \in (0, \infty)^d$ and a measure $\mu$ on $U_d$, let $\mathbf{T}_z\mu$ be the pushforward measure of $\mu$ under the mapping $T_z : U_d \to [0, 1]^d$ given by
\begin{equation}\label{fun_T}
    T_z(w_1, \ldots, w_d) = \biggl(\frac{w_1}{w_1 + z_1}, \ldots, \frac{w_d}{w_d + z_d}\biggr).
\end{equation}
We write $T_1$ for the special case $z=(1,...,1)\in (0, \infty)^d$. The role of $z$ is in some sense to fix the total mass of the MCSBP, in a sense that will be made precise in the sequential sampling procedure in the next section.

Finally, we introduce some notation that will be used throughout the paper. 
Let $\mathbf{P}_d(\mathbb{N})$ denote the space of $d$-typed partitions of $\mathbb{N}$, defined in Definition \ref{def: partition-dtype}, and equip it with the metric $d_{\mathbf{P}_d(\mathbb{N})}$ introduced in Section \ref{subsec:multiLambdaCoal}. 
Moreover, for a metric space $S$, let $D([0,T], S)$ denote the Skorokhod space of càdlàg paths with values in $S$.

With these preparations, we can state the genealogical homeomorphism between multitype $\Lambda-$coalescents and MCSBPs.

\begin{theorem}[Homeomorphism] \label{theo:homeomorphismSpaces}
    Consider $a \in \mathbb{R}^d$ and $T>0$ fixed.
    Let $\mathcal{C}$ be the space of multitype $\Lambda$-coalescents with characteristic pair in $\mathbf{L}_{\mathrm{prop}}^{\mathrm{coal}}$ endowed with the Prohorov metric over the space of probability measures on the space $D([0, T], (\mathbf{P}_d(\mathbb{N}), d_{\mathbf{P}_d(\mathbb{N})}))$. In addition, consider the space $\mathcal{B}$ of multitype continuous state branching prcesses with characteristic triplet in  $\mathbf{\Psi}_a^{\mathrm{branch}}$ equipped with the Prohorov topology over the space of probability measures on the space $D([0, T],\mathbb{R}^d_+)$ endowed with the uniform Skorokhod topology. Then, $\mathcal{C}$ and $\mathcal{B}$ are homeomorphic.
     
     More precisely, we can define a explicit homeomorphism $H : \mathbf{\Psi}_a^{\mathrm{branch}} \to \mathbf{L}_{\mathrm{prop}}^{\mathrm{coal}}$ that maps a MCSBPs with triplet $(B, c, \mu)$ to the mutitype $\Lambda$-coalescent with pair $(\rho, \mathcal{Q})$ given as
    \begin{equation} \label{eq:explicitHomeomorphism1}
        \rho_{ii} = 2 c_i, \quad
        \rho_{ij} = b_{ji}
        \quad\text{and}\quad
        \mathcal{Q}_i =  \mathbf{T}_1 \mu_i
        \quad\text{for } i \in [d] \text{ and } j \in [d] \setminus \{i\}.
    \end{equation}
    Its inverse $H^{-1} : \mathbf{L}_{\mathrm{prop}}^{\mathrm{coal}} \to \mathbf{\Psi}_{a}^{\mathrm{branch}}$ maps a pair $(\rho', \mathcal{Q}')$ to $(B', c', \mu')$ with
    \begin{equation}
        c_i' = \frac{\rho_{ii}'}{2}, \quad
        b_{ij}' = \rho_{ji}'  \quad\text{and}\quad
        \mu_i' =  \mathbf{T}_1^{-1} \mathcal{Q}_i'
        \quad\text{for } i \in [d] \text{ and } j \in [d] \setminus \{i\}.
    \end{equation}
\end{theorem}
%{\color{red} Here maybe we need to be careful because at this level the homeomorphism is given between the parameter spaces but it does not explicitly states that it at the level of processes for instance that its continuous under the Skorokhod topology on their respective s}
This result is a multidimensional extension of Theorem 5 in \cite{caballeroRelativeFrequencyTwo2023}, which establishes that the space of $\Lambda$-coalescents with no atom at $\{1\}$ is homeomorphic to the space of CSBPs. The proof, which is provided in Section \ref{sect:proof_homeo}, follows the strategy of first proving a homeomorphism between the subspaces $\mathbf{\Psi}_a^{\mathrm{branch}}$ and $\mathbf{L}_{\mathrm{prop}}^{\mathrm{coal}}$ of the relative parameter spaces, and then showing that these subspaces along with the topologies were chosen appropriately, so that the homeomorphism extends to the respective spaces of processes.
We stress that in establishing Theorem \ref{theo:homeomorphismSpaces}, the sequential sampling procedure and the duality result, cf. Theorem \ref{theo:block_duality}, are crucial. 
For example, it will turn out that the diagonal entries of the matrix $B$ do not play a role in the duality result, which is why we fix these, leading us to consider $\mathbf{\Psi}_a^{\mathrm{branch}}$ instead of $\mathbf{\Psi}^{\mathrm{branch}}$. 
Going further, the duality applied to sequentially sampled frequency process $R^{(z)}$ that we will define below, provides an interpretation of a coalescent process as the local genealogy of the MCSBP. 

\medskip

    Theorem 4.2 in \cite{johnstonCoalescentStructureUniform2021} states a related result, where the authors identify the rates of the local genealogy of an MCSBP, which coincide with those obtained from the parameters in \eqref{eq:explicitHomeomorphism1}.
    This agreement is natural, as both approaches rely on the transformation \eqref{fun_T}. 
    However, the scope and methods differ. 
    The result of \cite{johnstonCoalescentStructureUniform2021} focuses on the explicit characterization of genealogical rates and is proved using Poissonization techniques.
    In contrast, our result is a homeomorphism on the level of processes and yields the genealogical rates as a consequence of a structural correspondence.

\begin{remark}\label{rem:homeo}
   In fact, there is a whole family of homeomorphisms. %In fact, f
   For any $z \in (0, \infty)^d$ (again interpreted as a fixed total mass) we can define $H_z : \mathbf{\Psi}_a^{\mathrm{branch}} \to \mathbf{L}_{\mathrm{prop}}^{\mathrm{coal}}$ as mapping a triplet $(B, c, \mu)$ to the pair $(\rho, \mathcal{Q})$ given as
    \begin{equation} \label{eq:explicitHomeomorphism}
        \rho_{ii} = 2 \frac{c_i}{z_i}, \quad
        \rho_{ij} = b_{ji} \frac{z_i}{z_j}
        \quad\text{and}\quad
        \mathcal{Q}_i = z_i \mathbf{T}_z \mu_i
        \quad\text{for } i \in [d] \text{ and } j \in [d] \setminus \{i\}.
    \end{equation}
    Its inverse $H_z^{-1} : \mathbf{L}_{\mathrm{prop}}^{\mathrm{coal}} \to \mathbf{\Psi}_{a}^{\mathrm{branch}}$ maps a pair $(\rho', \mathcal{Q}')$ to $(B', c', \mu')$ with
    \begin{equation}\label{eq:explicit_inverse}
        c_i' = z_i \frac{\rho_{ii}'}{2}, \quad
        b_{ij}' = \rho_{ji}' \frac{z_i}{z_j} \quad\text{and}\quad
        \mu_i' = \frac{1}{z_i} \mathbf{T}_z^{-1} \mathcal{Q}_i'
        \quad\text{for } i \in [d] \text{ and } j \in [d] \setminus \{i\}.
    \end{equation}
    Obviously, $H$ above corresponds to the choice $z=(1,...,1).$ Since in the sequential sampling procedure it is useful to allow for the more general choice of $z$, we will provide a proof of Theorem \ref{theo:homeomorphismSpaces} for general $z,$ which isn't harder than the special case.
\end{remark}
\medskip

\subsection{The frequency process associated to a MCSBP and the sequential sampling procedure}\label{sect:seqsamp_dualit}
We now turn to the sequential sampling procedure. We have already defined the general multitype Wright-Fisher process with jumps that is dual to the block counting process of a multitype $\Lambda-$coalescent. We will now consider the relative frequency process related to MCSBPs. In general, this frequency process is not Markovian, but as in \cite{caballeroRelativeFrequencyTwo2023} we will apply a \emph{sequential sampling procedure} that goes back to ideas of Gillespie (where it was called ``culling'', \cite{Gillespie74}). 
The idea is to keep the total population size constant on a dense set of times, and let the frequency process run independently, but with a fixed constant population size. 
This procedure, described in detail below, will lead to a Markov process  $R^{(r,z)}$ that may be interpreted as the frequency process of our MCSBP with fixed total mass process of constant size $z$. 
This ``sequentially sampled'' version of the frequency process of an MCSBP is
in fact a multitype Wright--Fisher frequency process of the type
defined in Section \ref{sec:lambdaWFfreq}. It is therefore in duality with the
block-counting process of a multitype $\Lambda$-coalescent.
Thus we will obtain the interpretation of multitype $\Lambda$-coalescents as \emph{(local) genealogies} of MCSBPs.
%Roughly speaking, this process describes the evolution of the relative proportions of one population inside a larger population evolving under the same branching dynamics. 
%As in the classical Fleming--Viot and Wright--Fisher setting, the genealogical structure of the population will be encoded in the evolution of these frequencies. 
%In our framework, however, the presence of jumps and multiple types leads naturally to a multidimensional jump-diffusion process whose dynamics are governed by the same parameters that define the multitype $\Lambda$-coalescent.

\begin{definition}[Population model, frequency and total mass process] Let $X$ and $Y$ be two independent MCSBPs (with $d$ types) %$d-$type CSBPs{\color{red} MCSBPs?} 
as defined in Section \ref{subsec:csbp} with the same triplet of parameters $(B,c,\mu).$ We interpret them as two independent populations with the same dynamics. Let %{\color{red}
\begin{equation}\label{eq:defRZ}
    R_i(t\wedge \kappa) = \frac{X_i(t\wedge \kappa)}{X_i(t\wedge \kappa) + Y_i(t\wedge \kappa)}
    \quad\text{and}\quad
    Z_i(t\wedge \kappa) = X_i(t\wedge \kappa) + Y_i(t\wedge \kappa) ,
    \quad i \in [d],
\end{equation}
where $\kappa := \min_{i \in [d]} \inf \bigl\{ t \geq 0 : Z_i(t) \in \{0, \infty\} \bigr\} .$ 
 
We call $R = (R_i(t))_{i\in[d],t\geq 0}$ the \emph{frequency process} and $Z = (Z_i(t))_{i\in[d],t\geq 0},$ the \emph{total mass process} of the population model $(X,Y).$
\end{definition}

Clearly, the pair $(X,Y)$ is a Markov process, and therefore also $(R,Z)$ is a Markov process up to time 
$\kappa$. 
The process $(R,Z)$ up until time $\kappa$ can be characterized as the solution to a martingale problem associated to its extended infinitesimal generator $\mathcal{A}$, that acts on functions
$f \in \mathcal{C}_c^2([0,1]^d\times \mathbb{R}_+^d)$ by 
\begin{align}\label{eq:branching_two_colonies}
        \mathcal{A} f(r, z) = {} 
        & \sum_{i = 1}^d c_i z_i \biggl( \frac{(1 - r_i) r_i}{z_i^2} \partial_{r_i r_i} f(r, z) + \partial_{z_i z_i} f(r, z) \biggr) \notag\\
        & + \sum_{i = 1}^d \sum_{j = 1}^d b_{ij} \biggl(
        \frac{z_j}{z_i} (r_j - r_i) \partial_{r_i} f(r, z) + z_j \partial_{z_i} f(r, z)
        \biggr) \\
        & + \sum_{i = 1}^d \int_{[0, 1)^d \setminus \{0\}} \bigl(
        \begin{aligned}[t]
             r_i f(r + (1 - r) \odot u, z + T_z^{-1}(u)) 
             & + (1 - r_i) f(r - r \odot u, z + T_z^{-1}(u)) \notag\\ 
             & - f(r, z) - \xi_i(T_z^{-1}(u)) \partial_{z_i} f(r, z) \bigr) z_i \mathbf{T}_z \mu_i(du) \,.
        \end{aligned}
    \end{align}
The precise formulation of this martingale problem is given in Proposition \ref{prop:inf_gen_rz}.
%}

We would like to study the frequency process $R$, which is the first coordinate of the pair $(R,Z)$ defined in \eqref{eq:defRZ}, but by itself is not Markovian. However, supposing $Z$ as fixed and equal to some vector $z \in \mathbb{R}_+^d$, we formally write the generator $\mathcal{A}$ acting on functions that depend only on $r \in [0, 1]^d$. In order to keep in mind the dependence on the total mass in each of the colonies, we write $\mathcal{A}^{(z)}$ for the resulting operator.
\medskip

%\textcolor{cyan}{Consider $g : \mathbb{R}_+^d \times \mathbb{R}_+^d \to [0, 1]^d$ to be defined by $g_i(x, y) = x_i / (x_i + y_i)$ for each $i \in [d]$. 
%For a function $f : [0, 1]^d \to \mathbb{R}$ we formally define $f(r) := f \circ g(x, y)$. Then,
%}
Then, looking closely at \eqref{eq:branching_two_colonies}%{\color{red} I am not sure that the following generator $\mathcal{A}^{(z)}$ follows directly from \eqref{eq:branching_two_colonies}. I think that there is a step missing which is the extended generator of the process $(R,Z)$ not the generator of $(X,Y)$}
, we see that the action of $\mathcal{A}^{(z)}$ on a test function $f \in \mathcal{C}^2([0, 1]^d)$
is formally given by
\begin{align} \label{eq:gen_d_colonies}
    \mathcal{A}^{(z)} f(r) = {}
    & \sum_{i = 1}^d \sum_{j \in [d] \setminus \{i\}} b_{ij} \frac{z_j}{z_i} (r_j - r_i)
    \partial_i f(r)
    + \sum_{i = 1}^d \frac{c_i}{z_i} r_i (1 - r_i) \partial_{ii} f(r) \notag\\
    % & + \sum_{i = 1}^d r_i \int_{[0,1)^d \setminus \{0\}} \bigl(
    % f(r + (1 - r) \odot u) - f(r) - (1 - r_i) u_i \partial_i f(r) \bigr)
    % z_i \mathbf{T}_z \mu_i(du) \notag\\
    % & + \sum_{i = 1}^d (1 - r_i) \int_{[0,1)^d \setminus \{0\}} \bigl(
    % f(r - r \odot u) - f(r) + r_i u_i \partial_i f(r) \bigr) z_i \mathbf{T}_z\mu_i(du),
    & + \sum_{i = 1}^d \int_{[0, 1)^d \setminus \{0\}} \bigl( r_i f(r + (1 - r) \odot u) + (1 - r_i) f(r - r \odot u) - f(r) \bigr) z_i \mathbf{T}_z \mu_i(du) \,.
\end{align}
%where $\odot$ denotes the Hadamard product of componentwise multiplication. 
%\textcolor{cyan}{
Observe that \eqref{eq:gen_d_colonies} coincides with \eqref{eq:lambdaWFGen} by setting 
\[
    \rho_{ii} = \frac{2 c_i}{z_i}, \quad \rho_{ij} = b_{ji} \frac{z_i}{z_j}, \quad \mathcal{Q}_i = z_i \mathbf{T}_z \mu_i
\]
for each $i, j \in [d]$ with $i \neq j$. 
Intuitively, if the total mass process is `frozen' at the value $z$%{\color{red} I guess we need to be careful with the phrase "total mass process is frozen at the value $z$"?}\textcolor{blue}{we say here that we are intutive and formal, that should be fine}
, then the associated frequency dynamics are formally described by a multitype $\Lambda$-Wright--Fisher process whose defining pair $(\rho,\mathcal{Q})$ is obtained from the branching triplet $(B,c,\mu)$ through the correspondence introduced in \eqref{eq:explicitHomeomorphism}. 
This provides an immediate connection between MCSBPs and multitype $\Lambda$-coalescents, that we will apply in the proofs of our main results.
%}

\medskip

%In order to prove that this is indeed the generator of a bona fide strong Markov process (which will also have the Feller property), we apply the following procedure:
In order to obtain the associated frequency process to a MCSBP described by \eqref{eq:gen_d_colonies}, we apply the following procedure:
We construct a pure jump Markov process, which will converge to a process whose generator is given by \eqref{eq:gen_d_colonies}. This limiting process is then reasonably interpreted as the frequency process of population $X$ within the population process $(X,Y)$, if the total mass $X+Y$ is being fixed to $z.$ %{\color{red} Maybe we need to be more precise with "total mass fixed to $z$?}\textcolor{blue}{again, we are providing intuition and say so, I'm fine with that.}
% As a frequency process of a population model, it is natural to expect that it gives rise to a moment duality, and indeed the dual process will turn out to be the block counting process of the multitype $\Lambda-$coalescent corresponding to the MCSBPs $X$ and $Y$ via the homeomorphism in Theorem \ref{theo:homeomorphismSpaces}. 
As a consequence of Theorem \ref{theo:block_duality}, the dual process of the sequentially sampled frequency process is the block counting process of the multitype $\Lambda$-coalescent corresponding to the MCSBPs $X$ and $Y$ via the homeomorphism in Theorem \ref{theo:homeomorphismSpaces}.

 \medskip

We now construct the approximating sequence of jump processes. 
\begin{definition}[Discrete approximation process]
Fix $r \in [0, 1]^d$ and $z \in (0, \infty)^d$. Let $\varepsilon, L > 0$ such that $0 < \varepsilon < \min_{i \in [d]} z_i \leq \max_{i \in [d]} z_i < L$, and let $(N(t))_{t\geq 0}$ be a Poisson process of rate 1.
For each $n \in \mathbb{N}$, let $\overline{R}^n=(\overline{R}^n(t))_{t\geq 0}$ be the process with values in $[0, 1]^d$ defined using the following scheme:
\end{definition}
\begin{enumerate}
    \item Set $\mathcal{R}^{(n)}(0) = r$.
    \item For every $m \in \mathbb{N}$ let $(R^{(m)}, Z^{(m)})$ be a realisation of $(R, Z)$ started at $(\mathcal{R}^{(n)}(m-1), z)$.
    \item Put
        \[
            \mathcal{R}^{(n)}(m) = R^{(m)}\biggl( \frac{1}{n} \wedge \tau^{(m)} \biggr) \,,
        \]
        where
        \[
            \tau^{(m)} := \min_{i \in [d]} \inf \{ t \geq 0 : Z_i^{(m)}(t) \notin (\varepsilon, L) \} \,.
        \]
    \item For every $t \in \mathbb{R}_+$, set $\overline{R}^n(t) := \mathcal{R}^{(n)}(N(nt))$.
\end{enumerate}
By construction, the discrete approximation process  $\overline{R}^n=(\overline{R}^n(t))_{t\geq 0}$ is a pure jump Markov process with exponential jump times of rate $n$, and whose transition kernel $\rho_{z, n} : [0, 1]^d \times \mathcal{B}([0, 1]^d)\to [0, 1]$ is given by
\begin{equation}
        \rho_{z, n}(r', A) := \mathbb{P}_{r', z}\Bigl( R(n^{-1} \wedge \tau) \in A, Z(n^{-1} \wedge \tau) \in \mathbb{R}_+^d \Bigr) \,,
\end{equation}
where $\tau = \min_{i \in [d]} \inf \{t \geq 0 : Z_i(t) \notin (\varepsilon, L)\}$, and $\mathbb{P}_{r', z}$ is the law of the process $(R, Z)$ started at $(r', z)$.
Moreover, $\overline{R}^n$ has infinitesimal generator whose action on functions $f \in \mathcal{C}([0, 1]^d)$ is given by
\begin{equation}
        \mathcal{A}_n^{(z)} f(r') = n \int_{[0, 1]^d} \bigl( f(u) - f(r') \bigr) \rho_{z, n}(r', du), \quad r' \in [0, 1]^d \,.
\end{equation}
Note that in this construction, we are sampling whole realisations of the process $\{(R(t \wedge \tau), Z(t \wedge \tau)) : t \geq 0\}$, which all have the same initial value for the second component $Z$, in order to construct the skeleton chain $\mathcal{R}^{(n)}$.

\medskip

By construction, $\overline{R}^n$ evolves as the first coordinate of the process $(R, Z)$ but, as $n$ tends to infinity, the fluctuations of $Z$ around $z$ become increasingly smaller, since the inter-jump times vanish in the limit.
Therefore, as $n$ tends to infinity, the choices of $\varepsilon$ and $L$ will become irrelevant, in the sense that the limit of the jump processes will not depend on these constants.
See also \cite{caballeroRelativeFrequencyTwo2023}, Section 4.2.

As in \cite{caballeroRelativeFrequencyTwo2023}, for $z\in (0,\infty)^n$ fixed, one can prove convergence of this sequence of processes to a process $R^{(z)}$ that we call the \emph{sequentially sampled process}.  
We write $R^{(z,r)}$ to indicate that the process starts from $r \in [0,1]^d$ and is associated with the initial total mass $z \in (0,\infty)^d$. To carry out the proof, which is postponed to Section \ref{sec:proof}, we first observe that by Theorem \ref{theo:SDEProperties}, the SDE describing the limiting sequentially sampled frequency process has a unique strong solution, whose generator is given by \eqref{eq:gen_d_colonies}. Then we prove that the approximating process $\overline{R}^{n}$ converges to this solution by showing the convergence of the associated generators.

% The strategy of the proof, carried out in Section \ref{sec:proof}, consists in first establishing the existence and uniqueness of a process solving the system of SDEs given in \eqref{eq:sde_d_colonies}, whose infinitesimal generator is $\mathcal{A}^{(z)}$, and then proving that the approximating process $\overline{R}^{n}$ converges to this solution via convergence of the associated generators. 
% {\color{red} I remember mentioned something about the SDE, i guess maybe we should mention that the existence and uniqueness result is not that straightforward because there are no multidimensional theory for this type of SDEs unlike in \cite{caballeroRelativeFrequencyTwo2023}?}

\medskip

We need some more notation to write down the SDE for the sequentially sampled process. Fix $r \in [0,1]^d$, $z \in (0, \infty)^d$. Recall the pushforward $\mathbf{T}_z\mu$ of a measure $\mu$ under the mapping $T_z $ defined in \eqref{fun_T}. Let $B = (B_1, \ldots, B_d)$ be a $d-$ dimensional Brownian motion, and for $j \in [d]$ 
%let $N_1^j(dt, du, dv)$ and $N_2^j(dt, du, dv)$ be Poisson random measures on $\mathbb{R}_+ \times [0, 1]^d \times \mathbb{R}_+$ with intensity measure $dt \mathbf{T}_z\mu_j(du) dv$, and denote their associated compensated random measures by $\widetilde{N}_1^j(dt, du, dv)$ and $\widetilde{N}_2^j(dt, du, dv)$ . 
let $\tilde{N}_j(dt, du, dv)$ be a compensated Poisson random measure on $\mathbb{R}_+ \times [0, 1]^d \times [0, 1]$ with intensity measure $dt \, z_j \mathbf{T}_z \mu_j(du) \, dv$. %, and denote its associated compensated random measures by $\widetilde{N}_j(dt, du, dv)$. %{\color{red}[ I guess the Poisson random measure is not used so why not omit and just say "consider $\widetilde{N}_j(dt, du, dv)$ a compensated Poisson random measure with..."?]}
All of these random variables are assumed to be defined on the same complete probability
space and independent of each other. Then consider the system of SDEs%\footnote{Formally
%    all of the terms in \eqref{eq:sde_d_colonies} should be multiplied by
%$1_{\{R^{(z, r)}(t) \in [0, 1]^d\}}$, but since later in Section \ref{sec:proof} we prove that the solution is always in $[0,1]^d$, we drop this term to ease the notation. }
\begin{align} \label{eq:sde_d_colonies}
    d R^{(z,r)}_i(t) = {}
    &
    \sum_{j \in [d] \setminus \{i\}}
    \biggl( b_{ij} \frac{z_j}{z_i} + z_j \int_{[0, 1]^d} u_i \mathbf{T}_z \mu_j(du) \biggr)
    (R^{(z,r)}_j(t) - R^{(z,r)}_i(t))  dt \notag\\
    & + \sqrt{ \frac{2c_i}{z_i} R^{(z,r)}_i(t) (1 - R^{(z,r)}_i(t)) } dB_i(t) \notag\\
    % & + \sum_{j = 1}^d \int_{[0,1)^d \setminus \{0\}} \int_0^\infty u_i (1 - R^{(z,r)}_i(t-))
    % 1_{\{v \leq z_i R^{(z,r)}_j(t-)\}} \widetilde{N}_1^j(dt, du, dv) \notag\\
    % & + \sum_{j = 1}^d \int_{[0,1)^d \setminus \{0\}} \int_0^\infty (-u_i R^{(z,r)}_i(t-))
    % 1_{\{v \leq z_i (1 - R^{(z,r)}_j(t-))\}} \widetilde{N}_2^j(dt, du, dv), \notag\\
    & + \sum_{i = 1}^d \int_{[0, 1]^d} \int_0^1 u_i (1_{\{v \leq R_j^{(z, r)}(t-)\}} - R_i^{(z, r)}(t-)) \widetilde{N}_j(dt, du, dv) \,, \notag \\
    R_i^{(z,r)}(0) = {} & r_i,
\end{align}
for $i \in [d]$.
 
\begin{theorem} \label{theo:culling_limit}
   % For any $z \in (0, \infty)^d$ and $r \in [0, 1]^d$, %there exists a unique strong solution $R^{(z,r)}$ to \eqref{eq:sde_d_colonies} with initial condition $R^{(z,r)}(0) = r$. {\color{red} the existence and uniqueness result for a strong solution is already given Theorem \ref{theo:SDEProperties} maybe omit this part from the statement and mention befor the statement that we know the existence of a unique solution by Theorem \ref{theo:SDEProperties}? }
    %Additionally,
    %$R^{(z,r)}$ is a Feller process whose generator is given by \eqref{eq:gen_d_colonies}, and
    For fixed $z \in (0,\infty)^d$, $r\in[0,1]^d$, and $T > 0$,
$\overline{R}^{n}$ converges weakly in $D([0,T],[0,1]^d)$ to
$R^{(z,r)}$ as $n \to \infty$, where $R^{(z,r)}$ is the unique strong solution
to \eqref{eq:sde_d_colonies}.
\end{theorem}

\begin{definition}[Sequentially sampled process]
The process $R^{(z,r)}=(R^{(z,r)}(t))_{t\geq 0}$ of Theorem \ref{theo:culling_limit} is called the \emph{sequentially sampled frequency process} of $(X,Y)$ with total mass $z$ and initial value $r.$
\end{definition}

This now clarifies the genealogical connection between MCSBPs and multitype $\Lambda-$coalescents:
%Recalling Definition \ref{def:multitype-lambda}, Theorem \ref{theo:block_duality}, and the homeomorphism $H_z$ of Theorem \ref{theo:homeomorphismSpaces}, we observe the following for the dual process of a sequentially sampled frequency process:
%If there are currently $n_i$ blocks of type $i$, there are pairwise mergers of blocks of type $i$ at rate $2c_iz_i^{-1},$  type switch from $j$ to $i$ of single blocks at rate $b_{ji} z_i z_j^{-1},$ multiple mergers of blocks at rates governed by $\mathbf{T}_z\mu_i$. 
From Theorem \ref{theo:culling_limit} and Theorem \ref{theo:block_duality} we thus see that the block counting process of a multitype $\Lambda$-coalescent is the moment dual of the frequency process associated to a MCSBP.
The homeomorphism that we give in Theorem \ref{theo:homeomorphismSpaces} essentially sends a MCSBP to the multitype $\Lambda$-coalescent in moment duality to the frequency process obtained through the sequential sampling procedure, which can be interpreted as describing the genealogy of the MCSBP.

Specifically, consider a MCSBP $X$ with characteristic triplet $(B, c, \mu)$. 
From \eqref{eq:sde_d_colonies} and Theorem \ref{theo:block_duality} it follows that the moment dual to the sequentially sampled frequency process with total size $z$ is given by the rates
% Specifically, recall that the transition rates of the moment dual to the sequentially sampled frequency process of a MCSBP $X$ with characteristic triplet $(B, c, \mu)$ are given by \eqref{eq:rates_app_dual}.
% Fix $z \in (0, \infty)^d$. 
% Note that for each $n \in \mathbb{N}_0^d$, each $i \in [d]$ and each $k \in [n]_0 \setminus \{0, e_i\}$ we can rewrite $q_{nm}$, with $m = n - k + e_i$, in \eqref{eq:rates_app_dual} as
\begin{equation}
    q_{nm} = \binom{n}{k} \biggl( \frac{2}{z_i} c_i 1_{\{k = 2e_i\}}
    + \sum_{j \in [d] \setminus \{i\}} b_{ji} \frac{z_i}{z_j} 1_{\{k = e_j\}}
    + \int_{[0, 1)^d \setminus \{0\}} u^k (1 - u)^{n - k} z_i \mathbf{T}_z \mu_i(du) \biggr) ,
\end{equation}
for $n \in \mathbb{N}_0^d$ and $m = n - k + e_i$ with $i \in [d]$ and $k \in [n]_0 \setminus \{0, e_i\}$. The expression within parentheses is obviously a particular case of \eqref{eq:lambda_i_multitype_coal}, with $(\rho,\mathcal{Q})$ given by  the homeomorphism $H_z$ in \eqref{eq:explicitHomeomorphism}.
%\textcolor{blue}{I deleted an equation here that if I see correctly was exactly eq. \eqref{eq:explicitHomeomorphism}, but it's too hot here today, so I hope I didn't get it wrong}
%\begin{equation} \label{eq:coalFromBranch}
%    \rho_{ii} = 2 \frac{c_i}{z_i},\
 %   \rho_{ij} = b_{ji} \frac{z_i}{z_j}
  %  \text{ and }
   % \mathcal{Q}_i(du) = z_i \mathbf{T}_z \mu_i(du)
%\end{equation}
%for each $i \in [d]$ and $j \in [d] \setminus \{i\}$.

\medskip

This is the underlying idea behind the homeomorphism $H_z$ established in Theorem \ref{theo:homeomorphismSpaces} (cf. also Remark \ref{rem:homeo}), and shows why this result is more than just a topological identity.

%\textcolor{blue}{Maybe we should reconsider the order in which we present the proofs, now that we have changed the order of the presentation of the results. At least the introduction to Sect. 3.1 needs to be changed.}

\subsection{Multitype Beta-coalescents and other examples}
%{\color{red} In this section are we using $H_1$ or $H_z$?} \textcolor{cyan}{As written, we use $H_z$. I think that the only place it changes things is in the multitype Beta coalescents, so we could change it to $H_1$. Should we?}

The homeomorphism of Theorem \ref{theo:homeomorphismSpaces} provides a systematic way to pass from MCSBPs to multitype $\Lambda$-coalescents. 
In particular, it allows us to \emph{construct coalescent processes in a canonical way from branching dynamics}. 
This correspondence operates not only at the level of parameters, but also at the level of dynamics through the moment duality of Theorem \ref{theo:block_duality}.
In this section, we illustrate how this mechanism recovers known examples and, more importantly, how it naturally leads to new ones. In particular, we can use the homeomorphism to define multitype Beta-coalescents not just entrywise as in example \ref{ex:betas}, but in a genuinely multitype manner.
Given a MCSBP with parameters $(B,c,\mu)$, the homeomorphism $H_z$ defined in \eqref{eq:explicitHomeomorphism} produces a multitype coalescent with parameters $(\rho,\mathcal{Q})$, which arises as the local genealogy of the branching system. 
The structure of the coalescent is therefore directly inherited from the Lévy measures $\mu_i$, with the measures $\mathcal{Q}_i = z_i \mathbf{T}_z \mu_i$ encoding the distribution of merger events.
A key outcome of this correspondence is that it provides a canonical multitype analogue of the Beta coalescent. 
In the one-dimensional setting, it is well known that $\alpha$-stable CSBPs give rise to Beta coalescents \cite{birknerAlphaStableBranchingBetaCoalescents2005}.
Our framework shows that this phenomenon extends in a natural and structurally meaningful way to the multitype setting.
More precisely, when the Lévy measures are supported on coordinate axes, one recovers entry-wise Beta coalescents, reflecting independent merger mechanisms across types. 
However, in the genuinely multivariate stable case, the homeomorphism produces coalescents whose merger events are no longer independent across types. 
Instead, they are coupled through a common radial parameter and an angular component, yielding a new class of processes that we call multitype Beta coalescents. 
To build intuition for this correspondence, we first discuss some examples in which the resulting coalescent structure is already well understood.

\begin{example}
We begin with the correspondence between multitype Feller diffusions and multitype Kingman coalescents. 
This example is well known; see Section 2.4 in \cite{johnstonMultitypeLcoalescents2022a} and Corollary 4.3 in \cite{johnstonCoalescentStructureUniform2021}. 
Recall from Example \ref{ex:fellers} that $\mu_i \equiv 0$ for all $i \in [d]$. 
It follows that, for each $i \in [d]$, $\mathcal{Q}_i = z_i \mathbf{T}_z \mu_i \equiv 0$. 
Consequently, no multiple mergers occur in the coalescent obtained by applying the homeomorphism $H_z$ to an interacting Feller system.
The dynamics are therefore entirely governed by the diffusion and linear terms: the coefficients $c_i$ determine binary mergers at rates $\rho_{ii} = 2c_i / z_i$, while the matrix $B$ induces type changes at rates $\rho_{ij} = b_{ji} z_i / z_j$. 
This recovers the multitype Kingman coalescent described in Example \ref{ex:mkingman}.
\end{example}

\begin{example}
    An immediate consequence of the above is that the local genealogy of a Feller diffusion with reservoir (Example \ref{ex:reservoir}) is a seed bank coalescent (Example \ref{ex:seedbank}). 
Indeed, consider a Feller diffusion with reservoir with parameters $c_1 > 0$ and $b_{12}, b_{21} > 0$. 
Applying the homeomorphism $H_z$ yields a coalescent with parameters $\mathcal{Q}_1 = \mathcal{Q}_2 \equiv 0$, $\rho_{11} = 2 c_1 / z_1$, $\rho_{12} = b_{21}z_1 / z_2$, $\rho_{21} = b_{12}z_2 / z_1$, and $\rho_{22} = 0$, which are precisely those of a seed bank coalescent.
\end{example}

\begin{example}
    The following example considers entry-wise stable MCSBPs (Example \ref{ex:stableCSBPs}) and entry-wise Beta coalescents (Example \ref{ex:betas}). 
This example follows from the one-dimensional result of \cite{birknerAlphaStableBranchingBetaCoalescents2005}, which identifies Beta coalescents as local genealogies of $\alpha$-stable CSBPs, although the present formulation does not appear to have been explicitly stated in the literature. 
Fix $i \in [d]$, and suppose that $\mu_i$ is of the form \eqref{eq:muEntryStable}. For measurable sets $A_1, \ldots, A_d \subset [0,1]$, by \eqref{eq:explicitHomeomorphism} and the definition of the pushforward measure,
\[
    \mathcal{Q}_i(A_1 \times \cdots \times A_d) 
    = \sum_{j \in [d]} z_i c_{ij} \int_0^\infty 
    \mathbf{1}_{A_j} \Bigl(\frac{w_j}{w_j + z_j}\Bigr) 
    \frac{dw_j}{w_j^{1 + \alpha_{ij}}}
    \prod_{k \in [d] \setminus \{j\}} \delta_0(A_k).
\]
With the change of variables $s_j = \frac{w_j}{w_j + z_j}$ for $j \in [d]$, this becomes
\[
    \mathcal{Q}_i(A_1 \times \cdots \times A_d) 
    = %\Bigl(z_i \, \mathrm{const}_i \prod_{j \in [d]} z_j^{-\alpha_{ij}} \Bigr)
    \sum_{j \in [d]} z_i c_{ij} z_j^{-\alpha_{ij}} \int_0^1 
    \mathbf{1}_{A_j}(s_j) s_j^{1 - \alpha_{ij}} (1 - s_j)^{\alpha_{ij} - 1} \frac{ds_j}{s_j^2}
    \prod_{k \in [d] \setminus \{j\}} \delta_0(A_k).
\]
Hence, $\mathcal{Q}_i$ is of the form \eqref{eq:betaProd}.
In particular, this means that entry-wise Beta coalescents describe the local genealogy of entry-wise stable MCSBPs. 
This is natural since both $\mu_i$ and $\mathcal{Q}_i$ are sums of measures supported on the coordinate axes.
\end{example}

We now turn to an example in which the product structure breaks down, leading to genuinely multivariate interactions. 
\begin{example}

This example concerns stable MCSBPs (Example \ref{ex:realStableMCSBP}). 
Fix $i \in [d]$, and suppose that $\mu_i$ is of the form \eqref{eq:muStable}. 
Taking a bounded and measurable function $f : [0, 1]^d \to \mathbb{R}$, we note that 
\[
    \int_{[0, 1)^d \setminus \{0\}} f(s) \mathcal{Q}_i(ds) 
    = z_i \int_{\mathbb{S}_+^d} \eta_i(d\theta) \int_0^\infty \frac{dr}{r^{1 + \alpha_i}} f(T_z(r \theta)) . 
\]
Fix $\theta \in \mathbb{S}_+^d$, let $k \in [d]$ be such that $\theta_k > 0$, and define $v = r \theta_{k} / (r \theta_k + z_k)$. 
Note that then we obtain $r = (z_k / \theta_k) v / (1 - v)$, and hence, for all $j \in [d]$, 
\[
    \frac{r \theta_j}{r \theta_j + z_j} = \frac{z_k \theta_j v}{z_k \theta_j v + z_j \theta_k (1 - v)} . 
\]
Given $\theta$ and $k$, the change of variable $v = r \theta_k / (r \theta_k + z_k)$ gives us 
\[
    \int_0^\infty \frac{dr}{r^{1 + \alpha_i}} f(T_z(r\theta)) 
    = \int_0^1 \theta_k^{\alpha_i} z_k^{-\alpha_i} v^{1 - \alpha_i} (1 - v)^{\alpha_i - 1} \frac{dv}{v^2} f\Bigl( \frac{z_k \theta_1 v}{z_k \theta_1 v + z_1 \theta_k (1 - v)}, \ldots, \frac{z_k \theta_d v}{z_k \theta_d v + z_d \theta_k (1 - v)} \Bigr) \,.
\]
Since the left hand side is independent of the choice of $k \in [d]$, it follows that the right hand side is also invariant under any choice of $k \in [d]$ such that $\theta_k > 0$. 
Thus, without loss of generality, given $\theta \in \mathbb{S}_+^d$, we may take $k_{\theta} = \min\{j \in [d] : \theta_j > 0\}$, the first index $k$ such that $\theta_k > 0$. 
Using this choice for the change of variables, we obtain
\begin{equation} \label{eq:QdefMBetas}
    \begin{split}
        \MoveEqLeft
        \int_{[0, 1)^d \setminus \{0\}} f(s) \mathcal{Q}_i(ds) \\ 
        & = z_i \int_{\mathbb{S}_+^d} \eta_i(d\theta) \theta_{k_{\theta}}^{\alpha_i} z_{k_{\theta}}^{-\alpha_i} \int_0^1 v^{1 - \alpha_i} (1 - v)^{\alpha_i - 1} \frac{dv}{v^2} f\Bigl( \frac{z_{k_{\theta}} \theta_1 v}{z_{k_{\theta}} \theta_1 v + z_1 \theta_{k_{\theta}} (1 - v)}, \ldots, \frac{z_{k_{\theta}} \theta_d v}{z_{k_{\theta}} \theta_d v + z_d \theta_{k_{\theta}} (1 - v)} \Bigr) \,.
    \end{split}
\end{equation}
In particular, $\mathcal{Q}_i$ can be viewed as a mixture, with respect to the angular measure $\eta_i$, of pushforwards of Beta-type kernels. 
The variable $v$ plays the role of a radial parameter governing the overall size of the merger event, while the direction $\theta$ and the vector $z$ determine how this mass is distributed across the different types.
Contrary to the entry-wise case, the resulting coordinates are no longer independent. 
Instead, they are coupled through the common parameter $v$, yielding genuinely multivariate merger events. 
In particular, all participating types evolve coherently along the curve determined by $\theta$ and $z$, and the proportions of blocks involved in the merger are no longer given by independent Beta-distributed random variables, but by a deterministic transformation of a single Beta-type variable.
\end{example}
This provides a natural multivariate extension of the Beta-coalescent mechanism, in which the independence between types is replaced by a dependence structure encoded by the angular measure $\eta_i$, and the parameter $z \in (0, \infty)^d$. To the best of our knowledge this has not been considered in the literature so far.

\begin{definition}[Multitype Beta coalescents] \label{def:multiBeta}
    Let $\eta = (\eta_1, \ldots, \eta_d)$ be a vector of finite measures on $\mathbb{S}_+^d$, $\alpha = (\alpha_1, \ldots, \alpha_d) \in (0, 2)^d$, and $z = (z_1, \ldots, z_d) \in (0, \infty)^d$. 
    Assume that $\eta_i(\mathbb{S}_+^d \setminus \{e_i\}) = 0$ whenever $\alpha_i \in [1, 2)$ for each $i \in [d]$. 
    We say that a multitype $\Lambda$-coalescent is a $d$-type Beta coalescent with parameters $(\alpha, \eta, z)$ if $\rho_{ii} = 0$ for all $i \in [d]$, and $\mathcal{Q}_i$ satisfies \eqref{eq:QdefMBetas} for each $i \in [d]$. 
\end{definition}

\begin{remark}
    In the one-dimensional case $d = 1$, the above definition reduces to the classical Beta coalescent. 
    Indeed, the angular measure $z_1^{1 - \alpha_1} \eta_1$ is supported on the singleton $\mathbb{S}_+^1 = \{1\}$, so that the dependence on $\theta$ disappears and the construction reduces to a Beta-type kernel driven solely by the parameter $\alpha_1$. 
    In particular, one recovers the $\mathrm{Beta}(2 - \alpha_1, \alpha_1)$ coalescent associated with $\alpha_1$-stable CSBPs.
\end{remark}

\begin{remark}
The condition that $\eta_i(\mathbb{S}_+^d \setminus \{e_i\}) = 0$ whenever $\alpha_i \in [1, 2)$ for each $i \in [d]$ ensures the integrability condition \eqref{eq:intcond_lambda}. Hence, if $\alpha_i \in [1, 2)$, large mergers can involve only lineages of a single type.
\end{remark}

In particular, when $\alpha_i = 1$ for all $i \in [d]$, we obtain a multitype version of the Bolthausen--Sznitman coalescent. 
In this regime, as in the multitype Kingman coalescent, mergers can only involve lineages of one type. 
Lineages may change types through the migration mechanism. 
This is illustrated in Figure \ref{fig:simulations}, where we present examples of different multitype coalescents. 
Specifically, in panel (a) we present a realization of a multitype Kingman coalescent, in panels (c) and (d) examples of multitype Bolthausen--Sznitman coalescents, and panel (b) depicts an example of a multitype Beta coalescent with $\eta_2 = \delta_{(1/\sqrt{3}, 1/\sqrt{3}, 1/\sqrt{3})}$. 
In the examples, it can be noted that the only coalescent where multitype mergers appear is the one in panel (b). 

Definition \ref{def:multiBeta} makes precise the genealogical structure emerging from the stable case. 
For each type $i \in [d]$, the parameter $\alpha_i$ governs the frequency and size of large merger events, as in the one-dimensional setting.
Meanwhile, the measure $\eta_i$ and the vector $z$ determine how the mergers are distributed across types, enconding the dependence structure between coordinates. 
This leads to coherent multitype merger events, in which the evolution of different types is intrinsically linked. 
Finally, we emphasize that this class of coalescents arise canonically from the homeomorphism $H_z$ applied to stable MCSBPs. 

\begin{figure}[ht]
    \centering
    \subfloat[Multitype Kingman coalescent]{\includegraphics[width=0.45\linewidth]{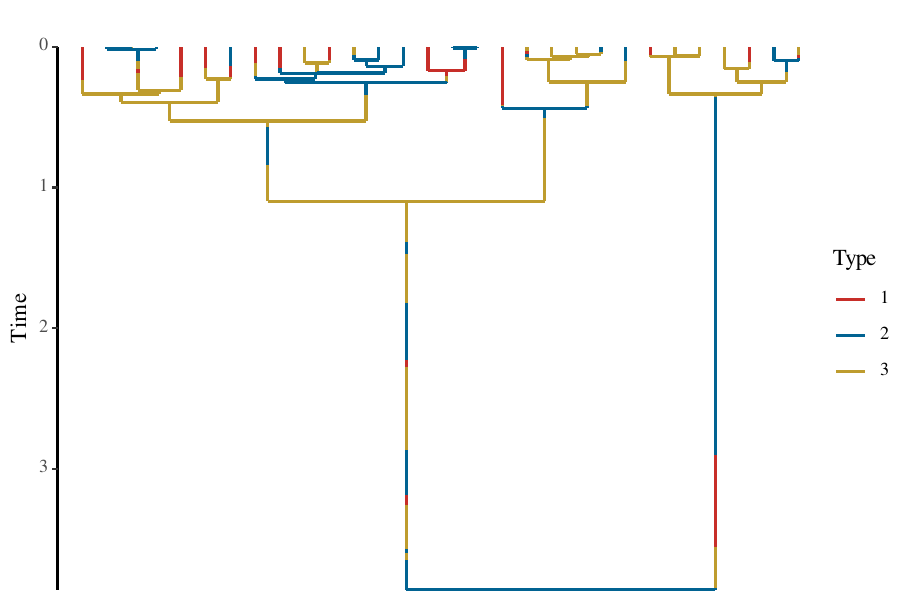}} \hfill
    \subfloat[Multitype Beta coalescent]{\includegraphics[width=0.45\linewidth]{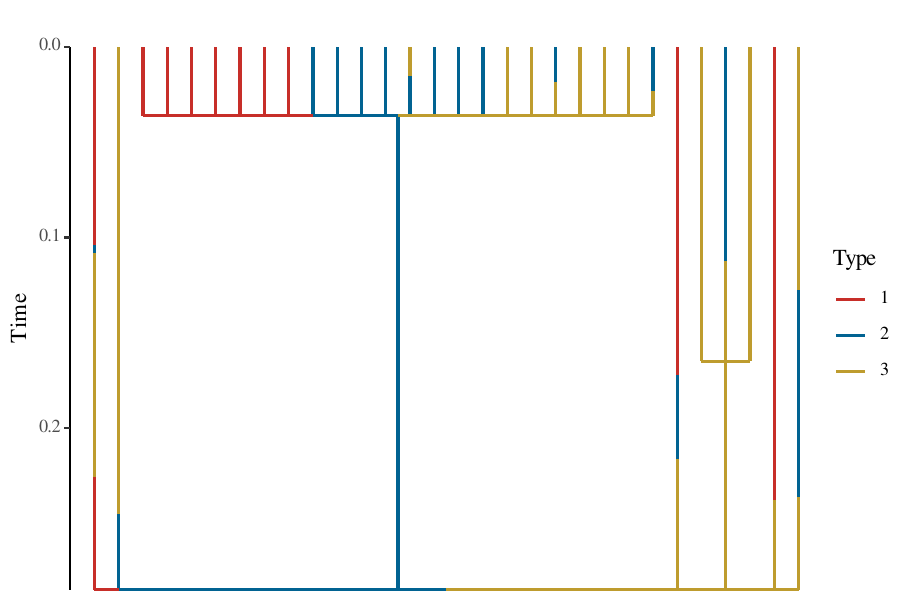}} \\
    \subfloat[Multitype Bolthausen--Sznitman coalescent]{\includegraphics[width=0.45\linewidth]{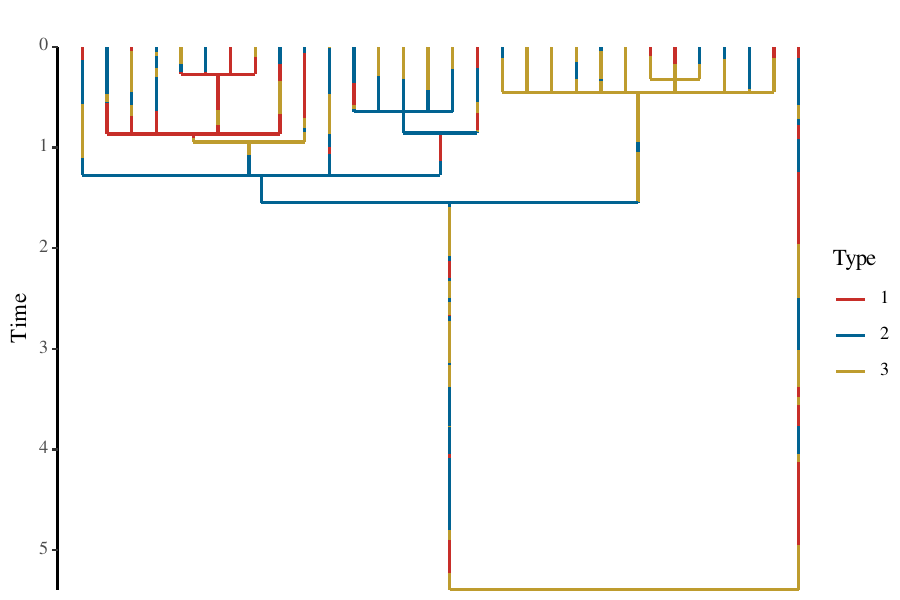}} \hfill
    \subfloat[Multitype Bolthausen--Sznitman coalescent]{\includegraphics[width=0.45\linewidth]{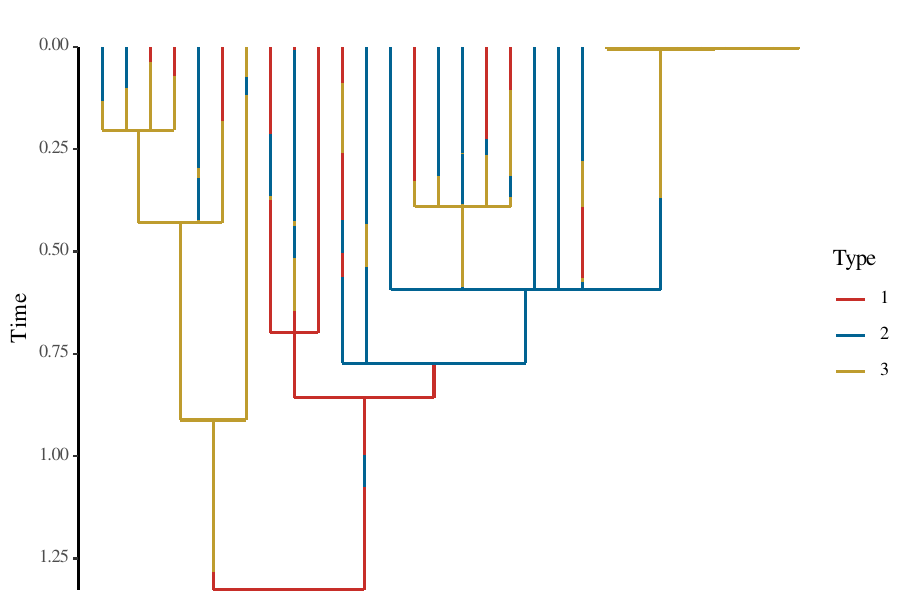}}
    \caption{%
    Examples of different multitype $\Lambda$-coalescents with three types. %
    The color of the line indicates the type of the lineage at that point in time. %
    In (a) an example of the multitype Kingman coalescent is present. 
    In (b) we present an example of a multitype Beta coalescent with parameters $\alpha = (9/10, 7/8, 1/2)$, $z = (3, 2, 1)$, $\eta_1 = \delta_{e_1}$, $\eta_2 = \delta_{(1/\sqrt{3}, 1/\sqrt{3}, 1/\sqrt{3})}$, and $\eta_3 = \delta_{e_3}$. % 
    In this case, mergers of lineages of different types occur if the resulting block is of type $2$. %
    In (c) we present an example of the multitype Bolthausen--Sznitman with parameter $z = (3, 3, 3)$. %
    Finally, in (d) we present an example of the multitype Bolthausen--Sznitman with parameter $z = (3, 6, 9)$. %
    Both in (c) and (d), large mergers involve only lineages of the same type. %{\color{red} maybe it will be important to also mention the measures $\eta_i$ in the pictures (c) and (d) in order to show why there are only large mergers of the same type?}
    }
    \label{fig:simulations}
\end{figure}

\section{Proof of the Homeomorphism}\label{sect:proof_homeo}\label{sect:homeo}

We now present the proof of Theorem \ref{theo:homeomorphismSpaces}.
The proof is structured as follows: First we define suitable topologies on the parameter spaces $\mathbf{\Psi}_a^{\mathrm{branch}}$ and $\mathbf{L}_{\mathrm{prop}}^{\mathrm{coal}}$. Then we establish a lemma stating that the function $H_z$, as defined in Theorem \ref{theo:homeomorphismSpaces} (and Remark 1), is a homeomorphism between those spaces.
Next, we discuss the spaces of multitype $\Lambda$-coalescents and MCSBPs, showing that the topologies we defined on the parameter spaces are natural, in the sense that convergence in the parameter space imply convergence of the corresponding processes.
Finally, we combine these results to complete the proof of Theorem \ref{theo:homeomorphismSpaces}.
\subsection{Homeomorphism between parameter spaces}

%To formulate it precisely, we first introduce natural topologies on the parameter spaces of the two classes of processes. We then give a bijection between the parameter space of multitype $\Lambda$-coalescents and that of MCSBPs, and show that the bijection is a fact a homeomorphism between the parameter spaces with our choice of topologies. Finally we prove that the topologies on the parameter spaces are indeed those that induce weak convergence of the processes (in the Prohorov metric on path space). Thus our final result is a genuine topolgical equivalence between the two classes of processes. As an application, we will later be able to discuss \emph{genuinely multitype beta-coalescents.}

\begin{definition}[Parameter space of multitype $\Lambda$-coalescents] \label{def:ParamCoal}
Write $\mathcal{M}^{\mathrm{coal}}$ for the space of vectors $(\mathcal{Q}_1, \ldots, \mathcal{Q}_d)$ of measures $\mathcal{Q}_i$ on $[0, 1]^d$ with no mass on $\{0\}$ that satisfy the integrability condition \eqref{eq:intcond_lambda}.
%\begin{equation}
   % \sum_{i = 1}^d \int_{[0, 1]^d \setminus \{0\}} \sum_{j = 1}^d u_j^{1 + \delta_{ij}} \mathcal{Q}_i(du) < \infty \,.
%\end{equation}
Let 
\[\mathbf{L}^{\mathrm{coal}} := \mathbb{R}_+^{d \times d} \times \mathcal{M}^{\mathrm{coal}},\]
be equipped with the product topology, where on $\mathbb{R}_+^{d \times d}$ we use the topology given by the Frobenius norm $\lVert A \rVert_{F} := \bigl( \sum_{i = 1}^d \sum_{j = 1}^d \lvert a_{ij} \rvert^2 \bigr)^{1/2}$, and on $\mathcal{M}^{\mathrm{coal}}$ we use the topology of weak convergence of the measures $u_j^{1 + \delta_{ij}} \mathcal{Q}_{i}(du)$ for each $i, j \in [d]$. Moreover, as before, consider the subspace
\begin{equation}
    \mathbf{L}_{\mathrm{prop}}^{\mathrm{coal}} :=
    \Bigl\{ (\rho, \mathcal{Q}) \in \mathbf{L}^{\mathrm{coal}} : \mathcal{Q}_i\bigl([0, 1]^d \setminus [0, 1)^d\bigr) = 0 \text{ for all } i \in [d] \Bigr\}
\end{equation}
of $\mathbf{L}^{\mathrm{coal}}$ equipped with the subspace topology.
\end{definition}

Obviously, the characterizing pair \eqref{eq:charact_1_mtlc} of a multitype $\Lambda$-coalescent takes values in $\mathbf{L}^{\mathrm{coal}}$.
Thus, in the specified topology, $(\mathcal{Q}_1^{(n)}, \ldots, \mathcal{Q}_d^{(n)}) \to (\mathcal{Q}_1, \ldots, \mathcal{Q}_d)$ as $n \to \infty$ if and only if $u_j^{1 + \delta_{ij}} \mathcal{Q}_i^{(n)}(du) \to u_j^{1 + \delta_{ij}} \mathcal{Q}_i(du)$ weakly, as $n \to \infty$, for any $i, j \in [d]$.
We will show below that convergence in $\mathbf{L}^{\mathrm{coal}}$ with this topology implies weak convergence of the corresponding multitype $\Lambda$-coalescents in the Skorokhod space of c\`adl\`ag functions.

\medskip

For the parameter space of MCSBPs we need to be a little more careful.
Recall from Section \ref{subsec:csbp} that a MCSBP can be characterized by the tuple $(B, c, \mu)$, where $\mu = (\mu_1, \ldots, \mu_d)$ is a vector of measures on $U_d$.

\begin{definition}[Parameter space of MCSBPs] \label{def:paramMCSBP}
 Let  $\mathcal{M}^{\mathrm{branch}}$ stand for the set of tuples of measures $\mu = (\mu_1, \ldots, \mu_d)$ over $U_d$ such that for any $i \in [d]$, $\mu_i$ satisfies \eqref{eq:integrability_conditions_mu}. Let
 \[\mathbf{\Psi}^{\mathrm{branch}} :=  \mathbb{R}_{(+)}^{d \times d} \times \mathbb{R}_+^d \times \mathcal{M}^{\mathrm{branch}}\]
 equipped with the product topology, where on $\mathbb{R}_{(+)}^{d \times d}$ and $\mathbb{R}_+^d$ we consider the usual topologies defined by the Frobenius and $L_2$ norm respectively, and $\mathcal{M}^{\mathrm{branch}}$ we equip with the topology of weak convergence of the measures $(1 \wedge w_j)^{1 + \delta_{ij}} \mu_i(dw)$ for each $i, j \in [d]$. Moreover, for $a \in \mathbb{R}^d$ define
\begin{equation}
    \mathbf{\Psi}_a^{\mathrm{branch}} := \bigl\{ (B, c, \mu) \in \mathbf{\Psi}^{\mathrm{branch}} : b_{ii} = a_i \text{ for all } i \in [d] \bigr\}\,,
\end{equation}
and equip it with the subspace topology.
\end{definition}

By Definition \ref{def:mcsbp}, the tuple $(B, c, \mu)$ of any MCSBP belongs to $\mathbf{\Psi}^{\mathrm{branch}}$.
% By the assumptions established in Section \ref{subsec:csbp} \textcolor{cyan}{[By Definition \ref{def:mcsbp}, given in Section \ref{subsec:csbp},]} $(B, c, \mu) \in \mathbf{\Psi}^{\mathrm{branch}}.${\color{red} How about: "By Definition \ref{def:mcsbp}, the tuple $(B, c, \mu)$ of any MCSBP belongs to $\mathbf{\Psi}^{\mathrm{branch}}$".}
Similar to the topology on $\mathcal{M}^{\mathrm{coal}}$ above, this means that $\mu^{(n)} \to \mu$ in $\mathcal{M}^{\mathrm{branch}}$, as $n \to \infty$, if and only if $(1 \wedge w_j)^{1 + \delta_{ij}} \mu_i^{(n)}(dw) \to (1 \wedge w_j)^{1 + \delta_{ij}} \mu_i(dw)$ weakly as $n \to \infty$.
In this case we will prove that convergence in this space implies the convergence of the corresponding MCSBPs in the Skorokhod space of c\`adl\`ag functions.

\medskip

%section{Proof and discussion of the homeomorphism}\label{sect:homeo}

We now state and prove that $H_z$ is a homeomorphism between the parameter spaces $\mathbf{\Psi}_a^{\mathrm{branch}}$ and $\mathbf{L}_{\mathrm{prop}}^{\mathrm{coal}}$.

%\textcolor{blue}{In the next Lemma and the subsequent proofs all the definitions from the homeomorphism were given again, which isn't necessary, I changed all the references to referring to the equations \eqref{eq:explicitHomeomorphism} and \eqref{eq:explicit_inverse}. Or did I overlook something? Also, equation with label eq:coalFromBranch has disapperared because it was the same as eq:explicitHomeomorphism \eqref{eq:explicitHomeomorphism}, so I removed all references to that. Please check if I got these right, i.e. please check all references to \eqref{eq:explicitHomeomorphism} and \eqref{eq:explicit_inverse} whether they are correct. }

\begin{lemma} \label{lem:homeomorphismHz}
    Let $a \in \mathbb{R}^d$ be fixed and $z \in (0, \infty)^d$ given.
    The function $H_z : \mathbf{\Psi}_{a}^{\mathrm{branch}} \to \mathbf{L}_{\mathrm{prop}}^{\mathrm{coal}}$ as defined in \eqref{eq:explicitHomeomorphism} is a homeomorphism between $\mathbf{\Psi}_{a}^{\mathrm{branch}}$ and $\mathbf{L}_{\mathrm{prop}}^{\mathrm{coal}}$ whose inverse $H_z^{-1}$ is given explicitly in \eqref{eq:explicit_inverse}.
\end{lemma}

\begin{proof}
    The proof is done in three steps.
    We first prove that $H_z$ is a well defined function.
    After that, we show that $H_z$ is a bijection.
    Finally, we prove that $H_z$ and $H_z^{-1}$ are continuous, i.e. that $H_z$ is a homeomorphism.

    Consider a triplet $(B, c, \mu) \in \mathbf{\Psi}_a^{\mathrm{branch}}$, and let
    $(\rho, \mathcal{Q})$ be defined by \eqref{eq:explicitHomeomorphism}.
    Then, it is clear, because $B\in\mathbb{R}_{(+)}^d$, that $\rho \in \mathbb{R}_+^{d \times d}$.
    On the other hand, for each $i, j \in [d]$, by definition of $\mathcal{Q}_i$, we deduce
    \begin{align}
        \int_{[0, 1]^d} u_j^{1 + \delta_{ij}} \mathcal{Q}_i(du)
        & = z_i \int_{U_d} \biggl( \frac{w_j}{z_j + w_j} \biggr)^{1 + \delta_{ij}} \mu_i(dw) \notag\\
        & \leq \frac{z_i}{(1 \wedge z_j)^{1+\delta_{ij}}} \int_{U_d} (1 \wedge w_j)^{1 + \delta_{ij}} \mu_i(dw) \,.
    \end{align}
    By condition \eqref{eq:integrability_conditions_mu} it follows that $\mathcal{Q} \in \mathcal{M}^{\mathrm{coal}}$.
    Moreover, it is straightforward that $\mathcal{Q}_i([0, 1]^d \setminus [0, 1)^d) = 0$ for all $i \in [d]$.
    As a consequence $(\rho, \mathcal{Q}) \in \mathbf{L}_{\mathrm{prop}}^{\mathrm{coal}}$, so $H_z$ is well defined.

    Let us prove that $H_z$ is one-to-one.
    Assume that $H_z(B^{(1)}, c^{(1)}, \mu^{(1)}) = H_z(B^{(2)}, c^{(2)}, \mu^{(2)})$.
    By \eqref{eq:explicitHomeomorphism} it is immediate that $B^{(1)} = B^{(2)}$ and $c^{(1)} = c^{(2)}$.
    Now fix a measurable function $f : U_d \to \mathbb{R}_+$.
    By definition of $\mathbf{T}_z$, for every $i \in [d]$ we get
    \begin{align}
        \int_{U_d} f(w) \mu_i^{(1)}(dw)
        & = \int_{[0, 1)^d \setminus \{0\}} f\biggl( \frac{z_1 u_1}{1 - u_1}, \ldots, \frac{z_d u_d}{1 - u_d} \biggr) \mathbf{T}_z \mu_i^{(1)}(du) \notag\\
        & = \frac{1}{z_i} \int_{[0, 1)^d \setminus \{0\}} f\biggl( \frac{z_1 u_1}{1 - u_1}, \ldots, \frac{z_d u_d}{1 - u_d} \biggr) \mathcal{Q}_i^{(1)}(du) \notag\\
        & = \frac{1}{z_i} \int_{[0, 1)^d \setminus \{0\}} f\biggl( \frac{z_1 u_1}{1 - u_1}, \ldots, \frac{z_d u_d}{1 - u_d} \biggr) \mathcal{Q}_i^{(2)}(du) \notag\\
        & = \int_{[0, 1)^d \setminus \{0\}} f\biggl( \frac{z_1 u_1}{1 - u_1}, \ldots, \frac{z_d u_d}{1 - u_d} \biggr) \mathbf{T}_z \mu_i^{(2)}(du)
        = \int_{U_d} f(w) \mu_i^{(2)}(dw)\,.
    \end{align}
    Subsequently $\mu_i^{(1)} = \mu_i^{(2)}$ for every $i \in [d]$.
    This proves that $H_z$ is injective.

    To show that $H_z$ is surjective consider $(\rho', \mathcal{Q}') \in \mathbf{L}_{\mathrm{prop}}^{\mathrm{coal}}$.
    Let us define $(B', c', \mu')$ by \eqref{eq:explicit_inverse}.
    %setting, for $i \in [d]$ and $j \in [d] \setminus \{i\}$,
   % \begin{equation} \label{eq:inverseHz}
    %    c_i' = \frac{z_i}{2} \rho_{ii}, \
       % b_{ii}' = a_i, \
       % b_{ij}' = \frac{z_i}{z_j} \rho_{ji} \text{ and }
      %  \mu_i(dw)' = \frac{1}{z_i} \mathbf{T}_z^{-1} \mathcal{Q}_i(dw)\,,
   % \end{equation}
   % where $\mathbf{T}_z^{-1} \mathcal{Q}_i$ is the pushforward measure of $\mathcal{Q}_i$ under $T_z^{-1}$.
    We need to prove that $(B', c', \mu') \in \mathbf{\Psi}_a^{\mathrm{branch}}$.
    By \eqref{eq:explicit_inverse}, it is obvious that $B' \in \mathbb{R}_{(+)}^{d \times d}$ with $b_{ii} = a_i$ for $i \in [d]$, and that $c \in \mathbb{R}_+^d$.
    Thus, it suffices to prove that $\mu' \in \mathcal{M}^{\mathrm{branch}}$, for which we will prove that the integrability condition \eqref{eq:integrability_conditions_mu} holds for $\mu'.$
    %\begin{equation} \label{eq:integrability_conditions_muprime}
       % \int_{U_d} \sum_{j = 1}^d (1 \wedge w_j)^{1 + \delta_{ij}} \mu_i'(dw) < \infty
        %\quad\text{holds for each $i \in [d]$.}
   % \end{equation}
    This in turn follows from the inequality
    $1 \wedge ( L v / (1 - v)) \leq (1 + L) v$ valid for any $L > 0$ and
    $v \in [0, 1)$.
    Certainly, by applying this inequality, we deduce
    \begin{align*}
        \int_{[0, 1)^d \setminus \{0\}} \sum_{j = 1}^d \biggl( 1 \wedge \biggl( \frac{z_j u_j}{1 - u_j} \biggr) \biggr)^{1 + \delta_{ij}} \mathcal{Q}_i(du)
        & \leq \sum_{j = 1}^d (1 + z_j)^{1 + \delta_{ij}} \int_{[0, 1)^d \setminus \{0\}}
        u_j^{1 + \delta_{ij}} \mathcal{Q}_i(du) < \infty \,.
    \end{align*}
    Then \eqref{eq:integrability_conditions_mu} is fulfilled because of the
    equality
    \begin{equation}
        \int_{U_d} \sum_{j = 1}^d (1 \wedge w_j)^{1 + \delta_{ij}} \mu_i'(dw)
        = \frac{1}{z_i} \int_{[0, 1)^d \setminus \{0\}} \sum_{j = 1}^d \biggl(
            1 \wedge \biggl( \frac{z_j u_j}{1 - u_j} \biggr)
        \biggr)^{1 + \delta_{ij}} \mathcal{Q}_i(du) .
    \end{equation}
    In light of this $H_z$ is onto and therefore a bijection, with inverse
        $H_z^{-1}$.

    Let us prove that $H_z$ is continuous.
    Suppose that $(B^{(n)}, c^{(n)}, \mu^{(n)}) \to (B, c, \mu)$ as $n \to \infty$ in $\mathbf{\Psi}_a^{\mathrm{branch}}$.
    From \eqref{eq:explicitHomeomorphism}, it is obvious that $\rho^{(n)} \to \rho$ as $n \to \infty$, so if suffices to prove that $\mathcal{Q}^{(n)} \to \mathcal{Q}$ in $\mathcal{M}^{\mathrm{coal}}$.
    Consider $f \in C_b([0, 1]^d)$ and $i, j \in [d]$ given.
    Then
    \begin{align*} 
        \MoveEqLeft 
        \int_{[0, 1]^d} f(u) u_j^{1 + \delta_{ij}} \mathcal{Q}_i^{(n)}(du) \\
        & = z_i \int_{U_d} \biggl( \frac{w_j}{z_j + w_j} \biggr)^{1 + \delta_{ij}} f \circ T_z(w) \mu_i^{(n)}(dw) \\
        & = z_i \int_{U_d} \biggl( \frac{1}{z_j + w_j} 1_{\{w_j < 1\}} + \frac{w_j}{z_j + w_j} 1_{\{w_j \geq 1\}} \biggr)^{1 + \delta_{ij}} f \circ T_z (w) (1 \wedge w_j)^{1 + \delta_{ij}} \mu_i^{(n)}(dw) \,.
    \end{align*}
    Noting that
    \[
        w \mapsto \biggl( \frac{1}{z_j + w_j} 1_{\{w_j < 1\}} + \frac{w_j}{z_j + w_j} 1_{\{w_j \geq 1\}} \biggr)^{1 + \delta_{ij}} f \circ T_z (w)
    \]
    defines a bounded continuous function from $U_d$ to $\mathbb{R}$, the assumption $\mu^{(n)} \to \mu$ as $n \to \infty$ in $\mathcal{M}^{\mathrm{branch}}$ implies
    \[
        \int_{[0, 1]^d} f(u) u_j^{1 + \delta_{ij}} \mathcal{Q}_i^{(n)}(du)
        \to \int_{[0, 1]^d} f(u) u_j^{1 + \delta_{ij}} \mathcal{Q}_i(du) \,.
    \]
    As $i, j \in [d]$ were arbitrary, it follows that $\mathcal{Q}^{(n)} \to \mathcal{Q}$ as $n \to \infty$ in $\mathcal{M}^{\mathrm{coal}}$.
    Thus, $H_z$ is continuous.

    To prove that $H_z^{-1}$ is continuous suppose that $(\rho^{\prime (n)}, \mathcal{Q}^{\prime (n)}) \to (\rho', \mathcal{Q}')$ as $n \to \infty$ in $\mathbf{L}_{\mathrm{prop}}^{\mathrm{coal}}$.
    By \eqref{eq:explicit_inverse}, it is immediate that $B^{\prime (n)} \to B'$ and $c^{\prime (n)} \to c'$ as $n \to \infty$.
    We need to show that $\mu^{\prime (n)} \to \mu'$ as $n \to \infty$ in $\mathcal{M}^{\mathrm{branch}}$.
    Now consider $f \in C_b(U_d)$ and $i, j \in [d]$ fixed.
    Then
    \begin{align*}
        \MoveEqLeft
        \int_{U_d} f(w) (1 \wedge w_j)^{1 + \delta_{ij}} \mu_i^{\prime (n)}(dw) \\
        & = \frac{1}{z_i} \int_{[0, 1]^d} \biggl( \frac{z_j}{1 - u_j} 1_{\{u_j < 1 / (1 + z_j)\}} + \frac{1}{u_j} 1_{\{u_j \geq 1 / (1 + z_j)\}} \biggr)^{1 + \delta_{ij}} f \circ T_z^{-1}(u) u_j^{1 + \delta_{ij}} \mathcal{Q}_i^{\prime (n)} (du) \,.
    \end{align*}
    To obtain the convergence
    \begin{equation} \label{eq:weakConvergenceTheo}
        \int_{U_d} f(w) (1 \wedge w_j)^{1 + \delta_{ij}} \mu_i^{\prime (n)}(dw)
        \to \int_{U_d} f(w) (1 \wedge w_j)^{1 + \delta_{ij}} \mu_i'(dw)
        \quad\text{ as } n \to \infty
    \end{equation}
    note that the mapping
    \[
        u \mapsto \biggl( \frac{z_j}{1 - u_j} 1_{\{u_j < 1 / (1 + z_j)\}} + \frac{1}{u_j} 1_{\{u_j \geq 1 / (1 + z_j)\}} \biggr)^{1 + \delta_{ij}} f \circ T_z^{-1}(u)
    \]
    from $[0, 1]^d$ to $\mathbb{R}$ is almost everywhere, with respect to $\mathcal{Q}_i'$, continuous and bounded. This last assertion follows because $\mathcal{Q}_i'([0,1]^d \setminus [0,1)^d)=0$.
    Thus, the convergence in \eqref{eq:weakConvergenceTheo} is attained by applying the assumption
    $u_j^{1 + \delta_{ij}} \mathcal{Q}_i^{\prime (n)} \to u_j^{1 + \delta_{ij}} \mathcal{Q}_i'$ as $n \to \infty$ weakly.
    This proves that $(1 \wedge w_j)^{1 + \delta_{ij}} \mu_i^{\prime (n)} \to (1 \wedge w_j)^{1 + \delta_{ij}} \mu_i'$ as $n \to \infty$ weakly, entailing the continuity of $H_z^{-1}$.
\end{proof}

\subsection{The topological space of multitype \texorpdfstring{$\mathbf{\Lambda}$}{L}-coalescents}
\label{subsec:multiLambdaCoal}

We now show that the topology on the parameter space of the multitype $\Lambda$-coalescents corresponds to the topology of weak convergence on the path space of these processes. So far we have only worked with the block counting process of a multitype
$\Lambda$-coalescent.
We now turn to study the process itself, starting by introducing its state space, the $d$-type partitions of $\mathbb{N}$.

\begin{definition}\label{def: partition-dtype}
    Let $S$ be a countable set. 
    We say that $\bm{\pi} = (\pi, \mathsf{t}_{\bm{\pi}})$ is a $d$-type partition of $S$ if $\pi$ is a partition of $S$, and $\mathsf{t}_{\bm{\pi}} : \pi \to [d]$ is a function that assigns types to the blocks of $\pi$. 
    Let $\mathbf{P}_d(S)$ be the collection of all $d$-type partitions of $S$. 
    If $S'$ is a subset of $S$, we put $\bm{\pi}\vert_{S'} = (\pi\vert_{S'}, \mathsf{t}_{\bm{\pi}\vert_{S'}})$ for the restriction of $\bm{\pi}$ to $S'$. 
    Namely, $\pi\vert_{S'}$ is composed of elements of the form $B \cap S'$ such that $B \in \pi$ and $B \cap S' \neq \varnothing$, while for such elements, we define 
    \[
        \mathsf{t}_{\bm{\pi}\vert_{S'}}(B \cap S') = \mathsf{t}_{\bm{\pi}}(B) \,.
    \]
    % We say that $\bm{\pi} = \{(\pi_i, \mathfrak{t}_i) : i \in \mathbb{N}\}$
    % is a $d$-type partition of $\mathbb{N}$ if
    % $\pi = \{\pi_i : i \in \mathbb{N}\}$ is a partition of $\mathbb{N}$
    % and for each $i \in \mathbb{N}$, $\mathfrak{t}_i \in [d]_0$ is the type
    % of block $\pi_i$, where $\mathfrak{t}_i = 0$ if and only if $\pi_i = \varnothing$.
    % We assume that the elements of $\pi$ are arranged by their minimal elements;
    % i.e. $\min \pi_i \leq \min \pi_j$ for $i < j$, with the convention
    % $\min \varnothing = \infty$, where there is strict inequality if
    % $\pi_i \neq \varnothing$.
    % Let $\mathbf{P}_d$ be the collection of all $d$-type partitions of
    % $\mathbb{N}$.
    % For $m \in \mathbb{N}$ and
    % $\bm{\pi} \in \mathbf{P}_d$ we put
    % $\bm{\pi}\vert_{[m]} := \{(\pi_i \vert_{[m]},
    % \mathfrak{t}_i 1_{\{\pi_i\vert_{[m]} \neq \varnothing\}}) : i \in \mathbb{N}\}$
    % for the restriction of $\bm{\pi}$ to $[m]$,
    % where $\pi_i\vert_{[m]} := \pi_i \cap [m]$.
\end{definition}

Multitype $\Lambda$-coalescents can be defined as partition valued processes with values in $\mathbf{P}_d(\mathbb{N})$, which we will denote by $\Pi$, in this framework.
By Theorem 1.5 in \cite{johnstonMultitypeLcoalescents2022a}, the restriction $\Pi\vert_{[M]}$ of $\Pi$ to $[M]$ is a Markov chain with a finite state space.
This chain can jump from the state $\bm{\pi}$ with $b_j$ blocks of type $j$, for $j \in [d]$, to the state $\bm{\pi}'$ if and only if there exist $k_j \in [b_j]_0$ for each $j \in [d]$ such that $\bm{\pi}'$ can be constructed by merging $k_1$ blocks of type 1, $k_2$ of type 2, etc. and assigning the resulting block the type $i \in [d]$.
In this case, for $k \in [b]_0 \setminus \{0, e_i\}$, the jump rate from $\bm{\pi}$ to $\bm{\pi}'$ is given by
\begin{equation}\label{rate_partvalued}
    \tilde{\lambda}_{b, k}^i
    = \rho_{ii} 1_{\{k = 2 e_i\}}
    + \sum_{j \in [d] \setminus \{i\}} \rho_{ij} 1_{\{k = e_j\}}
    + \int_{[0, 1]^d \setminus \{0\}} u^k (1 - u)^{n - k} \mathcal{Q}_i(du) \,,
\end{equation}
with $(\rho, \mathcal{Q}) \in \mathbf{L}^{\mathrm{coal}}$ given.

% Note that given a $d$-type partition $\bm{\pi} \in \mathbf{P}_d(\mathbb{N})$ we can
% define a function $\vartheta_{\bm{\pi}} : \mathbb{N} \to [d]$
% by
% \(
%     \vartheta_{\bm{\pi}}(i) := \mathfrak{t}_{\mathfrak{j}(i)}
% \),
% where $\mathfrak{j}(i) \in \mathbb{N}$ is the index such that
% $i \in \pi_{\mathfrak{j}(i)}$.
We now define a distance $d_{\mathbf{P}_d}(S)$ on $\mathbf{P}_d(S)$, for $S = \mathbb{N}$ or $S = [M]$ with $M \in \mathbb{N}$. 
We put 
\[
    d_{\mathbf{P}_d(S)}(\bm{\pi}, \bm{\pi}')
    = \Bigl( \max \bigl\{m \in S : \bm{\pi}\vert_{[m]} = \bm{\pi}'\vert_{[m]} \bigr\}\Bigr)^{-1} ,
\]
%{\color{blue} Above in the definition of the distance, should we consider the type? i.e.
%\[
 %   d_{\mathbf{P}_d(S)}(\bm{\pi}, \bm{\pi}')
%    = \Bigl( \max \bigl\{m \in S : \bm{\pi}\vert_{[m]} = \bm{\pi}'\vert_{[m]} \bigr\}\Bigr)^{-1} ,
%\]?
%also should we define $\mathsf{t}_{\bm{\pi}|M}$?} \textcolor{cyan}{[$\mathsf{t}_{\bm{\pi}\vert_{[M]}}$ is defined in Definition\ref{def: partition-dtype}.]}
where $\bm{\pi}, \bm{\pi}' \in \mathbf{P}_d(S)$, and $\pi\vert_{[m]}$ denotes the restriction of $\pi$ to $[m]$, which is an element of $\mathbf{P}_d([m])$.
This distance takes into account both the partition structure and the types of
the blocks of $d$-type partitions.

For $S = \mathbb{N}$, consider the mapping $\bm{\pi} \mapsto (\bm{\pi}\vert_{[1]}, \bm{\pi}\vert_{[2]}, \ldots)$ from $\mathbf{P}_d(\mathbb{N})$ to $\mathbf{P}_d([1]) \times \mathbf{P}_d([2]) \times \cdots$. 
If we equip the latter space with the product of discrete topologies, and the image of the mapping with the relative topology, then the topology induced by $d_{\mathbf{P}_d(\mathbb{N})}$ on $\mathbf{P}_d(\mathbb{N})$ coincides with the one induced by $\bm{\pi} \mapsto (\bm{\pi}\vert_{[1]}, \bm{\pi}\vert_{[2]}, \ldots)$. 
This implies that $\mathbf{P}_d(\mathbb{N})$ is Polish.

\begin{proposition} \label{prop:conv_mtlc_measures}
    Let $\{\Pi^{(n)} : n \in \mathbb{N}\}$ be a sequence of multitype
    $\Lambda$-coalescents such that for each $n \in \mathbb{N}$,
    $\Pi^{(n)}$ is characterized by the pair $(\rho^{(n)}, \mathcal{Q}^{(n)}) \in \mathbf{L}^{\mathrm{coal}}$.
    Furthermore let $\Pi$ be a multitype $\Lambda$-coalescent characterized by
    $(\rho, \mathcal{Q}) \in \mathbf{L}^{\mathrm{coal}}$.
    If $(\rho^{(n)}, \mathcal{Q}^{(n)}) \to (\rho, \mathcal{Q})$ as $n \to \infty$ in $\mathbf{L}^{\mathrm{coal}}$, then
    \begin{equation}
        \Pi^{(n)} \Rightarrow \Pi \text{ as $n \to \infty$}
    \end{equation}
    weakly in $D([0, T], (\mathbf{P}_d(\mathbb{N}), d_{\mathbf{P}_d(\mathbb{N})}))$ for all $T > 0$.
\end{proposition}

\begin{proof}
    Let $\Pi\vert_{[M]}$ ($\Pi^{(n)}\vert_{[M]}$ resp.) be the restriction of $\Pi$ ($\Pi^{(n)}$ resp.) to $[M]$.
    By the definition of the topology imposed on $\mathbf{P}_d(\mathbb{N})$, it suffices to prove that for every $M \in \mathbb{N}$,
    \begin{equation}
        \Pi^{(n)}\vert_{[M]} \Rightarrow \mathbf{\Pi}\vert_{[M]} \text{ as $n \to \infty$}
    \end{equation}
    weakly in $D([0, T], (\mathbf{P}_d([M]), d_{\mathbf{P}_d([M])}))$ for all $T > 0$.

    Putting $\ell := \min\{j \in [d] \setminus \{i\} : k_j > 0\}$ we note that the integral in the rates given in \eqref{rate_partvalued}
    \begin{equation}
        \int_{[0, 1]^d \setminus \{0\}} u^k (1 - u)^{n - k} \mathcal{Q}_i(du)
        = \begin{dcases*}
            \int_{[0, 1]^d \setminus \{0\}} u^{k - 2 e_i} (1 - u)^{n - k} u_i^2 \mathcal{Q}_{i}(du) & if $k_i \geq 2$, \\
            \int_{[0, 1]^d \setminus \{0\}} u^{k - e_{\ell}} (1 - u)^{n - k} u_{\ell} \mathcal{Q}_i(du) & if $k_i < 2$.
        \end{dcases*}
    \end{equation}
    Using this last expression, it is clear that the assumption of convergence $(\rho^{(n)}, \mathcal{Q}^{(n)}) \to (\rho, \mathcal{Q})$ as $n \to \infty$ implies the convergence of the transition rates of the processes $\{\Pi^{(n)}\vert_{[M]} : n \in \mathbb{N}\}$ to those of the process $\Pi\vert_{[M]}$, as $n \to \infty$, for every $M \in \mathbb{N}$, i.e.
    \begin{equation}
        \lim_{n \to \infty} \tilde{\lambda}_{b, k}^{i, (n)}
        = \tilde{\lambda}_{b, k}^i.
    \end{equation}

    By the fact that $\Pi^{(n)}\vert_{[M]}$ and $\Pi\vert_{[M]}$ are continuous-time Markov chains with finite state space, the convergence of the transitions imply the weak convergence
    \begin{equation} \label{eq:weakConvM}
        \Pi^{(n)}\vert_{[M]} \Rightarrow \Pi\vert_{[M]},
        \quad\text{ as $n \to \infty$,}
    \end{equation}
    in the space of c\`adl\`ag paths from $[0, T]$ to
    $(\mathbf{P}_d([M]), d_{\mathbf{P}_d}([M])$ endowed with the Skorokhod topology
    for any $T > 0$.

    We now turn to prove the weak convergence $\Pi^{(n)} \Rightarrow \Pi$ as $n \to \infty$. 
    Fix $\varepsilon > 0$. 
    In view of \eqref{eq:weakConvM}, for each $M \in \mathbb{N}$, there exists a compact subset $K_M \subset D([0, T], (\mathbf{P}_d([M]), d_{\mathbf{P}_d([M])}))$ such that 
    \[  
        \sup_{n \in \mathbb{N}} \mathbb{P}( \Pi^{(n)}\vert_{[M]} \not\in K_M ) \leq \frac{\varepsilon}{2^M} \,.
    \]
   % {\color{red}
    %\[  
     %   \sup_{n \in \mathbb{N}} \mathbb{P}( \Pi^{(n)}\vert_{[M]} \in K_M ) \geq 1-\frac{\varepsilon}{2^M} \,.
    %\]?
    %}
    Observe that by the remark stated before Proposition \ref{prop:conv_mtlc_measures}, $D([0, T], (\mathbf{P}_d(S), d_{\mathbf{P}_d(S)}))$ is Polish for $S = \mathbb{N}$ or $S = [M]$ with $M \in \mathbb{N}$. 
    Now, for each $M \in \mathbb{N}$, define the continuous function $R_M : \mathbf{P}_d(\mathbb{N}) \to \mathbb{P}_d([M])$ by $R_M(\Pi) := \Pi\vert_{[M]}$. 
    By problem 13 in Chapter 3 in \cite{ethierMarkovProcessesCharacterization2005}, the mapping $\tilde{R}_M : D([0, T], (\mathbf{P}_d(\mathbb{N}), d_{\mathbf{P}_d(\mathbb{N})})) \to D([0, T], (\mathbf{P}_d([M]), d_{\mathbf{P}_d([M])}))$ given by $x \mapsto R_M \circ x$ is continuous. 
    Thus, $K := \bigcap_{M \in \mathbb{N}} \tilde{R}_M^{-1}(K_M)$ is a compact subset of $\mathbf{P}_d(\mathbb{N})$ such that 
    \[
        \sup_{n \in \mathbb{N}} \mathbb{P}(\Pi^{(n)} \not\in K) \leq \varepsilon \,,
    \]
    meaning that $(\Pi^{(n)})_{n \in \mathbb{N}}$ is tight. 
    Thus, we can take a subsequence $(\Pi^{(n_k)})_{k \in \mathbb{N}}$ and a random element $\Pi^{(\infty)}$ in $D([0, T], (\mathbf{P}_d(\mathbb{N}), d_{\mathbf{P}_d(\mathbb{N})})$ such that $\Pi^{(n_k)} \Rightarrow \Pi^{(\infty)}$ as $k \to \infty$. 
    By the continuous mapping theorem and \eqref{eq:weakConvM} we deduce that $\Pi\vert_{[M]}$ and $\Pi^{(\infty)}\vert_{[M]}$ have the same law for every $M \in \mathbb{N}$. 

    We claim that the set of functions 
    \[
        \mathcal{E} = \bigcup_{M \in \mathbb{N}} \{ f \circ \tilde{R}_M : f \in C_b( D([0, T], (\mathbf{P}_d([M]), d_{\mathbf{P}_d([M])})) ) \}
    \]
    is an algebra that separates points. 
    If we prove this, it will follow that $\Pi$ has the same law as $\Pi^{(\infty)}$ by Theorem 4.5 of Chapter 3 in \cite{ethierMarkovProcessesCharacterization2005}. 
    In turn, this implies that $\Pi^{(n)} \Rightarrow \Pi$ as $n \to \infty$, which concludes the proof. 
    
    Let us prove the claim. 
    To see that it separates points, note that if $x, y \in D([0, T], (\mathbf{P}_d(\mathbb{N}), d_{\mathbf{P}_d(\mathbb{N})}))$ are such that $x \neq y$, there exists $M \in \mathbb{M}$ such that $\tilde{R}_M(x) \neq \tilde{R}_M(y)$. %{\color{red} Here $x, y \in (\mathbf{P}_d(\mathbb{N}), d_{\mathbf{P}_d(\mathbb{N})}))$ but $\tilde{R}_M$ is defined on the space $D([0, T], (\mathbf{P}_d(\mathbb{N}), d_{\mathbf{P}_d(\mathbb{N})})$ so I do not know if you mean $R_M(x) \neq R_M(y)$, or $x,y\in D([0, T], (\mathbf{P}_d(\mathbb{N}), d_{\mathbf{P}_d(\mathbb{N})})$?}
    % Being $\mathbf{P}_d([M])$ Polish, it follows that $D([0, T], (\mathbf{P}_d([M]), d_{\mathbf{P}_d([M])}))$ is also Polish. 
    Since $D([0, T], (\mathbf{P}_d([M]), d_{\mathbf{P}_d([M])}))$ is metric, it is a normal space, see Section 15 of Chapter 5 in \cite{Willard1970} for the precise definition. 
    %{\color{red} For what is the fact that $D([0, T], (\mathbf{P}_d([M]), d_{\mathbf{P}_d([M])}))$ is Polish used? Because I guess for Urysohn's lemma we only need that it is a metric space? Also you already stated above that $D([0, T], (\mathbf{P}_d([M]), d_{\mathbf{P}_d([M])}))$ is Polish.} 
    %\textcolor{cyan}{I added the reference for the version of Urysohn's lemma I used here. It requires normality of the space, which holds for Polish spaces.}
    Hence, Urysohn's lemma implies the existence of a bounded continuous function $f \in C_b( D([0, T], (\mathbf{P}_d([M]), d_{\mathbf{P}_d([M])}))$ such that $f(\tilde{R}_M(x)) \neq f(\tilde{R}_M(y))$. 
    Thereby, $\mathcal{E}$ separates points. 
    In order to see that it is an algebra, consider $M, L \in \mathbb{N}$, $f \in C_b( D([0, T], (\mathbf{P}_d([M]), d_{\mathbf{P}_d([M])})) )$ and $g \in C_b( D([0, T], (\mathbf{P}_d([L]), d_{\mathbf{P}_d([L])})) )$.
    Without loss of generality suppose that $M \leq L$. 
    Observe that the function $R_{L, M} : \mathbf{P}_d([L]) \to \mathbf{P}_d([M])$ given by $\pi \mapsto \pi\vert_{[M]}$ is continuous, so the extension 
    \[ 
        \tilde{R}_{L, M} : x \in D([0, T], (\mathbf{P}_d([L]), d_{\mathbf{P}_d([L])})) \mapsto R_{L, M} \circ x \in  D([0, T], (\mathbf{P}_d([M]), d_{\mathbf{P}_d([M])}))
    \] 
    is also continuous. 
    Therefore, $f \circ \tilde{R}_{L, M} \in C_b(D([0, T], (\mathbf{P}_d([L]), d_{\mathbf{P}_d([L])})))$ and 
    \[
        f \circ \tilde{R}_{M} = (f \circ \tilde{R}_{L, M}) \circ \tilde{R}_L .
    \]
    Hence, 
    \[
        (f \circ \tilde{R}_M) + (g \circ \tilde{R}_L)  
        = ( (f \circ \tilde{R}_{L, M}) + g ) \circ \tilde{R}_L \in \mathcal{E}
        \quad\text{and}\quad 
        (f \circ \tilde{R}_M)(g \circ \tilde{R}_L) 
        = ( (f \circ \tilde{R}_{L, M}) g ) \circ \tilde{R}_L \in \mathcal{E} .
    \]
    This proves that $\mathcal{E}$ is an algebra, concluding the proof. 
\end{proof}

\subsection{The topological space of MCSBPs}\label{subsec:topologicalMCSBPs}
We now turn to the equivalent result for MCSBPs.
%Recall from Definition \ref{def:paramMCSBP} the topology on the parameter space $\mathbf{\Psi}^{\mathrm{branch}}$ of MCSBPs.

\begin{proposition} \label{prop:convergence_mtbp_triplet}
    Let $\{X^{(n)} : n \in \mathbb{N}\}$ be a sequence of $d$-type CSBPs such that for each
    $n$ the characteristic triplet of $X^{(n)}$ is $(B^{(n)}, c^{(n)}, \mu^{(n)}) \in \mathbf{\Psi}^{\mathrm{branch}}$.
    In addition, consider a $d$-type CSBP $X$ with characteristic triplet
    $(B, c, \mu) \in \mathbf{\Psi}^{\mathrm{branch}}$.
    If
    \begin{equation} \label{eq:triplet_convergence}
        (B^{(n)}, c^{(n)}, \mu^{(n)}) \to (B, c, \mu) \text{ as $n \to \infty$ in $\mathbf{\Psi}^{\mathrm{branch}}$,}
    \end{equation}
    then $X^{(n)} \Rightarrow X$ as $n \to \infty$, weakly on the space of c\`adl\`ag
    paths from $\mathbb{R}_+$ to $[0, \infty]^d$ with respect to the Skorokhod
    topology if the branching mechanism of $X$ is nonexplosive, and with
    respect to the uniform Skorokhod topology if the branching mechanism is
    explosive.
\end{proposition}

In the proof of the previous result we will consider $d$ auxiliary L\'evy processes in $\mathbb{R}^d$ from which the MCSBP can be obtained through a multiparameter Lamperti time-change.
From \cite{duffieAffineProcessesApplications2003}, a MCSBP $X$ with characteristic triplet $(A, \gamma, \nu) \in \mathbf{\Psi}^{\mathrm{branch}}$ can be characterized by the Laplace transform of its semigroup, which satisfies that for each $x, \lambda \in \mathbb{R}_+^d$,
\[
    \log \mathbb{E}_x[ e^{-\langle X(t), \lambda \rangle} ]
    = -\langle u(t, \lambda), x\rangle,% - \Psi_t(\lambda),
\]
where $u(t, \lambda) = (u_1(t, \lambda), \ldots, u_d(t, \lambda))$ satisfy a system of Riccati equations.
Said Riccati equations are given, for $i \in [d]$, by
\[
    \partial_t u_i(t, \lambda)
     = - \varphi_i( u(t, \lambda) ), \qquad \text{$\lambda\geq0$},
\]
with 
\[
    \varphi_i(\lambda)
    = \gamma_i \lambda_i^2 - \langle A e_i, \lambda\rangle
    + \int_{U_d} \bigl(e^{-\langle \lambda, w \rangle} - 1 + \lambda_i (1 \wedge w_i) \bigr)
    \nu_i(dw), \qquad \text{$\lambda\geq0$}.
\]
For each $i \in [d]$ we will consider a L\'evy process $\mathcal{Y}_i$ on $\mathbb{R}^d$ with characteristic exponent
\begin{equation} \label{eq:char_exponent_levy}
    \varphi_i(-\mathrm{i} \lambda) = - \gamma_i \lambda_i^2 + \mathrm{i} \langle A e_i, \lambda \rangle + \int_{U_d} \bigl( e^{\mathrm{i} \langle \lambda, z\rangle} - 1 - \mathrm{i} \lambda_i (1 \wedge z_i) \bigr) \nu_i(dz), \quad \lambda \in \mathbb{R}^d,
\end{equation}
or, equivalently, characteristic triplet $(A e_i, \Gamma_i, \nu_i)$, where $\Gamma_i = (\gamma_i \delta_{ij} \delta_{ik})_{j, k \in [d]}$, which is different from the characteristic triplet of the MCSBP $X$.
Then, according to Theorem 1 in \cite{caballeroAffineProcessesMathbb2017}, the $d$-type CSBP $X$ with characeristic triplet $(A, \gamma, \nu)$ can be seen as a multiparameter time-change of the $d$ Lévy processes $\mathcal{Y}_i, \ldots, \mathcal{Y}_i$.

An important property that we will use in the proof of Proposition \ref{prop:convergence_mtbp_triplet} is that the weak convergence of the auxiliary L\'evy processes implies the weak convergence of the associated MCSBPs.
Thus, before proving Proposition \ref{prop:convergence_mtbp_triplet} we will prove an auxiliary lemma, which essentially states that the convergence of the parameters in $\mathbf{\Psi}^{\mathrm{branch}}$ implies the convergence of the associated L\'evy processes through the multiparameter Lamperti time-change.
The latter convergence, of L\'evy processes, will be proved through the pointwise convergence of their associated characeristic exponent, as in Theorem 8.7 in \cite{satoLevyProcessesInfinitely2013}.

To prepare the way for the auxiliary lemma, note that if we define $\tilde{A} \in \mathbb{R}^{d \times d}$ by setting its $(i,j)$-th entry, with $i, j \in [d]$, as
\[
    \tilde{a}_{ij}
    = a_{ij} - \delta_{ij} \int_{U_d} \Bigl( (1 \wedge w_i) - w_i \bigl( 1_{\{\lVert w \rVert \leq 1\}} + (2 - \lVert w \rVert) 1_{\{1 < \lVert w \rVert \leq 2\}} \bigr) \Bigr) \nu_i(dw) \,,
\]
then we can rewrite \eqref{eq:char_exponent_levy} as
\begin{equation} \label{eq:char_exp_alt}
    \varphi_i(-\mathrm{i} \lambda)
    = - \gamma_i \lambda_i^2 + \mathrm{i} \langle\tilde{A} e_i, \lambda\rangle
    + \int_{U_d} \bigl( e^{ \mathrm{i} \langle \lambda, z \rangle } - 1 - \mathrm{i} z_i \lambda_i \upsilon(z) \bigr) \nu_i(dz),
    \quad \lambda \in \mathbb{R}^d,
\end{equation}
where $\upsilon(z) := 1_{\{\lVert z \rVert \leq 1\}} + (2 - \lVert z \rVert) 1_{\{1 < \lVert z \rVert \leq 2\}}$ is a continuous function with compact support commonly used in the representation of the characteristic exponent (see section 8,  Chapter 2 in \cite{satoLevyProcessesInfinitely2013}).

\begin{lemma} \label{lem:char_exp_convergence}
    Suppose that $\{(B^{(n)}, c^{(n)}, \mu^{(n)}) : n \in \mathbb{N}\}$ and $(B, c, \mu)$ are as in Proposition \ref{prop:convergence_mtbp_triplet}.
    For each $i \in [d]$ let $\varphi_i^{(n)}$ and $\varphi_i$, constructed as in \eqref{eq:char_exponent_levy}, be the characteristic exponents of the L\'evy processes $\mathcal{Y}_i^{(n)}$ and $\mathcal{Y}_i$, respectively.
    Under condition \eqref{eq:triplet_convergence} we get
    \begin{equation} \label{eq:convergence_of_characteristics}
        \lim_{n \to \infty} \varphi_i^{(n)}(-\mathrm{i}\lambda) = \varphi_i(-\mathrm{i}\lambda)
    \end{equation}
    for any $\lambda \in \mathbb{R}^d$ and $i \in [d]$.
\end{lemma}

\begin{proof}
    By \eqref{eq:char_exponent_levy}, for each $i \in [d]$, $\mathcal{Y}_i^{(n)}$ and $\mathcal{Y}_i$ are L\'evy processes on $\mathbb{R}^d$, whose characteristic triplets are $(B^{(n)} e_i, G_i^{(n)}, \mu_i^{(n)})$ and $(B e_i, G_i, \mu_i)$ respectively, where $G_i^{(n)} = (c_i^{(n)} \delta_{ij} \delta_{ik})_{j, k = 1}^d$ and $G_i = (c_i \delta_{ij} \delta_{ik})_{j, k = 1}^d$.
    Considering the alternative representation \eqref{eq:char_exp_alt} of $\varphi_i^{(n)}$ and $\varphi_i$, with $\tilde{B}^{(n)}$ and $\tilde{B}$ instead of $B^{(n)}$ and $B$ respectively, it will suffice to prove that the conditions of Theorem 8.7, Chapter 2 in \cite{satoLevyProcessesInfinitely2013} hold for each $i \in [d]$ under condition \eqref{eq:triplet_convergence} to deduce \eqref{eq:convergence_of_characteristics}.
    Namely, we will prove that for each $i \in [d]$:
    \begin{enumerate}
        \item $\int_{U_d} f(w) \mu_i^{(n)}(dw) \to \int_{U_d} f(w) \mu_i(dw)$ as $n \to \infty$ for every $f \in C_b(U_d)$ such that $f = 0$ in a neighborhood of $0$;
        \item for every $z \in \mathbb{R}^d$,
            \[
                \adjustlimits \lim_{\varepsilon \downarrow 0} \limsup_{n \to \infty} \Bigl\lvert z_i^2 c_i^{(n)} - z_i^2 c_i + \int_{U_d} \langle z, w \rangle^2 1_{\{\lVert w \rVert \leq \varepsilon\}} \mu_i^{(n)}(dw) \Bigr\rvert = 0;
            \]
        \item $\tilde{b}_{ij}^{(n)} \to \tilde{b}_{ij}$ as $n \to \infty$.
    \end{enumerate}

    We will first prove (1).
    For each $j \in [d]$ define the open subset $V_j = \{x \in \mathbb{S}_+^d : x_j > 1/\sqrt{2d}\}$ of $\mathbb{S}_+^d := \{x \in \mathbb{R}_+^d : \lVert x \rVert = 1\}$, and note that $\mathbb{S}_+^d \subset \bigcup_{j \in [d]} V_j$.
    By Theorem 2.13 in \cite{rudinRealComplexAnalysis2013}, there exists a collection of functions $\chi_1, \ldots, \chi_d \in C_c(\mathbb{S}_+^d)$ %{\color{blue}$C_c(U_d)?$} 
    such that $\sum_{j \in [d]} \chi_j \equiv 1$, and for each $j \in [d]$ both $0 \leq \chi_j \leq 1$ and $\operatorname{supp} \chi_j \subset V_j$ hold.
    For a function $f \in C_b(U_d)$ such that $f = 0$ in a neighborhood of $0$,
    \[
        \int_{U_d} f(w) \mu_i^{(n)}(dw)
        = \sum_{j = 1}^d \int_{U_d} \chi_j\biggl( \frac{w}{\lVert w \rVert} \biggr) \frac{f(w)}{(1 \wedge w_j)^{1 + \delta_{ij}}} (1 \wedge w_j)^{1 + \delta_{ij}} \mu_i^{(n)}(dw) \,.
    \]
    Note that
    \[
        w \mapsto \chi_j\biggl( \frac{w}{\lVert w \rVert} \biggr) \frac{f(w)}{(1 \wedge w_j)^{1 + \delta_{ij}}}
    \]
    defines a function in $C_b(U_d)$.
    Then, for each $j \in [d]$ we deduce, by the hypothesis $(1 \wedge w_j)^{1 + \delta_{ij}} \mu_i^{(n)}(dw) \to (1 \wedge w_j)^{1 + \delta_{ij}} \mu_i(dw)$ as $n \to \infty$ weakly, that
    \begin{align*}
        \MoveEqLeft
        \int_{U_d} \chi_j\biggl( \frac{w}{\lVert w \rVert} \biggr) \frac{f(w)}{(1 \wedge w_j)^{1 + \delta_{ij}}} (1 \wedge w_j)^{1 + \delta_{ij}} \mu_i^{(n)}(dw) \\
        & \to \int_{U_d} \chi_j\biggl( \frac{w}{\lVert w \rVert} \biggr) \frac{f(w)}{(1 \wedge w_j)^{1 + \delta_{ij}}} (1 \wedge w_j)^{1 + \delta_{ij}} \mu_i(dw) \text{ as } n \to \infty \,,
    \end{align*}
    which entails
    \[
        \lim_{n \to \infty} \int_{U_d} f(w) \mu_i^{(n)}(dw) = \int_{U_d} f(w) \mu_i(dw)\,.
    \]
    Thus, (1) is proved.

    Having proved (1), we can now show that (3) holds.
    Indeed, it suffices to note that
    \[
        w \mapsto (1 \wedge w_i) - w_i \bigl(1_{\{\lVert w \rVert \leq 1\}} + (2 - \lVert w \rVert) 1_{\{1 < \lVert w \rVert \leq 2\}} \bigr)
    \]
    defines a function in $C_b(U_d)$ that is equal to $0$ in a neighborhood of $0$.
    Therefore, (3) follows by (1) and the hypothesis $b_{ij}^{(n)} \to b_{ij}$ as $n \to \infty$ for each $j \in [d]$.

    Let us finish the proof by showing that (2) holds.
    In this direction, note that by elementary properties of the limit superior,
    \begin{align*} \MoveEqLeft
        \limsup_{n \to \infty} \Bigl\lvert z_i^2 c_i^{(n)} - z_i^2 c_i + \int_{U_d} \langle z, w \rangle^2 1_{\{\lVert w \rVert \leq \varepsilon\}} \mu_i^{(n)}(dw) \Bigr\rvert \\
        & \leq \limsup_{n \to \infty} z_i^2 \lvert c_i^{(n)} - c_i \rvert
        + \limsup_{n \to \infty} \int_{U_d} \langle z, w \rangle^2 1_{\{\lVert w \rVert \leq \varepsilon\}} \mu_i^{(n)}(dw)\,.
    \end{align*}
    As $c_i^{(n)} \to c_i$, $\limsup_{n \to \infty} z_i^2 \lvert c_i^{(n)} - c_i\rvert = 0$, so it will suffice to show that
    \begin{equation} \label{eq:limLimsup}
        \adjustlimits \lim_{\varepsilon \downarrow 0} \limsup_{n \to \infty}
        \int_{U_d} \langle z, w \rangle^2 1_{\{\lVert w \rVert \leq \varepsilon\}} \mu_i^{(n)}(dw) = 0 \,.
    \end{equation}
    To prove this, write $\mathcal{D}_i := \{\varepsilon > 0 : \mu_i(\{\lVert w \rVert = \varepsilon\}) > 0\}$ for the set of radii $\varepsilon > 0$ for which the sphere of radius $\varepsilon$ has positive measure under $\mu_i$.
    As 
    \[ 
        \varepsilon \mapsto \limsup_{n \to \infty} \int_{U_d} \langle z, w \rangle^2 1_{\{\lVert w \rVert \leq \varepsilon\}} \mu_i^{(n)}(dw)
    \] 
    is an increasing function from $(0, \infty)$ to $\mathbb{R}_+$, we can consider, without loss of generality, $\varepsilon \in (0, 1) \setminus \mathcal{D}_i$.
    By considering $\varepsilon \in (0, 1)$,
    \begin{align*}
        \int_{U_d} \langle z, w \rangle^2 1_{\{\lVert w \rVert \leq \varepsilon\}} \mu_i^{(n)}(dw) = {}
        & z_i^2 \int_{U_d} 1_{\{\lVert w \rVert \leq \varepsilon\}} (1 \wedge w_i)^2 \mu_i^{(n)}(dw) \\
        & + \sum_{j \in [d] \setminus \{i\}} z_j \int_{U_d} \Bigl( \sum_{k = 1}^d (1 + \delta_{ik}) z_k w_k \Bigr) 1_{\{\lVert w \rVert \leq \varepsilon\}} (1 \wedge w_j) \mu_i^{(n)}(dw) \,.
    \end{align*}
    Now, if $\varepsilon \in (0, 1) \setminus \mathcal{D}_i$, the assumption $(1 \wedge w_j)^{1 + \delta_{ij}} \mu_i^{(n)}(dw) \to (1 \wedge w_j)^{1 + \delta_{ij}} \mu_i(dw)$ as $n \to \infty$ weakly entails that, as $n \to \infty$,
    \[
        \int_{U_d} \langle z, w \rangle^2 1_{\{\lVert w \rVert \leq \varepsilon\}} \mu_i^{(n)}(dw) \to \int_{U_d} \langle z, w \rangle^2 1_{\{\lVert w \rVert \leq \varepsilon\}} \mu_i(dw) \,,
    \]
    so for $\varepsilon \in (0, 1) \setminus \mathcal{D}_i$,
    \[
        \limsup_{n \to \infty} \int_{U_d} \langle z, w \rangle^2 1_{\{\lVert \varepsilon \rVert \leq \varepsilon\}} \mu_i^{(n)}(dw)
        = \int_{U_d} \langle z, w \rangle^2 1_{\{\lVert w \rVert \leq \varepsilon\}} \mu_i(dw) \,.
    \]
    Finally, note that by Dominated Convergence we deduce
    \[
        \lim_{\varepsilon \downarrow 0} \int_{U_d} \langle z, w \rangle^2 1_{\{\lVert w \rVert \leq \varepsilon\}} \mu_i(dw) = 0\,,
    \]
    which allows us to obtain \eqref{eq:limLimsup}, finishing the proof.
\end{proof}

We now turn to the proof of Proposition \ref{prop:convergence_mtbp_triplet}.

\begin{proof}[Proof of Proposition \ref{prop:convergence_mtbp_triplet}]
    For each $n \in \mathbb{N}$ and every $i \in [d]$, let
    $\mathcal{Y}_i^{(n)}$ be a L\'evy process on $\mathbb{R}^d$ with characteristic triplet given by
    $(B^{(n)} e_i, G_i^{(n)}, \mu_i^{(n)})$, where
    $G_i^{(n)} = (c_i^{(n)} \delta_{ij} \delta_{ik})_{j, k = 1}^d$.
    Then, Lemma \ref{lem:char_exp_convergence},
    along condition \eqref{eq:triplet_convergence}, imply that
    \[
        \mathcal{Y}_i^{(n)} \Rightarrow \mathcal{Y}_i \text{ as $n \to \infty$ for each $i \in [d]$, }
    \]
    where $\mathcal{Y}_i$ is a L\'evy process with characteristic triplet $(B e_i, G_i, \mu_i)$
    in which $G_i = (c_i \delta_{ij} \delta_{ik})_{j,k=1}^d$,
    and the convergence is weak in the space of c\`adl\`ag paths from
    $\mathbb{R}_+$ to $\mathbb{R}_+^d$ with the Skorokhod topology.

    It then transpires, by the continuity of multiparameter time changes, see
    Theorem 2 in \cite{caballeroAffineProcessesMathbb2017}, that
    \[
        X^{(n)} \Rightarrow X \text{ as $n \to \infty$,}
    \]
    weakly on the space of c\`adl\`ag paths from $\mathbb{R}_+$ to $[0, \infty]^d$
    endowed with the Skorokhod topology or the uniform Skorokhod topology,
    according to if the branching mechanism of $X$ is nonexplosive or explosive.
\end{proof}

To finish the section we present the proof of Theorem \ref{theo:homeomorphismSpaces}.
\begin{proof}[Proof of Theorem \ref{theo:homeomorphismSpaces}]
    By Lemma \ref{lem:homeomorphismHz} we know that $H_z$ defines a homeomorphism between $\mathbf{\Psi}_a^{\mathrm{branch}}$ and $\mathbf{L}_{\mathrm{prop}}^{\mathrm{coal}}$.
    Propositions \ref{prop:conv_mtlc_measures} and \ref{prop:convergence_mtbp_triplet} tell us that the homeomorphism between the parameters of MCSBPs and those of multitype $\Lambda$-coalescents extends to a homeomorphism between the associated random processes, finishing the proof.
\end{proof}

\section{Proof of the sequential sampling procedure}\label{sec:proof}

In this section we will provide the proof of Theorem \ref{theo:culling_limit}.  
%\textcolor{cyan}{
 We show that the sequence of processes obtained from the sequential sampling
procedure converges to the unique strong solution of
\eqref{eq:sde_d_colonies}, whose well-posedness follows from
Theorem \ref{theo:SDEProperties}. We also verify that the limiting process has
the desired properties.

We begin with an auxiliary result that characterizes the dynamics of the process
$(R,Z)$ defined in \eqref{eq:defRZ} through a martingale problem.
%}
\begin{proposition} \label{prop:inf_gen_rz}
    For any fixed $f \in C_c^2([0, 1]^d \times \mathbb{R}_+^d)$, the process
    \[
        M(t\wedge\kappa) := f(R(t\wedge\kappa), Z(t\wedge\kappa)) - f(r, z) - \int_0^{t\wedge\kappa} \mathcal{A} f(R(s), Z(s)) ds,
    \]
    is a local martingale, where $\kappa = \min_{i \in [d]} \inf\{t \geq 0 : Z_i(t) \in \{0, \infty\}\}$, and
    \begin{align*}
        \mathcal{A} f(r, z) = {} 
        & \sum_{i = 1}^d c_i z_i \biggl( \frac{(1 - r_i) r_i}{z_i^2} \partial_{r_i r_i} f(r, z) + \partial_{z_i z_i} f(r, z) \biggr) \\
        & + \sum_{i = 1}^d \sum_{j = 1}^d b_{ij} \biggl(
        \frac{z_j}{z_i} (r_j - r_i) \partial_{r_i} f(r, z) + z_j \partial_{z_i} f(r, z)
        \biggr) \\
        & + \sum_{i = 1}^d \int_{[0, 1)^d \setminus \{0\}} \bigl(
        \begin{aligned}[t]
             r_i f(r + (1 - r) \odot u, z + T_z^{-1}(u)) 
             & + (1 - r_i) f(r - r \odot u, z + T_z^{-1}(u)) \\ 
             & - f(r, z) - \xi_i(T_z^{-1}(u)) \partial_{z_i} f(r, z) \bigr) z_i \mathbf{T}_z \mu_i(du) \,,
        \end{aligned}
        % & + \sum_{i = 1}^d z_i r_i \int_{[0, 1)^d \setminus \{0\}} \Bigl(
        % \begin{aligned}[t] 
        %     & f\bigl(r + (1 - r) \odot u, z + T_z^{-1}(u)\bigr) - f(r, z) \\
        %     & - (1 - r_i) u_i \partial_{r_i} f(r, z) - \xi_i(T_z^{-1}(u)) \partial_{z_i} f(r, z) \Bigr) \mathbf{T}_z\mu_i(du) 
        % \end{aligned} \\
        % & + \sum_{i = 1}^d z_i (1 - r_i) \int_{[0, 1)^d \setminus \{0\}} \Bigl(
        % \begin{aligned}[t]
        %     & f\bigl(r - r \odot u, z + T_z^{-1}(u)\bigr) - f(r, z) \\
        %     & + r_i u_i \partial_{r_i} f(r, z) - \xi_i(T_z^{-1}(u)) \partial_{z_i} f(r, z) \Bigr) \mathbf{T}_z\mu_i(du), 
        % \end{aligned}% \\
    \end{align*}
    recalling that $\odot$ represents the Hadamard product.
\end{proposition}

\begin{proof} Let us recall that given $x, y \in U_d$, $r \in [0, 1]^d$ and $z \in \mathbb{R}_+^d$
are defined through
\[
    r_i = \frac{x_i}{x_i + y_i}
    \quad\text{and}\quad
    z_i = x_i + y_i
    \quad \text{for $i \in [d]$.}
\]
Moreover, $x = z \odot r$ and $y = z - x$.
By this correspondence and the chain rule we get, for any function $f \in C_c^2([0, 1]^d \times \mathbb{R}_+^d)$ and for each $i \in [d]$,
\begin{align*}
    \partial_{x_i} f(r, z) = {}
    & \frac{(1 - r_i)}{z_i} \partial_{r_i} f(r, z) + \partial_{z_i} f(r, z), \\
    \partial_{x_i x_i} f(r, z) = {}
    % & \frac{2 (1 - r_i)}{z_i^2} (\partial_{r_i z_i} f(r, z) - \partial_{r_i} f(r, z))
    % + \frac{(1 - r_i)^2}{z_i^2} \partial_{r_i r_i} f(r, z)
    % + \partial_{z_i z_i} f(r, z), \\
    % & {\color{blue}= -\frac{2 (1 - r_i)}{z_i^2}\partial_{r_i} f(r, z)+\frac{(1 - r_i)}{z_i} \partial_{r_i z_i} f(r, z)
    % + \frac{(1 - r_i)^2}{z_i^2} \partial_{r_i r_i} f(r, z)
    % + \partial_{z_i z_i} f(r, z),?} \\
    &  - 2 \frac{1 - r_i}{z_i^2} \partial_{r_i} f(r, z) + \frac{(1 - r_i)^2}{z_i^2} \partial_{r_i r_i} f(r, z) + 2 \frac{1 - r_i}{z_i} \partial_{r_i z_i} f(r, z) + \partial_{z_i z_i} f(r, z),  \\
    \partial_{y_i} f(r, z) = {}
    & -\frac{r_i}{z_i} \partial_{r_i} f(r, z) + \partial_{z_i} f(r, z), \\
    \partial_{y_i y_i} f(r, z) = {}
    % & \frac{2 r_i}{z_i^2} (\partial_{r_i} f(r, z) - \partial_{r_i z_i} f(r, z))
    % + \frac{r_i^2}{z_i^2} \partial_{r_i r_i} f(r, z)
    % + \partial_{z_i z_i} f(r, z).\\
    % &{\color{blue} =\frac{2 r_i}{z_i^2} \partial_{r_i} f(r, z)-\frac{ r_i}{z_i} \partial_{r_i z_i} f(r, z) - \partial_{r_i z_i} f(r, z))
    % + \frac{r_i^2}{z_i^2} \partial_{r_i r_i} f(r, z)
    % + \partial_{z_i z_i} f(r, z).?} \\
    & 
    2 \frac{r_i}{z_i^2} \partial_{r_i} f(r, z) + \frac{r_i^2}{z_i^2} \partial_{r_i r_i} f(r, z) - 2 \frac{r_i}{z_i} \partial_{r_i z_i} f(r, z) + \partial_{z_i z_i} f(r, z).
\end{align*}
Hence, we can compute $\mathcal{A}f(r, z)$ by considering
$f(r, z) = f \circ g (x, y)$, where $g(x, y) = (r, z)$,
and the action of $\mathcal{A}$ over $f \circ g$, given by \eqref{eq:branching_two_colonies}.
Therefore, by definition of $(R, Z)$ and Dynkin's formula, see Proposition IV.1.7 in \cite{ethierMarkovProcessesCharacterization2005}, we deduce that
\[
    M(t) := f(R(t), Z(t)) - f(r, z) - \int_0^t \mathcal{A} f(R(s), Z(s)) ds,
\]
defines a local semimartingale.
To obtain Proposition \ref{prop:inf_gen_rz} it remains to compute $\mathcal{A} f(r, z)$.

Note that
\[
    \partial_{x_i} f(r, z) + \partial_{y_i} f(r, z)
    = \frac{(1 - 2 r_i)}{z_i} \partial_{r_i} f(r, z) + 2 \partial_{z_i} f(r, z).
\]
Hence, by using the identities $x_i = z_i r_i$ and $y_i = z_i (1 - r_i)$ for each $i \in [d]$, we obtain
\begin{equation}
\begin{split}
    \langle Bx, \nabla_x f(r, z)\rangle + \langle By, \nabla_y f(r, z)\rangle
    & = \sum_{i = 1}^d \biggl[
        \begin{aligned}[t]
            & \sum_{j = 1}^d b_{ij} z_j r_j \biggl( \frac{(1 - r_i)}{z_i} \partial_{r_i} f(r, z) + \partial_{z_i} f(r, z) \biggr) \\
            & + \sum_{j = 1}^d b_{ij} z_j (1 - r_j) \biggl( - \frac{r_i}{z_i} \partial_{r_i} f(r, z) + \partial_{z_i} f(r, z) \biggr) \biggr]
        \end{aligned} \\
    & = \sum_{i = 1}^d \sum_{j = 1}^d b_{ij} \biggl(
        \frac{z_j}{z_i} (r_j - r_i) \partial_{r_i} f(r, z) + z_j \partial_{z_i} f(r, z)
     \biggr)
\end{split}
\end{equation}
On the other hand, we note that
\[
    r_i \partial_{x_i x_i} f(r, z) + (1 - r_i) \partial_{y_i y_i} f(r, z)
    = \frac{r_i (1 - r_i)}{z_i^2} \partial_{r_i r_i} f(r, z) + \partial_{z_i z_i} f(r, z),
\]
and therefore
\begin{equation}
    \sum_{i = 1}^d c_i \bigl( x_i \partial_{x_i x_i} f(r, z) + y_i \partial_{y_i y_i} f(r, z) \bigr)
    = \sum_{i = 1}^d c_i z_i \biggl( \frac{r_i (1 - r_i)}{z_i^2} \partial_{r_i r_i} f(r, z) + \partial_{z_i z_i} f(r, z) \biggr).
\end{equation}
We now turn to the part concerning $\mu_i$.
In this case, for $w \in U_d$, we note that $(x, y) \mapsto (r, z)$ maps $(x + w, z)$
to
\[
    (\tilde{r}_w, \tilde{z}_w)
    := \biggl( \biggl( \frac{x_1 + w_1}{x_1 + y_1 + w_1}, \ldots, \frac{x_d + w_d}{x_d + y_d + w_d} \biggr), (x_1 + y_1 + w_1, \ldots, x_d + y_d + w_d) \biggr).
\]
Straightforward computations yield $x_i + y_i + w_i = z_i + w_i$ and
\[
    \frac{x_i + w_i}{x_i + y_i + w_i}
    = r_i \frac{z_i}{z_i + w_i} + \frac{w_i}{z_i + w_i}
    = r_i + (1 - r_i) \frac{w_i}{z_i + w_i}
    = r_i + (1 - r_i) T_z(w)_i,
\]
for $i \in [d]$.
Thus,
\begin{align*} \MoveEqLeft
    x_i \int_{U_d} \bigl(f(\tilde{r}_w, \tilde{z}_w) - f(r, z) - \xi_i(w) \partial_{x_i} f(r, z) \bigr) \mu_i(dw) \\
    & = z_i r_i \int_{U_d} \biggl( 
    \begin{aligned}[t] 
        f(r + (1 - r) \odot T_z(w)) 
        & - f(r, z) \\
        & - \xi_i(w) \frac{(1 - r_i)}{z_i} \partial_{r_i} f(r, z)  
        - \xi_i(w) \partial_{z_i} f(r, z) \biggr) \mu_i(dw). 
    \end{aligned}
\end{align*}
Similarly, if we denote by $(\hat{r}_w, \hat{z}_w)$ the image of $(x, y + w)$ under
the map $(x, y) \mapsto (r, z)$,
\begin{align*} \MoveEqLeft
    y_i \int_{U_d} \bigl(f(\hat{r}_w, \hat{z}_w) - f(r, z) - \xi_i(w) \partial_{x_i} f(r, z) \bigr) \mu_i(dw) \\
    & = z_i (1 - r_i) \int_{U_d} \biggl( 
    \begin{aligned}[t] 
        f(r  - r \odot T_z(w)) 
        & - f(r, z) \\
        & + \xi_i(w) \frac{r_i}{z_i} \partial_{r_i} f(r, z) - \xi_i(w) \partial_{z_i} f(r, z) \biggr) \mu_i(dw). 
    \end{aligned}
\end{align*}
Summing the last two displays, we obtain the first equality in  
\begin{align*}
    \MoveEqLeft 
    x_i \int_{U_d} \bigl( f(\tilde{r}_w, \tilde{z}_w) - f(r, z) - \xi_i(w) \delta_{x_i} f(r, x) \bigr) \mu_i(dw) \\ 
    & + y_i \int_{U_d} \bigl( f(\tilde{r}_w, \tilde{z}_w) - f(r, z) - \xi_i(w) \delta_{x_i} f(r, x) \bigr) \mu_i(dw) \\ 
    & \quad = \int_{U_d} \bigl(
        \begin{aligned}[t]
            r_i f( r + (1 - r) \odot T_z(w), z + w) 
            + (1 - r_i) f(r - r \odot T_z(w), z + w) 
            & - f(r, z) \\ 
            & - \xi_i(w) \partial_{z_i} f(r, z)  \bigr) z_i \mu_i(dw)
        \end{aligned} \\
    & \quad = \int_{[0, 1)^d \setminus \{0\}} \bigl(
        \begin{aligned}[t]
            r_i f(r + (1 - r) \odot u, z + T_z^{-1}(u)) 
            & + (1 - r_i) f(r - r \odot u, z + T_z^{-1}(u)) \\
            & - f(r, z) - \xi_i(T_z^{-1}(u)) \partial_{z_i} f(r, z) \bigr) z_i \mathbf{T}_z \mu_i(du) \,.
        \end{aligned}
\end{align*}
The latter equality follows by definition of $T_z$ and $\mathbf{T}_z$. 
Then, by the previous identities, the form of $\mathcal{A}f(r, z)$ displayed in Proposition \ref{prop:inf_gen_rz} follows. 
\end{proof}

%\subsection{Sequential sampling}
%\label{subsec:culling}

We now proceed to the proof of Theorem \ref{theo:culling_limit}. 

\begin{proof}[Proof of Theorem \ref{theo:culling_limit}]
Let us consider $z \in (0, \infty)^d$ fixed. 
By the existence of a unique strong solution to \eqref{eq:genWFMod} and Theorem \ref{theo:homeomorphismSpaces}, we deduce the existence of a unique strong solution $R^{(z, r)}$ to \eqref{eq:sde_d_colonies}. 
Furthermore, we know that such solution is Feller, and its generator is provided by \eqref{eq:gen_d_colonies}.

For each $f \in C^2([0, 1]^d)$ and $(r, z) \in [0, 1]^d \times \mathbb{R}_+^d$, we set 
\begin{align*} 
    \mathcal{B}_z f(r) 
    & := \sum_{i = 1}^d \frac{c_i}{z_i} (1 - r_i) r_i \partial_{r_i r_i} f(r) + \sum_{i = 1}^d \sum_{j = 1}^d b_{ij} \frac{z_j}{z_i} (r_j - r_i) \partial_{r_i} f(r) \\
    & \quad + \sum_{i = 1}^d \int_{[0, 1)^d \setminus \{0\}} \bigl( r_i f(r + (1 - r) \odot u) + (1 - r_i) f(r - r \odot u) - f(r) \bigr) z_i \mathbf{T}_{z} \mu_i(du) .
\end{align*}
By Proposition \ref{prop:inf_gen_rz} we observe that there exists a local martingale $M_f$ such that the equality 
\[
    f\Bigl( R\Bigl( \frac{1}{n} \wedge \tau \Bigr) \Bigr) - f(r) 
    = \int_0^{n^{-1} \wedge \tau} \mathcal{B}_{Z(s)} f(R(s)) ds + M_f\Bigl( \frac{1}{n} \wedge \tau \Bigr)
\]
holds. 
Now, from 
a simple modification of the proof of Theorem 1 in
\cite{caballeroRelativeFrequencyTwo2023}, which treats the case $d=1$, one obtains that $\mathcal{A}^{(z)}_nf(r)\to\mathcal{A}^{(z)}f(r)$ as $n\to\infty$ uniformly on $[0,1]^d$.
Hence, by Theorem 17.28 in \cite{Ka} we conclude that for any $z \in \mathbb{R}_+^d \setminus \{0\}$ and $T > 0$,
$\overline{R}^{n} \Rightarrow R^{(z,r)}$ as $n \to \infty$ in $D([0, T], [0, 1]^d)$,
where $R^{(z, r)}$ is the solution to \eqref{eq:sde_d_colonies} with initial
condition $R^{(z, r)}(0) = r.$
This concludes the proof of Theorem \ref{theo:culling_limit}. 
% The modification boils down to observing that for $s \in [0, \tau)$, $Z_i \in [\varepsilon, L]$ for all $i \in [d]$, so $R(s)$ is well defined and we also have that for $(r, z) \in [0, 1]^d \times [\epsilon, L]^d$,
% \[
%   \lvert B_z f(r) \rvert \leq 
%   K_f < \infty ,
% \]
% where 
% \[
%     K_f := 
%     \lVert f \rVert_{\infty, 2} \Bigl( \sum_{i = 1}^d \frac{c_i}{4 \varepsilon} + \sum_{i = 1}^d \sum_{j, k = 1}^d \int_{[0, 1)^d \setminus \{0\}} u_j u_k z_i \mathbf{T}_z \mu_i(du) \Bigr)
%     + \lVert f \rVert_{\infty, 1} \sum_{i = 1}^d \sum_{j \in [d] \setminus \{i\}} \Bigl( b_{ij} \frac{L}{\varepsilon} + \int_{[0, 1)^d \setminus \{0\}} u_j z_i \mathbf{T}_z \mu_i(du) \Bigr) .
% \]
% In the previous display 
% \[
%     \lVert f \rVert_{\infty, 1} = \max_{i \in [d]} \lVert \partial_i f \rVert_{\infty} 
%     \quad\text{and}\quad
%     \lVert f \rVert_{\infty, 2} = \max_{i, j \in [d]} \lVert \partial_{ij} f \rVert_{\infty}.
% \]
\end{proof}

\appendix 

\section{Proof of Theorem \ref{theo:SDEProperties}}\label{proof_theo:SDEProperties}
In this section, we prove Theorem \ref{theo:SDEProperties}. We begin by proving
pathwise uniqueness for solutions to \eqref{eq:lambdaWF}. We then prove weak
uniqueness. Together with the Yamada--Watanabe theorem, this yields existence
and uniqueness of a strong solution.
\subsection{Pathwise uniqueness}
We start by stating that if a process $R = (R_1, \ldots, R_d)$ satisfies
\eqref{eq:lambdaWF} and its initial value is in $[0, 1]^d$, then with
probability one, it will remain in the set $[0, 1]^d$. The proof is a straightforward
modification of Proposition 2.1 in \cite{fuStochasticEquationsNonnegative2010} and is therefore is omitted.

\begin{proposition}
    If $R$ satisfies \eqref{eq:lambdaWF} and $\mathbb{P}(R(0) \in [0, 1]^d) = 1$,
    then \[\mathbb{P}(R(t) \in [0, 1]^d \text{ for all } t \geq 0) = 1.\]
\end{proposition}

In what follows we will assume that the initial value $R(0)$ is in $[0, 1]^d$.
In particular, if $R(0) = r \in [0, 1]^d$, we emphazise the dependence on the
initial value by denoting the process as $R^{(r)}$.
We now state and prove the pathwise uniqueness of solutions with values in
$[0, 1]^d$ that satisfy \eqref{eq:lambdaWF}.

\begin{proposition} \label{prop:pathwise_uniqueness}
    Pathwise uniqueness holds for the solution of \eqref{eq:lambdaWF}.
\end{proposition}

\begin{proof}
    For $x, u \in [0, 1]^d$, $i, j \in [d]$, and $v \in \mathbb{R}_+$ we define
    the functions
    \begin{align*}
        % m_i(x)
        % & =  \sum_{k \in [d] \setminus \{i\}}
        % \biggl( b_{ik} \frac{z_k}{z_i} + z_k \int_{[0, 1)^d \setminus \{0\}} u_i \mathbf{T}_z\mu_k(du) \biggr)
        % (x_k - x_i), \\
        % \bar{m}_i(x)
        % & = \sum_{k \in [d] \setminus \{i\}}
        % \biggl( z_k \int_{[0, 1)^d \setminus \{0\}} u_i \mathbf{T}_z \mu_k(du) \biggr) (x_k - x_i), \\
        % \sigma_i(x)
        % & = \sqrt{\frac{2 c_i}{z_i} x_i (1 - x_i)} ,\\
        % g_i^j(x, u, v)
        % & = u_i (1 - x_i) 1_{\{v \leq z_j x_j\}}, \\
        % h_i^j(x, u, v)
        % & = -u_i x_i 1_{\{v \leq z_j (1 - x_j)\}}.
        m_i(x)
        & = \sum_{k \in [d] \setminus \{i\}} \Bigl( \rho_{ki} + \int_{[0, 1]^d} u_i \mathcal{Q}_k(du) \Big) (x_k - x_i) ,\\
        \bar{m}_i(x)
        & = \sum_{k \in [d] \setminus \{i\}} \Bigl( \int_{[0, 1]^d} u_i \mathcal{Q}_k(du) \Bigr) (x_k - x_i) ,\\
        \sigma_i(x) 
        & = \sqrt{\rho_{ii} x_i (1 - x_i)}, \\
        g_i^j(x, u, v) 
        & = u_i (1_{\{v \leq x_j\}} - x_i) .
    \end{align*}
    Consider $i \in [d]$ fixed.
    As $\widetilde{m}_i = m_i - \bar{m}_i$ is affine, there exists a constant $K_m^i > 0$ such that for any
    $x, y \in [0, 1]^d$, $|\widetilde{m}_i(x) - \widetilde{m}_i(y)| \leq K_m^i \lVert x - y\rVert$.
    Now, it is easily seen that for any choice of $x, y \in [0, 1]^d$,
    \[
        |\sigma_i(x) - \sigma_i(y) |^2
        \leq %\frac{2 c_i}{z_i} 
        \rho_{ii} |x_i - y_i|.
    \]
    Finally, fixing $i, j \in [d]$, it is clear that the function %functions
    $x \mapsto x_i + g_i^j(x, u, v) = x_i (1 - u_i) + u_i 1_{\{v \leq x_j\}}$ 
    %and $x \mapsto x_i + h_i^j(x, u, v)$
    is increasing in $x_i$ for any $(u, v) \in [0, 1]^d \times [0, 1]$.
    Additionally, the inequality 
    \begin{align}
        |g_i^j(x, u, v) - g_i^j(y, u, v)|
        & \leq 
        % u_i |x_i - y_i| 1_{\{v \leq z_j (x_j \wedge y_j)\}}
        % + u_i 1_{\{z_j (x_j \wedge y_j) < v \leq z_j (x_j \vee y_j)\}},
        % u_i |x_i - y_i| (1_{\{v \leq x_j \wedge y_j\}} + 1_{\{v > x_j \vee y_j\}}) 
        % +{\color{blue}2?} u_i 1_{\{x_j \wedge y_j < v \leq x_j \vee y_j\}},
        %\mathcolor{cyan}{
            u_i |x_i - y_i| (1_{\{v \leq x_j \wedge y_j\}} + 1_{\{v > x_j \vee y_j\}}) 
            + 2 u_i 1_{\{x_j \wedge y_j < v \leq x_j \vee y_j\}},
        \label{eq:preLip_g} %\\
        % |h_i^j(x, u, v) - h_i^j(x, u, v)|
        % & \leq u_i |x_i - y_i| 1_{\{v \leq z_j (1 - x_j \vee y_j)\}}
        % + u_i 1_{\{z_j (1 - x_j \vee y_j) < v \leq z_j (1 - x_j \wedge y_j)\}},
        % \label{eq:preLip_h}
    \end{align}
    %{\color{blue} The factor $2$ above doesn't follow from the fact that $x,y\in[0,1]^d$ are arbitrary?}
    %\textcolor{cyan}{[Yes, you are right. I confused $|-x_i - (1 - y_i)|$ with $|x_i - (1 - y_i)|$. I have updated the bounds with this correction.]}
    readily implies 
    \begin{align}\label{aux_prov_1}
        \int_{[0, 1]^d} \int_0^1 \bigl( g_i^j(x, u, v) - g_i^j(x, u, v) \bigr)^2 \mathcal{Q}_i(du) 
        % & \leq {\color{blue}2?}\Bigl( \int_{[0, 1]^2} u_i^{1 + \delta_{ij}} \mathcal{Q}_{j}(du) \Bigr) \bigl( |x_i - y_i| + |x_j - y_j| \bigr) .
        & %\mathcolor{cyan}{
        \leq \Bigl( \int_{[0, 1]^2} u_i^{1 + \delta_{ij}} \mathcal{Q}_{j}(du) \Bigr) \bigl( |x_i - y_i| + 4 |x_j - y_j| \bigr) .
        %}
        % \int_{[0, 1)^d \setminus \{0\}} \int_0^\infty \bigl( (g_i^i(x, u, v) - g_i^i(y, u, v))^2
        % + (h_i^i(x, u, v) - h_i^i(y, u, v))^2 \bigr) dv \mathbf{T}_z\mu_i(du) \notag\\
        % & \leq \biggl( 4 z_i \int_{[0, 1]^d} u_i^2 \mathbf{T}_z\mu_i(du) \biggr)
        % |x_i - y_i|,
    \end{align}
    %{\color{blue} If the $2$ above in \eqref{eq:preLip_g} is correct then that is why I corrected the factor above.}
    % and for $j \in [d] \setminus \{i\}$,
    % \begin{align} \MoveEqLeft
    %     \int_{[0, 1)^d \setminus \{0\}} \int_0^\infty \bigl( |g_i^j(x, u, v) - g_i^j(y, u, v)|
    %     + |h_i^j(x, u, v) - h_i^j(y, u, v)|\bigr) dv \mathbf{T}_z\mu_j(du)
    %     \label{eq:gij_lipschitz}\notag\\
    %     & \leq \biggl( 2 z_j \int_{[0, 1]^d} u_i \mathbf{T}_z\mu_j(du) \biggr)
    %     \bigl( |x_i - y_i| + |x_j - y_j| \bigr).
    % \end{align}
    The result now follows by modifications of the proofs of Theorem 3.1
    in \cite{maStochasticEquationsTwotype2014} and Theorem 3.2 in
    \cite{liStrongSolutionsJumptype2012}.

    Consider the sequence of functions $\{\phi_k\}$ defined
    in Section 3.1 of \cite{liStrongSolutionsJumptype2012}.
    We recall that these functions satisfy the properties
    \begin{enumerate}
        \item $\phi_k(z) \to |z|$ non-decreasingly as $k \to \infty$ for $z \in \mathbb{R}$;
        \item $0 \leq \mathrm{sgn}(z) \phi'_k(z) \leq 1$, for $z\in\mathbb{R}$, where $\mathrm{sgn}$ is the sign function;
        \item $0 \leq |z| \phi_k''(z) \leq 2/k$ for $z \in \mathbb{R}$.
    \end{enumerate}
    Now, if $R$ and $S$ are two solutions to \eqref{eq:lambdaWF} and we
    define $\zeta = R - S$, then $\zeta$ is a solution to
    \begin{align*}
        d \zeta_i(t)= {}
        & \bigl( m_i(R(t)) - m_i(S(t)) \bigr) dt
        + \bigl( \sigma_i(R(t)) - \sigma_i(S(t)) \bigr) dB_i(t) \\
        & + \sum_{j = 1}^d \int_{[0, 1]^d} \int_0^1 \bigl( g_i^j(R(t-), u, v) - g_i^j(S(t-), u, v) \bigr) \widetilde{\mathcal{N}}_j(dt, du, dv)
        % & + \sum_{j = 1}^d \int_{[0, 1)^d \setminus \{0\}} \int_0^\infty
        % \bigl( g_i^j(R(t-), u, v) - g_i^j(S(t-), u, v) \bigr)
        % \widetilde{N}_1^j(dt, du, dv) \\
        % & + \sum_{j = 1}^d \int_{[0, 1)^d \setminus \{0\}} \int_0^\infty
        % \bigl( h_i^j(R(t-), u, v) - h_i^j(S(t-), u, v) \bigr)
        % \widetilde{N}_2^j(dt, du, dv),% \\
    \end{align*}
    for $i \in [d]$.
    Then, for any $k \in \mathbb{N}$, Itô's formula yields
    \begin{align*}
        \phi_k(\zeta_i(t)) = {}
        & M_{\phi_k}(t) + \phi_k(\zeta_i(0))
        + \int_0^t \phi_k'(\zeta_i(v)) \bigl(\widetilde{m}_i(R(v)) - \widetilde{m}_i(S(v))\bigr) dv\\
        &+ \int_0^t \phi_k''(\zeta_i(v)) \frac{1}{2} \bigl(\sigma_i(R(v)) - \sigma_i(S(v))\bigr)^2 dv \\
        & + \int_0^t \int_{[0, 1]^d} \int_0^1 D_{g_i^i(R(v), u, v, w) - g_i^i(S(v), u, w)} \phi_k(\zeta_i(v)) dw \mathcal{Q}_i(du) dv \\
        & + \sum_{j \in [d] \setminus \{i\}} \int_0^t \int_{[0, 1]^d} \int_0^1 \Delta_{g_i^j(R(v), u, w) - g_i^j(S(v), u, w)} \phi_k(\zeta_i(v)) dw \mathcal{Q}_j(du) dv ,
        % & + \int_0^t \int_{[0, 1)^d \setminus \{0\}} \int_0^\infty
        % D_{g_i^i(R(v) u, w) - g_i^i(S(v),u, w)} \phi_k(\zeta_i(v))
        % dw \mathbf{T}_z \mu_i(du) dv \\
        % & + \int_0^t \int_{[0, 1)^d \setminus \{0\}} \int_0^\infty
        % D_{h_i^i(R(v),u, w) - h_i^i(S(v),u, w)} \phi_k(\zeta_i(v))
        % dw \mathbf{T}_z \mu_i(du) dv \\
        % & + \sum_{j \in [d] \setminus \{i\}} \int_0^t \int_{[0, 1)^d \setminus \{0\}} \int_0^\infty
        % \Delta_{g_i^j(R(v), u, w) - g_i^j(S(v), u, w)} \phi_k(\zeta_i(v))
        % dw \mathbf{T}_z \mu_j(du) dv \\
        % & + \sum_{j \in [d] \setminus \{i\}} \int_0^t \int_{[0, 1)^d \setminus \{0\}} \int_0^\infty
        % \Delta_{h_i^j(R(v), u, w) - h_i^j(S(v), u, w)} \phi_k(\zeta_i(v))
        % dw \mathbf{T}_z \mu_j(du) dv,% \\
    \end{align*}
    where $M_{\phi_k}$ is a martingale due to the fact that $\zeta_i(v) \in [-1, 1]$ for
    all $v \in \mathbb{R}_+$ almost surely in conjunction with upcoming the bounds \eqref{eq:phi_prime},
    \eqref{eq:phi_pprime}, \eqref{eq:delta_muj} and
    \eqref{eq:D_mui}, and where for $x,y\in\mathbb{R}$,
    \[
        \Delta_y \phi_k(x) = \phi_k(x + y) - \phi_k(x)
        \quad\text{and}\quad
        D_y \phi_k(x) = \Delta_y \phi_k(x) - y \phi_k'(x) .
    \]
    We now note that
    \begin{align}
        \phi_k'(\zeta_i(v)) \bigl( \widetilde{m}_i(R(v)) - \widetilde{m}_i(S(v)) \bigr)
        & \leq |\widetilde{m}_i(R(v)) - \widetilde{m}_i(S(v))| \leq K_m^i \lVert \zeta(v) \rVert,
        \label{eq:phi_prime} \\
        \phi_k''(\zeta_i(v)) \bigl( \sigma_i(R(v)) - \sigma_i(S(v)) \bigr)^2
        & \leq %\frac{2 c_i}{z_i} 
        \rho_{ii} \phi_k''(\zeta_i(v)) |\zeta_i(v)|
        \leq %\frac{4 c_i}{k z_i}
        \frac{2 \rho_{ii}}{k}, \label{eq:phi_pprime}
    \end{align}
    while
    for $j \in [d] \setminus \{i\}$,
    \begin{equation} \label{eq:delta_muj}
        \begin{split}
            \int_{[0, 1]^d} \int_0^1 &\Delta_{g_i^j(R(v), u, w) - g_i^j(S(v), u, w)} \phi_k(\zeta_i(v)) dw \mathcal{Q}_j(du) \\
            & \leq \int_{[0, 1]^d} \int_0^1 \bigl| g_i^j(R(v), u, w) - g_i^j(S(v), u, w) \bigr| dw \mathcal{Q}_j(du) \\ 
            & % \leq {\color{blue}2?}\Bigl( \int_{[0, 1]^d} u_i \mathcal{Q}_j(du) \Bigr) \bigl( |\zeta_i(v)| + |\zeta_j(v)| \bigr) .
            \leq 
                \Bigl( \int_{[0, 1]^d} u_i \mathcal{Q}_j(du) \Bigr) \bigl( |\zeta_i(v)| + 2 |\zeta_j(v)| \bigr) .
        \end{split}
    \end{equation}
    % \begin{align}   \label{eq:delta_muj}
    %     \MoveEqLeft
    %     \int_{[0, 1)^d \setminus \{0\}} \int_0^\infty
    %     \Delta_{g_i^j(R(v), u, w) - g_i^j(S(v), u, w)} \phi_k(\zeta_i(v))
    %     dw \mathbf{T}_z \mu_j(du) \notag \\
    %     \MoveEqLeft 
    %     \quad + \int_{[0, 1)^d \setminus \{0\}} \int_0^\infty
    %     \Delta_{h_i^j(R(v), u, w) - h_i^j(S(v) u, w)} \phi_k(\zeta_i(v))
    %     dw \mathbf{T}_z \mu_j(du) \notag\\
    %     & \leq
    %     \int_{[0, 1)^d \setminus \{0\}} \int_0^\infty |g_i^j(R(v), u, w) - g_i^j(S(v), u, w)| dw \mathbf{T}_z\mu_j(du) \notag\\
    %     & \quad + 
    %     \int_{[0, 1)^d \setminus \{0\}} \int_0^\infty |h_i^j(R(v), u, w) - h_i^j(S(v), u, w)| dw \mathbf{T}_z\mu_j(du) \notag\\
    %     & \leq \biggl( 2 z_j \int_{[0, 1)^d \setminus \{0\}} u_i \mathbf{T}_z\mu_j(du) \biggr)
    %     \bigl( |\zeta_i(v)| + |\zeta_j(v)| \bigr).
    % \end{align}
    On the other hand, by Lemma 3.1 in \cite{liStrongSolutionsJumptype2012},
    we deduce the first inequality in
    % \begin{align} \MoveEqLeft \label{eq:D_mui}
    %     \int_{[0, 1)^d \setminus \{0\}} \int_0^\infty
    %     D_{g_i^i(R(v)u, w) - g_i^i(S(v), u, w)} \phi_k(\zeta_i(v))
    %     dw \mathbf{T}_z \mu_i(du) \notag \\
    %     \MoveEqLeft
    %     \quad + \int_{[0, 1)^d \setminus \{0\}} \int_0^\infty
    %     D_{h_i^i(R(v), u, w) - h_i^i(S(v), u, w)} \phi_k(\zeta_i(v))
    %     dw \mathbf{T}_z \mu_i(du) \notag\\
    %     & \leq \frac{2}{k |\zeta_v^i|} \int_{[0, 1)^d \setminus \{0\}} \int_0^\infty
    %     (g_i^i(R(v)u, w) - g_i^i(S(v), u, w))^2 
    %     dw \mathbf{T}_z \mu_i(du) \notag\\
    %     & \quad + \frac{2}{k |\zeta_v^i|} \int_{[0, 1)^d \setminus \{0\}} \int_0^\infty
    %     (h_i^i(R(v), u, w) - h_i^i(S(v), u, w))^2 
    %     dw \mathbf{T}_z \mu_i(du) \notag\\
    %     & \leq \frac{8 z_i}{k} \int_{[0, 1)^d \setminus \{0\}} u_i^2 \mathbf{T}_z\mu_i(du).
    % \end{align}
    \begin{equation} \label{eq:D_mui}
        \begin{split} 
            \MoveEqLeft
            \int_{[0, 1]^d} \int_0^1 D_{g_i^i(R(v)u, w) - g_i^i(S(v), u, w)} \phi_k(\zeta_i(v)) dw \mathcal{Q}_i(du) \\
            & \leq \frac{2}{k |\zeta_i(v)|} \int_{[0, 1]^d} \int_0^1 (g_i^i(R(v)u, w) - g_i^i(S(v), u, w))^2 dw \mathcal{Q}_i(du) 
            \leq \frac{10}{k} %\frac{4}{k} 
            \int_{[0, 1]^d} u_i^2 \mathcal{Q}_i(du) ,
        \end{split}
    \end{equation}
    and the latter inequality is due to \eqref{aux_prov_1} with $i = j$. 
    %{\color{blue} Adjust \eqref{aux_prov_1} to the case $i=j$?}
    Then,
    \begin{align*}
        \phi_k(\zeta_i(t)) \leq {}
        & M_t^{\phi_k} + \phi_k(\zeta_i(0))
        + \int_0^t K_m^i \lVert \zeta(v) \rVert dv
        + \int_0^t \frac{\rho_{ii}}{k} dv
        %+ \int_0^t \frac{2 c_i}{k z_i} dv 
        % \\ 
        % & + \int_0^t \frac{8 z_i}{k} \int_{[0, 1)^d \setminus \{0\}} u_i^2 \mathbf{T}_z \mu_i(du) dv\\
        + \int_0^t %\frac{4}{k} 
        \frac{10}{k} \int_{[0, 1]^d} u_i^2 \mathcal{Q}_i(du) dv \\
        & %+ \sum_{j \in [d] \setminus \{i\}} {\color{blue}2?}\Bigl( \int_{[0, 1]^d} u_i \mathcal{Q}_j(du) \Bigr) \bigl( |\zeta_i(v)| + |\zeta_j(v)| \bigr) dv .
        + \sum_{j \in [d] \setminus \{i\}} \Bigl( \int_{[0, 1]^d} u_i \mathcal{Q}_j(du) \Bigr) \bigl( |\zeta_i(v)| + 2 |\zeta_j(v)| \bigr) dv .
        % + \sum_{j \in [d] \setminus \{i\}} \int_0^t
        % \biggl( 2 z_j \int_{[0, 1)^d \setminus \{0\}} u_i \mathbf{T}_z \mu_j(du) \biggr)
        % ( |\zeta_i(v)| + |\zeta_j(v)| ) dv.
    \end{align*}
    By taking expectations on both sides of the previous inequality and using the fact that $\zeta(t) \in[-1,1]^d$ for all $t \geq 0$, we obtain
    \begin{align*}
        \mathbb{E}\bigl[ \phi_k(\zeta_i(t)) \bigr] \leq {}
        & \frac{t}{k} \Bigl( \rho_{ii} + 10 %4 
        \int_{[0, 1]^d} u_i^2 \mathcal{Q}_i(du) \Bigr) + \mathbb{E}[\phi_k(\zeta_i(0))] + \int_0^t K_i \mathbb{E}[\| \zeta(v) \|] dv ,
        % \frac{t}{k} \biggl( \frac{2 c_i}{z_i} + 8 z_i \int_{[0, 1)^d \setminus \{0\}} u_i^2 \mathbf{T}_z \mu_i(du)  \biggr)
        % + \mathbb{E}[\phi_k(\zeta_i(0))] + \int_0^t K_i \mathbb{E}[\lVert \zeta(v) \rVert] dv,
    \end{align*}
    with
    \[
        K_i = K_m^i
        + 3%2 
        \sum_{j \in [d] \setminus \{i\}} \int_{[0, 1]^d} u_i \mathcal{Q}_j(du) .
        % + 4 \sum_{j \in [d] \setminus \{i\}} z_j \int_{[0, 1]^d} u_i \mathbf{T}_z \mu_j(du).
    \]
    Letting $k \to \infty$ we deduce, using that $\phi_k(z) \to \lvert z \rvert$ non-decreasingly
    and the Monotone Convergence Theorem,
    \[
        \mathbb{E}[|\zeta_i(t)|]
        \leq \mathbb{E}[|\zeta_i(0)|]
        + \int_0^t K_i \mathbb{E}[\lVert \zeta(v) \rVert] dv .
    \]
    As $i$ was arbitrary, this in turn implies the existence of a constant
    $K > 0$ such that
    \begin{equation} \label{eq:pre_gronwall}
        \mathbb{E}[\lVert \zeta(t) \rVert] \leq \sqrt{d} \mathbb{E}[\lVert \zeta(0) \rVert]
        + \int_0^t K \mathbb{E}[\lVert \zeta(v) \rVert] dv .
    \end{equation}
    If $R(0) = S(0)$, then $\zeta(0) = 0$, and the result follows from
    Gronwall's lemma and the right continuity of $t \mapsto \zeta(t)$.
\end{proof}

\subsection{Existence of a strong solution}
Having proved pathwise uniqueness for \eqref{eq:lambdaWF}, we now turn to the existence of a strong solution, which will then be unique. By the Yamada--Watanabe theorem, it is enough to establish the existence of a weak solution; see, for instance, \cite[p.~104]{situTheoryStochasticDifferential2005}. We prove this by constructing a process $R$ such that, for every $f \in \mathcal{C}^2([0,1]^d)$, the process 
\begin{equation} \label{eq:martingale_problem} M_f(t) := f(R(t)) - f(R(0)) - \int_0^t \mathcal{A}_{\mathrm{WF}} f(R(s))\,ds, \qquad t \geq 0, 
\end{equation} 
is a martingale, where $\mathcal{A}_{\mathrm{WF}}$ is given in \eqref{eq:lambdaWFGen}. In this way, we obtain a solution to the martingale problem associated with $\mathcal{A}_{\mathrm{WF}}$, and hence a weak solution to \eqref{eq:lambdaWF}. This also yields the existence part of Theorem \ref{theo:culling_limit}.

Before turning to the proof, note that \eqref{eq:lambdaWFGen} may be rewritten as 
\begin{equation} \label{eq:genWFMod}
    \begin{split}
        \mathcal{A}_{\mathrm{WF}} f(r) 
        & = \sum_{i = 1}^d \sum_{j \in [d] \setminus \{i\}} \Bigl(\rho_{ji} + \int_{[0, 1]^d} u_i \mathcal{Q}_j(du) \Bigr) (r_j - r_i) \partial_i f(r) 
        + \sum_{i = 1}^d \frac{\rho_{ii}}{2} r_i (1 - r_i) \partial_{ii} f(r) \\ 
        & \quad + \sum_{i = 1}^d \int_{[0, 1]^d} \int_0^1 \bigl( f(r + g^i(r, u, v)) - f(r) - \langle g^i(r, u, v), \nabla f(r) \rangle \bigr) dv \mathcal{Q}_i(du) ,
    \end{split}
\end{equation}
with $g^i = (g_1^i, \ldots, g_d^i)$ for each $i \in [d]$. 
Certainly, observe that for each $i \in [d]$ and $r, u \in [0, 1]^d$, 
\begin{align*}
    \MoveEqLeft 
    \int_0^1 \bigl( f(r + g^i(r, u, v)) - f(r) - \langle g^i(r, u, v), \nabla f(r) \rangle \bigr) dv \\
    & = r_i f(r + (1 - r) \odot u) + (1 - r_i) f(r - r \odot u) - f(r) 
    - r_i \langle (1 - r) \odot u, \nabla f(r) \rangle 
    + (1 - r_i) \langle r \odot u, \nabla f(r) \rangle \\
    & = r_i f(r + (1 - r) \odot u) + (1 - r_i) f(r - r \odot u) - f(r) 
    + \sum_{j = 1}^d u_j (r_j - r_i) \partial_j f(r) .
\end{align*}
Therefore, for each $r \in [0, 1]^d$, using that $\int_{[0, 1]^d} u_j \mathcal{Q}_i(du) < \infty$ for each $i \neq j$,
\begin{align*}
    \MoveEqLeft
    \sum_{i = 1}^d \int_{[0, 1]^d} \int_0^1 \bigl( f(r + g^i(r, u, v)) - f(r) - \langle g^i(r, u, v), \nabla f(r) \rangle \bigr) dv \mathcal{Q}_i(dv) \\
    & = \sum_{i = 1}^d \int_{[0, 1]^d} \bigl( r_i f(r + (1 - r) \odot u) + (1 - r_i) f(r - r \odot u) - f(r) \bigr) \mathcal{Q}_i(du) \\
    & \quad - \sum_{i = 1}^d \sum_{j \in [d] \setminus \{i\}} \Bigl( \int_{[0, 1]^d} u_j \mathcal{Q}_i(du) \Bigr) (r_i - r_j) \partial_j f(r) .
\end{align*}
This allows us to deduce the equivalence of \eqref{eq:lambdaWFGen} and \eqref{eq:genWFMod}. 

Let us now prove that there exists a strong solution to \eqref{eq:lambdaWF}.

\begin{proof}[Existence of a strong solution]
    Without loss of generality, we assume that $R(0)$ is deterministic.
    We start by stating that for any $i \in [d]$ and $x \in [0, 1]^d$
    we have the simple bound
    \begin{align}\MoveEqLeft
        |m_i(x)| + \sigma_i(x)^2
        + \sum_{j = 1}^d \int_{[0, 1]^d} \int_0^1 g_i^j(x, u, v)^2 dv \mathcal{Q}_j(du)
        % + \sum_{j = 1}^d \int_{[0, 1)^d \setminus \{0\}} \int_0^\infty (g_i^j(x, u, v)^2 + h_i^j(x, u, v)^2)
        % dv \mathbf{T}_z\mu_j(du)
        \notag\\
        & \leq
        \sum_{k \in [d] \setminus \{i\}} \Bigl( \rho_{ki} + \int_{[0, 1]^d} u_i \mathcal{Q}_k(du) \Bigr) 
        + \frac{\rho_{ii}}{4} % {\color{blue}\frac{\rho_{ii}}{2}?}
        + \sum_{j = 1}^d \int_{[0, 1]^d} u_i^2 \mathcal{Q}_j(du) .
        % \leq  \sum_{k \in [d] \setminus \{i\}} \biggl( b_{ik} \frac{z_k}{z_i}
        % + z_k \int_{[0, 1)^d \setminus \{0\}} u_i \mathbf{T}_z\mu_k(du) \biggr)
        % + \frac{c_i}{2 z_i}
        % + \sum_{j = 1}^d 2 z_j \int_{[0, 1)^d \setminus \{0\}} u_i^2 \mathbf{T}_z\mu_j(du).
        \label{eq:useful_bound}
    \end{align}
    By a straightforward generalisation of Proposition 4.1 in \cite{liStrongSolutionsJumptype2012},
    this bound gives us that a c\`adl\`ag process is a weak solution (hence a strong solution) to
    \eqref{eq:lambdaWF} if and only if \eqref{eq:martingale_problem} is a
    locally bounded martingale.

    For each $i \in [d]$, we now consider the sequence $(W_n)_{n \in \mathbb{N}}$ of non-decreasing sets defined
    by
    \[
        W_n = \{\lVert u \rVert > 1/n\} \subset [0, 1]^d
    \]
    Then, for each $i \in [d]$ and $n \in \mathbb{N}$,
    \[  
        \mathcal{Q}_i(W_n) \leq n^2 \int_{[0, 1]^d} \sum_{j = 1}^d u_j^2 \mathcal{Q}_i(du) < \infty.
    \]
    Moreover, using \eqref{eq:preLip_g} and the previous remark, we note that 
    \[
        x \mapsto \int_{W_n} \int_0^1 g_i^j(x, u, v) dv \mathcal{Q}_j(du) 
        \text{ is Lipschitz for each $i, j \in [d]$ and $n \in \mathbb{N}$.}
    \]
    % $(\mathbf{T}_z\mu_i \times \mathrm{Leb}) (V_n) < \infty$, where $\mathrm{Leb}$ is the Lebesgue measure,
    % for all $n \in \mathbb{N}$ and $i \in [d]$.
    % Moreover, for any $n \in \mathbb{N}$, $i, j \in [d]$, the mappings
    % \[
    %     x \in [0, 1]^d
    %     \mapsto \iint_{V_n} g_i^j(x, u, v) \mathbf{T}_z\mu_j(du) dv
    %     \quad\text{and}\quad
    %     x \in [0, 1]^d
    %     \mapsto \iint_{V_n} h_i^j(x, u, v) \mathbf{T}_z\mu_j(du) dv
    % \]
    % are seen to be Lipschitz, cf. \eqref{eq:preLip_g}, \eqref{eq:preLip_h} and \eqref{eq:gij_lipschitz}---the latter holds for $j = i$ when replacing $[0, 1]^d$ for $W_n$.
    Hence, by results on continuous-type SDEs,
    see Theorem 2.2 in Chapter 4 of \cite[p.~169]{ikedaStochasticDifferentialEquations1989},
    there is a weak solution, $R^{(n)}(t)=(R_1^{(n)}(t),\dots,R_d^{(n)}(t))$, to
    \begin{align} \label{eq:almost_sde_cont}
        R_i^{(n)}(t) = {}
        & R(0) + \int_0^t m_i(R^{(n)}(s)) ds + \int_0^t \sigma_i(R^{(n)}(s)) dB_i(s) \\ 
        & - \sum_{j = 1}^d \int_0^t \int_{W_n} \int_0^1 g_i^j(R^{(n)}(s), u, v) dv \mathcal{Q}_j(du) ds, \notag
        % & - \sum_{j = 1}^d \int_0^t \iint_{V_n}
        % (g_i^j(R^{(n)}(s), u, v) + h_i^j(R^{(n)}(s), u, v)) \mathbf{T}_z\mu_j(du) dv ds, \notag
    \end{align}
    with values in $[0, 1]^d$.
    Pathwise uniqueness holds for the SDE \eqref{eq:almost_sde_cont}
    due to Proposition \ref{prop:pathwise_uniqueness}, so it follows that
    there is a strong unique solution for \eqref{eq:almost_sde_cont} by the
    Yamada--Watanabe Theorem.

    Then, by an easy modification of Proposition 2.2 in \cite{fuStochasticEquationsNonnegative2010},
    there exists a unique strong solution $\{R^{(n)}(t) : t \geq 0\}$ to
    \begin{align} \label{eq:almost_sde}
        R_{i}^{(n)}(t) = {}
        & R(0) + \int_0^t m_i(R^{(n)}(s)) ds + \int_0^t \sigma_i(R^{( n)}(s)) dB_i(s)\notag\\
        & + \sum_{j = 1}^d \int_0^t \int_{W_n} \int_0^1 g_i^j(R^{(n)}(s-), u, v) \widetilde{\mathcal{N}}_j(ds, du, dv) .
        % & + \sum_{j = 1}^d
        % \int_0^t \iint_{V_n} g_i^j(R^{(n)}(s-), u, v) \widetilde{N}_1(ds, du, dv) \notag \\
        % & + \sum_{j = 1}^d
        % \int_0^t \iint_{V_n} h_i^j(R^{(n)}(s-), u, v) \widetilde{N}_2(ds, du, dv).%\notag \\
    \end{align}
    Indeed, due to the manner in which we chose $W_n$, the Poisson point processes that appear in
    the SDE \eqref{eq:almost_sde} have finite intensity.
    Consequently, the number of jumps in any compact interval of time is finite,
    allowing us to order the jumps of the Poisson processes by $0 < S_1 < S_2 < \cdots$. Further, we obtain that $S_k \to \infty$ as $k \to \infty$.
    %{\color{blue} Mention that $S_K\to\infty$ as $k\to\infty$?}.
    Doing this, we can now construct the solution in a pathwise manner.
    First consider a solution $\mathcal{R}^{(0)} = (\mathcal{R}^{(0)}_1, \ldots, \mathcal{R}^{(0)}_d)$ to \eqref{eq:almost_sde_cont}, and define
    $R(t) = \mathcal{R}^{(0)}(t)$ for $t \in [0, S_1)$.
    Define $r^{(1)}$ as the random vector whose $i$-th entry is given by
    \begin{align*}
        r_i^{(1)} = {}
        & \mathcal{R}_{i}^{(0)}(S_1-) 
        + \sum_{j = 1}^d \int_{\{S_1\}} \int_{W_n} \int_0^1 g_i^j(\mathcal{R}^{(0)}(s-), u, v) \widetilde{\mathcal{N}}_j(ds, du, dv) .
        % + \sum_{j = 1}^d \int_{\{S_1\}} \iint_{V_n} g_i^j(\mathcal{R}^{(0)}(s-), u, v) \widetilde{N}_1(ds, du, dv)\\
        % &+ \sum_{j = 1}^d \int_{\{S_1\}} \iint_{V_n} h_i^j(\mathcal{R}^{(0)}(s-), u, v) \widetilde{N}_2(ds, du, dv).% \\
    \end{align*}
    With the same Brownian motion used to construct $\mathcal{R}^{(0)}$ now
    consider a solution $\mathcal{R}^{(1)}$ to the SDE
    \begin{align*}
        R_{i}^{(n)}(t) = {}
        & r^{(1)} + \int_0^t m_i(R^{(n)}(s)) ds + \int_0^t \sigma_i(R^{(n)}(s)) dB^{(i)}(S_1 + s) \\
        & - \sum_{j = 1}^d \int_0^t \int_{W_n} \int_0^1 g_i^j(R^{(n)}(s), u, v) dv \mathcal{Q}_j(du) ds 
        % & - \sum_{j = 1}^d \int_0^t \iint_{V_n}
        % (g_i^j(R^{(n)}(s), u, v) + h_i^j(R^{(n)}(s), u, v)) \mathbf{T}_z\mu_j(du) dv ds
    \end{align*}
    and define $R(t) = \mathcal{R}^{(1)}(t - S_1)$ for $t \in [S_1, S_2)$.
    By proceeding recursively we obtain a strong solution to \eqref{eq:almost_sde}.

    Thus, by It\^o's formula, for any $f \in \mathcal{C}^2([0, 1]^d)$, the process
    \[
        f(R^{(n)}(t)) - f(R^{(n)}(0)) - \int_0^t \mathcal{A}_{\mathrm{WF}}^{(n)} f(R^{(n)}(s)) ds ,
        \quad t \geq 0,
    \]
    is a locally bounded martingale,
    where $\mathcal{A}_{\mathrm{WF}}^{(n)}$ is given by replacing the integrals over
    $[0, 1]^d$ with integrals over $W_n$ in \eqref{eq:genWFMod}.
    Observe that by \eqref{eq:useful_bound} and \eqref{eq:lambdaWF}, for any fixed $T > 0$, every discrete stopping time $\tau$ bounded above by $T$, and every $\delta > 0$, we have the bound 
    \[
        \mathbb{E}[ \| R^{(n)}(\tau + \delta) - R^{(n)}(\tau) \|^2 ]
        \leq \delta \sum_{i = 1}^d (d + 2) ( \delta K_{m_i}^2 + \tilde{K}_i ) ,
    \]
    where, for each $i \in [d]$,
    \[
        K_{m_i} := \sum_{j \in [d] \setminus \{i\}} \Bigl( \rho_{ji} + \int_{[0, 1]^d} u_i \mathcal{Q}_j(du) \Bigr)
        \quad\text{and}\quad 
        \tilde{K_i} = \frac{\rho_{ii}}{4} + \sum_{j = 1}^d \int_{[0, 1]^d} u_i^2 \mathcal{Q}_j(du) .
    \]
    In particular, the right-hand side tends to zero as $\delta\downarrow 0$, uniformly in $n$ and in the choice of $\tau$.
    Moreover, as $[0, 1]^d$ is compact and $\{R^{(n)}(t) : t \geq 0\}$ takes values in $[0, 1]^d$ for each $n \in \mathbb{N}$, we deduce that $\{R^{( n)}(t) : n \in \mathbb{N}\}$ is a tight sequence of random vectors for every fixed $t \geq 0$. 
    Therefore, Aldous' tightness criterion, see \cite[Theorem~1]{aldousStoppingTimes1978}, implies that the sequence $\{R^{(n)}:n\in\mathbb N\}$ is tight in $D(\mathbb R_+,[0,1]^d)$.

    % As $\{R^{(n)}(t) : t \geq 0\}$ takes values in $[0, 1]^d$ for each $n \in \mathbb{N}$,
    % for every, $t \geq 0$, $\{R^{( n)}(t) : n \in \mathbb{N}\}$ is a tight
    % sequence of random vectors.{\color{blue}It is not clear what the tightness of $\{R^{( n)}(t) : n \in \mathbb{N}\}$ for fixed $t$ is used for...additionally, isn't this already part of the Aldous criterion? } 
    % \textcolor{cyan}{[As you point out, this is part of the hypotheses used for Aldous criterion, the other one being the convergence of the increments to $0$]}
    % In addition, by \eqref{eq:useful_bound} we can use an analogous argument to that
    % given in the proof of Lemma 4.3 in \cite{fuStochasticEquationsNonnegative2010}
    % to show that Aldous' criterion holds, and hence
    % $\{R^{(n)}(t) : t \geq 0\}$ is actually tight in $D(\mathbb{R}_+, [0, 1]^d)$.

    The existence of a strong solution follows from
    $\mathcal{A}_{\mathrm{WF}}^{(n)} f(x_n) \to \mathcal{A}_{\mathrm{WF}} f(x)$ if $x_n \to x$
    as $n \to \infty$,
  which is a consequence of a straightforward extension of the proof of
    Lemma 4.2 in \cite{liStrongSolutionsJumptype2012}.
\end{proof}

We are left with proving that \eqref{eq:lambdaWFGen} is
actually the infinitesimal generator of the solution to
\eqref{eq:lambdaWF}, as well as the Feller property for this process.
% Throughout this subsection we assume that $z \in (0, \infty)^d$ is given and fixed.
Note that from \eqref{eq:pre_gronwall}, putting $\zeta = R^{(r)} - R^{(s)}$, Gronwall's lemma gives
\begin{equation} \label{eq:Lipschitz_R}
	\mathbb{E}\big[ \lVert R^{(r)}(t) - R^{(s)}(t) \rVert \big]
	\leq K(t) \lVert r - s \rVert.
\end{equation}
\begin{proposition}
    The strong solution $\mathsf{R}$ of \eqref{eq:lambdaWF} is Feller and its
    infinitesimal generator $\mathcal{A}_{\mathrm{WF}}$ is given by
    \eqref{eq:lambdaWFGen}.
\end{proposition}

\begin{proof}
    For $t \geq 0$ and $f \in \mathcal{C}([0, 1]^d)$, we define the operator $T_t$ by
    $T_t f(r) := \mathbb{E}_r[f(\mathsf{R}(t))]$.
    Then $\{T_t : t \geq 0\}$ is the semigroup associated to $\mathsf{R}$.
    By \eqref{eq:Lipschitz_R} and a Taylor expansion of
    order 1, for any $f \in \mathcal{C}^1([0, 1]^d)$, $r, s \in [0, 1]^d$,
    \[
        \left|T_t f(r) - T_t f(s)\right|
        \leq \lVert \nabla f \rVert_{\infty} \mathbb{E}[ \lVert \mathsf{R}^{(r)}(t) - \mathsf{R}^{(s)}(t) \rVert ]
        \leq K(t) \lVert \nabla f \rVert_\infty \lVert r - s \rVert,
    \]
    which gives us the continuity of $r \mapsto T_t f(r)$ for
    $f \in \mathcal{C}^1([0, 1]^d)$.
    By a density argument and dominated convergence, this implies the continuity of
    $r \mapsto T_t g(r)$ for $g \in \mathcal{C}([0, 1]^d)$;
    this is $T_t(\mathcal{C}([0, 1]^d)) \subset \mathcal{C}([0, 1]^d)$.

    Given $f \in \mathcal{C}^2([0, 1]^d)$, Itô's formula gives us the equality
    \[
        f(\mathsf{R}^{(r)}(t)) = f(r) + \int_0^t \mathcal{A}_{\mathrm{WF}} f(\mathsf{R}^{(r)}(s)) ds + M_f(t),
    \]
    with $M_f$ a local martingale.
    In fact, $M_f$ is a locally bounded martingale,
    which follows easily from the bound
    \begin{equation} \label{eq:sup_inf_gen}
        \sup_{r \in [0, 1]^d} |\mathcal{A}_{\mathrm{WF}} f(r)| \leq K_f ,
    \end{equation}
    for some $K_f > 0$, from where we get
    $|M_f(t)| \leq K_f t + 2 \|f\|_\infty < \infty$.
    The bound given in \eqref{eq:sup_inf_gen} is a consequence of
    $f \in \mathcal{C}^2([0, 1]^2)$ and the bounds that entail \eqref{eq:useful_bound}.
    Then, as
    \[
        \mathbb{E}_r[f(\mathsf{R}(t))] - f(r) =
        \mathbb{E}_r \biggl[ \int_0^t \mathcal{A}_{\mathrm{WF}} f(\mathsf{R}(s)) ds \biggr],
    \]
    we deduce both that
    \[
        \sup_{r \in [0, 1]^d} |T_t f(r) - f(r)| \leq K t \to 0
        \quad\text{as } t \to 0,
    \]
    and that, for any $r \in [0, 1]^d$,
    \[
        \lim_{t \to 0} \frac{T_t f(r) - f(r)}{t}
        = \mathcal{A}_{\mathrm{WF}} f(r) .
    \]
    The former conclusion shows that $\mathsf{R}$ is Feller,
    while the latter lets us conclude that \eqref{eq:lambdaWFGen}
    is the infinitesimal generator of $\mathsf{R}$ due to
    Theorem 1.33 in \cite[p.~24]{bottcherLevyMattersIII2013}.
\end{proof}

\subsection{Proof of Theorem \ref{theo:block_duality} }\label{proof_theo:block_duality}
To conclude the section we 
will prove the moment duality of the multitype $\Lambda$-Wright--Fisher process and the block counting process.

We will first work at the level of infinitesimal generators, applying \eqref{eq:gen_d_colonies} to polynomials of the form $r^n = \prod_{i = 1}^d r_i^{n_i}$ for $r \in [0, 1]^d$ and $n \in \mathbb{N}_0^d$. %, where {\color{red}$r^n := \prod_{i = 1}^d r_i^{n_i}$.} %$r^n$ is understood as in \eqref{eq:multiidex_power}.
In this direction we will use the identities 
\begin{align*}
    (r_j - r_i) n_i r^{n - e_i}
    & = n_i (r^{n - e_i + e_j} - r^n), \\
    r_i (1 - r_i) n_i (n_i - 1) r^{n - 2 e_i}
    & = 2 \binom{n_i}{2} (r^{n - e_i} - r^n),
\end{align*}
along the equalities
\begin{align*} 
    % \MoveEqLeft
    r_i (r + (1 - r) \odot u)^n + (1 - r_i) (r - r \odot u)^n - r^n
    & = \sum_{k \in [n]_0 \setminus \{0, e_i\}} \binom{n}{k} u^k (1 - u)^{n - k} (r^{n - k + e_i} - r^n) .
    % r_i \Bigl[ \bigl( r + (1 - r) \odot u \bigr)^n - r^n - (1 - r_i) u_i n_i r^{n - e_i} \Bigr]
    % + (1 - r_i) \Bigl[ \bigl( r - r \odot u \bigr)^n - r^n + r_i u_i n_i r^{n - e_i} \Bigr] \\
    % & = \sum_{k \in [n] \setminus \{0, e_i\}} \binom{n}{k} u^k (1 - u)^{n - k}
    % (r^{n - k + e_i} - r^n) .
\end{align*}
where in the last sum we denote $\binom{n}{k} = \prod_{i = 1}^d \binom{n_i}{k_i}$. %are using the multiindex notation as in \eqref{eq:multiidex_power} and \eqref{eq:multiindex_binom}.
Then, by considering the function
$H : [0, 1]^d \times \mathbb{N}_0^d \to \mathbb{R}_+$ defined
as $H(r, n) = r^n$,
the previous expressions entail
\begin{align} \label{eq:inf_gen_app_dual}
    \mathcal{A}^{(z)} H(r, n) 
    & = 
    \sum_{i = 1}^d %2 \frac{c_i}{z_i} 
    \binom{n_i}{2} \rho_{ii} (r^{n - e_i} - r^n)
    + \sum_{i = 1}^d \sum_{j \in [d] \setminus \{i\}} n_i %b_{ij} \frac{z_j}{z_i} 
    \rho_{ji} (r^{n - e_i + e_j} - r^n) \notag\\
    & + \sum_{i = 1}^d \sum_{k \in [n]_0 \setminus \{0, e_i\}} % z_i
    \prod_{j = 1}^d \binom{n_j}{k_j} \lambda_{n, k}^i (r^{n - k + e_i} - r^n),
\end{align}
where for $n \in \mathbb{N}_0^d$, $k \in [n]_0$ and $i \in [d]$
we define
\begin{align*}
    \lambda_{n, k}^i 
    & = \int_{[0, 1]^d} \prod_{j = 1}^d u_j^{k_j} (1 - u_j)^{n_j - k_j} \mathcal{Q}_i(du) ,
    %& = \int_{[0, 1)^d \setminus \{0\}} \prod_{j = 1}^d u_j^{k_j} (1 - u_j)^{n_j - k_j}
    %\mathbf{T}_z \mu_i(du) , %\\
\end{align*}
whenever the quantities are finite.
Expression \eqref{eq:inf_gen_app_dual} and a straightforward modification of the proof of Theorem 4 in \cite{caballeroRelativeFrequencyTwo2023} complete the proof of Theorem \ref{theo:block_duality}.

%%%%%%%%%%%%%%%%%%%%%%%%%%%%%%%%%%%%%%%%%%%%%%%%%%%%%%%%%%%%%%%%%%%
%%                                                               %%
%% Supplementary Material, if any, should be provided in         %%
%% {supplement} environment  with title and short description.   %%
%%                                                               %%
%%%%%%%%%%%%%%%%%%%%%%%%%%%%%%%%%%%%%%%%%%%%%%%%%%%%%%%%%%%%%%%%%%%

%\begin{supplement}
%\stitle{Title of Supplement A.}
%\sdescription{Short description of Supplement A.}
%\end{supplement}
%\begin{supplement}
%\stitle{Title of Supplement B.}
%\sdescription{Short description of Supplement B.}
%\end{supplement}

%%%%%%%%%%%%%%%%%%%%%%%%%%%%%%%%%%%%%%%%%%%%%%%%%%%%%%%%%%%%%%%%%%%
%%                                                               %%
%% Use the two commands below for producing your bibliography    %%
%% with bibtex, then comment again the commands and include the  %%
%% content of the .bbl file in this file below the commands.     %%
%%                                                               %%
%%%%%%%%%%%%%%%%%%%%%%%%%%%%%%%%%%%%%%%%%%%%%%%%%%%%%%%%%%%%%%%%%%%

%\nocite{*}
%\bibliographystyle{amsplain}
%\bibliography{refs.bib}

% add below the content of your .bbl file produced by bibtex.

\providecommand{\bysame}{\leavevmode\hbox to3em{\hrulefill}\thinspace}
\providecommand{\MR}{\relax\ifhmode\unskip\space\fi MR }
\providecommand{\MRhref}[2]{%
  \href{http://www.ams.org/mathscinet-getitem?mr=#1}{#2}
}
\providecommand{\href}[2]{#2}

\end{document}